\documentclass[a4paper]{amsart}

\usepackage[utf8]{inputenc}
\usepackage[T1]{fontenc}

\usepackage{fullpage}
\usepackage[notcite,notref,color]{showkeys} 
\definecolor{labelkey}{gray}{.85}

\usepackage{mathtools,amssymb,bm,bbm,thmtools,thm-restate}
\mathtoolsset{showonlyrefs}

\usepackage[shortlabels]{enumitem}

\def\itm#1{\rm ({#1})} 
\def\itmit#1{\itm{\it #1\,}} 

\def\abc{\itmit{\alph{*}}}

\usepackage{textcase}
\usepackage{xstring}

\usepackage[UKenglish]{babel}
\usepackage{csquotes}
\usepackage[style=trad-abbrv,sortcites=true,backend=biber,sortlocale=en_UK,useprefix=true,isbn=false,url=false,doi=false,eprint=true]{biblatex}
\usepackage{hyperref}
\usepackage{bookmark}

\def\cA{\mathcal{A}}
\def\cD{\mathcal{D}}
\def\Ee{\mathop{{}\mathbb{E}}}

\def\cG{\mathcal{G}}
\def\cH{\mathcal{H}}

\def\cP{\mathcal{P}}
\def\cQ{\mathcal{Q}}
\def\Rphi{R_\phi}

\newcommand*{\charef}[1]{\mathrm{\uppercase{\StrChar{#1}{1}}\ref{#1}}}
\newcommand*{\subref}[1]{_{\charef{#1}}}

\setlist{itemsep=3pt,parsep=0pt,topsep=2pt,partopsep=0pt}  
\setlist{leftmargin=*,itemindent=0pt} 
\setenumerate{label=\textup{(\roman*)},ref=\textup{(\roman*)}}

\newcommand{\By}[2]{\overset{\mbox{\tiny{#1}}}{#2}} 
\newcommand{\ByRef}[2]{   \By{\eqref{#1}}{#2} }

\newcommand{\gByRef}[1]{  \ByRef{#1}{>} }

\newcommand{\nnREALS}{\mathbb{R}_{\ge0}} 
\newcommand{\NATS}{\mathbb{N}}

\newcommand*{\oct}[2]{O^{#1}(#2)}

\renewcommand{\subset}{\subseteq}

\newcommand{\eps}{\varepsilon}

\newcommand{\ind}[2]{#1[#2]}
\newcommand{\indicator}[1]{\mathbbm{1}_{#1}}
\DeclarePairedDelimiter\abs{\lvert}{\rvert}
\DeclareMathOperator\dom{Dom}
\DeclareMathOperator\im{Im}
\DeclareMathOperator\vdeg{vdeg}
\newcommand{\Gam}{\Gamma}
\newcommand{\gam}{\gamma}

\newcommand{\tz}{^{(0)}}

\newcommand{\tk}{^{(k)}}

\newcommand{\tl}{^{(\ell)}}
\newcommand{\tlp}{^{(\ell')}}

\newcommand{\tw}{^{(\omega)}}

\newcommand{\Uconc}{U_{\mathrm{conc}}}
\newcommand{\Ureg}{U_{\mathrm{reg}}}

\newcommand{\tX}{\tilde{X}}
\newcommand{\tY}{\tilde{Y}}
\newcommand{\tW}{\tilde{W}}

\renewcommand{\vec}[1]{\mathbf{#1}}

\newcommand{\EMAIL}[1]{  \textit{E-mail}: \texttt{#1}}

\usepackage{xparse}
\DeclarePairedDelimiter\sqbracks{\lbrack}{\rbrack}
\DeclarePairedDelimiterX\sqbracksvert[2]{\lbrack}{\rbrack}{#1 \delimsize\vert #2}
\DeclareDocumentCommand \Ex {somo}
{	
	\mathbb{E}%
	\IfNoValueTF{#4}%
		{\IfBooleanTF{#1}%
			{\sqbracks*{#3}}%
			{\IfNoValueTF{#2}%
				{\sqbracks{#3}}%
				{\sqbracks[#2]{#3}}
			}%
		}%
		{\IfBooleanTF{#1}%
			{\sqbracksvert*{#3}{#4}}%
			{\IfNoValueTF{#2}%
				{\sqbracksvert{#3}{#4}}%
				{\sqbracksvert[#2]{#3}{#4}}%
			}%
		}%
}

\DeclarePairedDelimiterXPP\vnorm[2]{}{\lVert}{\rVert}{_{#2}}{#1}
\newcommand*{\boverline}[1]{{\overbracket[0.5pt][-0.5pt]{#1}}}
\newcommand*{\cC}{\mathcal{C}}
\newcommand*{\ocC}{\boverline{\mathcal{C}}}
\newcommand*{\tcC}{\widetilde{\mathcal{C}}}

\newtheorem{theorem}{Theorem}
\newtheorem{lemma}[theorem]{Lemma}
\newtheorem{claim}{Claim}
\newtheorem{proposition}[theorem]{Proposition}
\newtheorem{corollary}[theorem]{Corollary}

\newtheorem*{fact*}{Fact}

\theoremstyle{definition}
\newtheorem{definition}[theorem]{Definition}
\newtheorem*{definition*}{Definition}

\theoremstyle{remark}

\newcommand{\oldqed}{}
\def\endofFact{\scalebox{.6}{$\Box$}}
\newenvironment{claimproof}[1][Proof]{
  \renewcommand{\oldqed}{\qedsymbol}
  \renewcommand{\qedsymbol}{\endofFact}
  \begin{proof}[#1]
}{
  \end{proof}
  \renewcommand{\qedsymbol}{\oldqed}
} 

\title{Regularity inheritance in hypergraphs}

  \author[P. Allen]{Peter Allen*}
  \author[E. Davies]{Ewan Davies\dag}
  \author[J. Skokan]{Jozef Skokan*\ddag}

  \thanks{%
    *
    Department of Mathematics, London School of Economics, Houghton Street,
    London WC2A 2AE, UK
    \\\EMAIL{p.d.allen|j.skokan@lse.ac.uk}}
 \thanks{%
    \dag{}
	Simons Institute for the Theory of Computing, 121 Calvin Lab \#2190, UC Berkeley, Berkeley, CA 94720, USA
    \\\EMAIL{maths@ewandavies.org}}
 \thanks{%
\ddag{}
Department of Mathematics, University of Illinois at Urbana-Champaign, 1409 W. Green Street, Urbana, IL 61801, USA}
 \thanks{PA was partially supported by the EPSRC, grant number EP/P032125/1.}
 \thanks{ED was partially supported by the ERC, grant number 339109}
 \thanks{JS was partially supported by the National Science Foundation, grant number DMS-1500121.}
\date{\today}

\begin{document}
\begin{abstract}
 We give a new approach to handling hypergraph regularity. This approach allows for vertex-by-vertex embedding into regular partitions of hypergraphs, and generalises to regular partitions of sparse hypergraphs. We also prove a corresponding sparse hypergraph regularity lemma.
\end{abstract}

\maketitle


\section{Introduction}

The regularity method is a rich topic in extremal combinatorics which has some remarkable applications. 
Its roots lie in Szemerédi's proof that dense subsets of the natural numbers contain arbitrarily long arithmetic progressions~\cite{szemeredi1975sets}, and since then many more applications have been found. 
The method consists of a \emph{regularity lemma} which states that any large structure can be decomposed into pieces which have random-like behaviour, and a \emph{counting lemma} which states that a random-like piece has approximately the same number of small substructures as an analogous genuinely random piece has in expectation. 
This paper primarily concerns the counting lemma in the setting of sparse hypergraphs.

With a precise formulation of the aforementioned lemmas for graphs one can prove the \emph{triangle removal lemma}, which states that graphs on $n$ vertices that contain at most $o(n^3)$ triangles may be made triangle-free by removing at most $o(n^2)$ edges. 
From this result one can deduce Roth's theorem~\cite{roth1953}, that dense subsets of the natural numbers contain an arithmetic progression of length three.

There are at least three natural and highly fruitful directions in which to generalise the above results: to larger subgraphs than the triangle, to hypergraphs, and to sparse host graphs. 
A counting lemma gives sufficient pseudorandomness conditions for the existence of a (large number of) triangles in an $n$-vertex graph, and one could generalise this to larger subgraphs (in particular those whose size grows with $n$), such as a collection of $n/3$ vertex-disjoint triangles or the square of a Hamilton cycle. 
A key result along this line of thought is the blow-up lemma of Koml{\'o}s, S{\'a}rk{\"o}zy, and Szemer{\'e}di~\cite{KSSblow}. 
Another direction is to generalise the regularity and counting lemmas to hypergraphs, e.g. to prove a hypergraph removal lemma. 
As observed by Solymosi~\cite{solymosi2004}, a suitable hypergraph version of the above triangle removal lemma implies a multidimensional generalisation of Szemerédi's theorem.
A third direction relates to sparse host graphs. 
Going in this direction, one can prove a \emph{relative removal lemma} for hypergraphs which are a subgraph of a sparse, highly pseudorandom \emph{majorising hypergraph}, and with it prove the Green--Tao theorem~\cite{GTprimes} which states that dense subsets of the primes contain arbitrarily long arithmetic progressions.
This was recently done by Conlon, Fox, and Zhao~\cite{CFZrelative}, whose methods require weaker pseudorandomness properties of the primes than were originally required by Green and Tao.
Combinations of these generalisations have also been developed, such as the hypergraph blow-up lemma of Keevash~\cite{keevash2011hypergraph}, and the blow-up lemma for sparse graphs of Allen, Böttcher, Hàn, Kohayakawa and Person~\cite{ABHKPblow}. 

The purpose of this paper is to develop an embedding method which generalises the standard one for graphs in a combination of all three directions. 
That is, we develop a tool for counting in regular subgraphs of sparse, pseudorandom hypergraphs that can be used to embed bounded-degree hypergraphs whose size grows with the number of vertices of the host hypergraph. 
The main advantage of our approach is that it constructs an embedding vertex-by-vertex in a way that generalises a well-studied approach to the counting lemma (and related embedding results) in graphs. 

In order to have applications of our embedding method, we also state and prove a sparse hypergraph regularity lemma. This is not especially novel, but to the best of our knowledge it was not explicitly in the literature (though it was well known how to prove such a thing).

\subsection{Regularity in graphs, sparse graphs, and hypergraphs}

In graphs there are several different notions of regularity, and when the graph in question is \emph{dense}, i.e.\ has $n$ vertices and $\Omega(n^2)$ edges, it is easy to show that many of these notions are essentially equivalent~\cite{T87random,Tpseudo,CGWquasi}. 
The original definition of regularity by Szemerédi~\cite{Sreg} for a pair $(X, Y)$ of disjoint sets of vertices in a graph roughly states that for large enough subsets $X'\subset X$ and $Y'\subset Y$, the density of the bipartite subgraphs induced by $(X,Y)$ and by $(X',Y')$ are approximately equal. 
An application of the Cauchy--Schwarz inequality shows that this is equivalent to containing approximately the minimum possible number of four-cycles crossing $(X, Y)$ given this density. 
These definitions can be extended to sparse graphs and hypergraphs, but the equivalence between different forms is a much more complex subject in these settings. 
In this paper we use a definition of regularity that is a generalisation of four-cycle minimality known as \emph{octahedron minimality}.

One way of extending the large and influential body of work on dense, regular graphs to sparse graphs (those with $n$ vertices and $o(n^2)$ edges) is to consider a suitably well-behaved \emph{majorising graph} $\Gam$ and study subgraphs $G\subset\Gam$ which are in some way regular relative to $\Gam$. 
Given proper definitions of these concepts, one obtains the dense setting by taking $\Gam$ to be the complete graph. 
Though the behaviour of $\Gam$ and of $G$ relative to $\Gam$ are both types of pseudorandomness, from now on we use the term \emph{regularity} to refer to how edges of graphs (and hypergraphs) $G$ are distributed inside a majorising graph (or hypergraph) $\Gam$, and reserve the term \emph{pseudorandomness} for the behaviour of $\Gam$. 

For graphs, a pseudorandomness condition known as \emph{jumbledness}, which controls the number of edges between pairs of sets in the graph, is somewhat standard. Jumbledness is quite a strong condition, demanding control over edges between very small sets of vertices. For combinatorial applications, one can usually obtain such control. However, it has recently been observed that for many applications in number theory (in which one wants to work with a graph derived from a number theoretic object, such as the set of primes) jumbledness either fails to be true or at best requires the assumption of commonly believed but unproved conjectures. This motivates the use of a weaker notion of pseudorandomness in terms of small subgraph counts, also known as linear forms conditions, which follow from jumbledness and which one also can unconditionally obtain in the number theoretic applications. Briefly, this notion of pseudorandomness asserts that when we count small subgraphs in $\Gam$, we obtain the same answer as if $\Gam$ were truly random, up to a small constant relative error.

We should point out that qualitatively, one cannot further weaken the pseudorandomness assumption: having a counting lemma for $G\subset\Gam$ in particular implies that the number of small subgraphs in $\Gam$ is as one would expect in a truly random structure. Quantitatively the story is different: in order to count subgraphs of $G$ of a given size, we need to assume counting of rather larger subgraphs in $\Gam$ and with much smaller error terms. We do not believe our quantative bounds are optimal, and it would be interesting to improve them. We made no attempt to optimise our proof, preferring clarity: we do not see any reason to believe that the proof strategy can give optimal quantative bounds.

The first main contribution of this paper is to define another notion of pseudorandomness for hypergraphs (and more general objects) $\Gam$ called \emph{typically hereditary counting} (THC). 
We show that for a hypergraph $H$, counting conditions which depend on the maximum degree of $H$, but not on the number of vertices of $H$, imply THC strong enough for embedding $H$ into $\Gam$ and into regular subgraphs of $\Gam$. 
We also show that THC holds with high probability in random hypergraphs which are not too sparse. 

As mentioned above, there is a standard strategy which allows one to embed a (potentially large) bounded degree graph $H$ into a regular partition of a dense graph $G$. This strategy does not directly generalise to sparse graphs, or to hypergraphs of uniformity greater than $2$ whether dense or not. However a consequence of our results is that in all of these settings one obtains the THC property, and this does allow for a simple generalisation of the standard graph strategy, as we illustrate in Theorem~\ref{thm:THCexample}. This is particularly valuable because the THC property has only one error parameter $\eps$, which does not change during the embedding process. In contrast the use of hypergraph regularity produces at least two error parameters (with very different sizes) and the error parameters tend to proliferate during the embedding process.


Given a rough understanding of our regularity and pseudorandomness concepts, a useful model to bear in mind is that a graph $G$ composed of unions of regular pairs shares many properties with a random model where for each pair, edges between the sets appear independently at random with probability equal to the density of the pair. 
A counting lemma generally shows that the number of copies of some small subgraph $H$ in $G$ is close to the expected such number in this analogous random model, and we will informally refer to the `expected number' of copies of $H$ in regular graphs $G$ even though the quantity is entirely deterministic. 
When considering a graph $G$ which is a subgraph of a pseudorandom majorising graph $\Gam$, the analogous random model is to let each edge of $\Gam$ be present in $G$ independently at random with a probability related to the relative density of $G$ in $\Gam$.

We give a hypergraph counting lemma (Theorem~\ref{thm:counting}) in the setting sketched above which is slightly stronger than the standard dense hypergraph counting lemma, and is comparable to what one obtains for sparse graphs with the methods of Conlon, Fox, and Zhao~\cite{CFZrelative}. 
Importantly, our methods also allow us to prove a \emph{one-sided counting lemma} (or embedding result) for bounded-degree hypergraphs $H$, where the required pseudorandomness of $\Gam$ does not grow with the number of vertices of $H$, see Section~\ref{sec:countandembed}.

To prove these counting and embedding results our main technique is \emph{regularity inheritance}, which has recently been successfully applied to embedding problems in sparse graphs~\cite{GKRSinherit,CFZextremal,ABSSregularity,ABHKPblow}. 
The basic idea is that given a pseudorandom graph $\Gam$ and regular subgraph $G$, a typical vertex should have the property that its neighbourhood in $\Gam$ is similarly pseudorandom, and its neighbourhood in $G$ is 
a similarly regular subgraph. 
We give a form of this inheritance in hypergraphs where the pseudorandomness of $\Gam$ is controlled by small subgraph counts (which THC-graphs are able to satisfy), and for our octahedron-minimality version of regularity. 

An outline of how to prove counting and embedding results given suitable regularity inheritance lemmas is quite straightforward. 
One assumes pseudorandomness and regularity hypotheses that imply, via regularity inheritance, that similar conditions hold in the neighbourhood of a typical vertex and, by iterating this, construct an embedding vertex-by-vertex. 
To make this work in hypergraphs is technically difficult because usual forms of (strong) hypergraph regularity require one to consider \emph{regular complexes} where edges of size $k$ are regular with respect to edges of size $(k-1)$, but the error parameter in this regularity is much larger than the density of the $(k-1)$-edges.

The second main contribution of this paper is a formal definition of certain pseudorandomness and regularity conditions which we call a \emph{good partial embedding} (GPE) and several lemmas which prove that the above sketch can be rigorously applied to embed hypergraphs with this definition. It may not be clear from this sketch what the difference between GPE and THC is (and why we are not doing the same thing twice). We will discuss this in the concluding remarks.

\subsection{Related work}

Our methods can be considered as an extension to hypergraphs of certain ideas present in~\cite{ABSSregularity,ABHKPblow} for sparse graphs. 
Amongst other things, in these papers the authors develop regularity inheritance lemmas and apply them to control good partial embeddings in sparse graphs. 
The main idea is to develop techniques for working with graphs $G\subset\Gam$ where the error parameter $\eps$ in a regularity condition for $G$ relative to $\Gam$ is much larger than the overall density of $\Gam$. 
It is natural to draw on such techniques when working with hypergraphs because, even in the case of dense hypergraphs, one is forced to consider situations where regularity error parameters are larger than overall densities. 

It may be useful to compare our setting to that of Conlon, Fox, and Zhao~\cite{CFZrelative}, where the pseudorandomness of the majorising hypergraph $\Gam$ is measured by linear forms conditions,
which state that the number of certain subgraphs of $\Gam$ is close to their expectation,
and the regularity condition for subgraphs $G\subset\Gam$ is known as \emph{weak regularity}. 
In~\cite{CFZrelative} they prove that to count copies of some $H$ in a weak-regular $G\subset\Gam$ it suffices to be able to count subgraphs of at most $2v(H)$ vertices in $\Gam$. 
There are two important differences between our methods and those of Conlon, Fox, and Zhao.
Firstly, if one only needs lower bounds on the number of copies of $H$, our pseudorandomness condition on $\Gam$ follows from counting conditions that are given in terms of the maximum degree of $H$ and do not grow with $v(H)$. 
Secondly, we work with an octahedron-minimality regularity that is much stronger than their weak regularity. 
Though it is significantly more technical to work with, the extra strength it has can be desirable in applications. 

That our pseudorandomness conditions for embedding some $H$ do not grow with $v(H)$ means we can find embeddings of bounded-degree $H$ into suitable $G$ where $v(H)$ grows with $v(G)$.
This setting is well-studied in graphs and hypergraphs, leading notably to the aforementioned sparse blow-up lemma~\cite{ABHKPblow} for graphs, and the hypergraph blow-up lemma of Keevash~\cite{keevash2011hypergraph}.
Keevash's result allows one to embed $n$-vertex hypergraphs $H$ of bounded degree in suitably regular $n$-vertex $G$, but works only in the dense case (where $\Gam$ is the complete graph), and requires a regularity condition that results from the \emph{regular approximation lemma}~\cite{RSreg2} which has a rather different flavour than the octahedron-minimality definition which we use. 
Broadly speaking, our choice of regularity results in a setting that resembles sparse regular graphs, and our methods are accordingly inspired by that setting, while regular approximation allows one to work in a setting that resembles dense regular graphs and select techniques accordingly. 
There is a great deal of technical difficulty in making this precise, however. 

We expect that the methods developed in this paper will lead to a hypergraph blow-up lemma which avoids the technicalities of regular approximation and which applies also in the sparse setting. This is work in progress.

\subsection{Organisation}

The rest of this paper is outlined as follows. 
In the next section we state our main results after giving the necessary preliminary definitions. 
Section~\ref{sec:sketch} contains a detailed sketch of a counting lemma for dense graphs, embellished by comparisons to the significantly more general and technical setting of sparse hypergraphs. These comparisons motivate many of the rather technical concepts defined in this paper.
In Section~\ref{sec:Gacount} we deal entirely with the majorising hypergraph $\Gam$ and prove that our pseudorandomness property THC follows from certain counting conditions, and that it typically holds in suitable $\Gam$ obtained from a random hypergraph. 
In Section~\ref{sec:reg} we state and prove a sparse hypergraph regularity lemma, and explain why its output matches what we need for our methods here.
We then deduce counting and embedding results for sparse hypergraphs of a rather standard form from our more general results for good partial embeddings. 
In Section~\ref{sec:count} we prove the more general results for counting and embedding in good partial embeddings.
In Section~\ref{sec:CS} we state and prove a range of results related to the Cauchy--Schwarz inequality that we use throughout the paper, and in Section~\ref{sec:inherit} we apply these results to prove Lemma~\ref{lem:k-inherit}.
In the final section we give some concluding remarks and prove a sample application, Theorem~\ref{thm:THCexample}, of our methods.

\section{Main concepts and results}\label{sec:main}

Before we can state our main results we need some definitions, and here we give a word of warning that our usage of the term `hypergraph' differs slightly from what is standard. 
The usual definition of a \emph{hypergraph} consists of a vertex set $V$ and edge set $E$ containing subsets of $V$. 
Normally one is interested in \emph{$k$-uniform hypergraphs} which are hypergraphs with the additional condition that $E$ only contains subsets of $V$ of size exactly $k$.
When considering hypergraph regularity, one is often forced to consider \emph{$k$-complexes} which correspond to a union of $\ell$-uniform hypergraphs for $\ell\in[k]$ on the same vertex set with the additional property that the edge set $E$ of a complex is down-closed: if $f\in E$ and $e\subset f$ then $e\in E$.
We prefer to give alternative definitions to better separate the roles of complexes and hypergraphs in our methods.

\subsection{Complexes, weighted hypergraphs, and homomorphisms}

The main topics of this paper are counting and embedding in hypergraphs, and here we give precise definitions of these terms, and of the phrase `number of copies' that we used informally in the introduction.
We are primarily interested in finding homomorphisms (which we will define) from a complex $H$ (as above) to a complex $\cG$ that in applications is usually in some way `inspired by' a uniform hypergraph. 
For this reason we exclusively use \emph{complex} to refer to the object $H$ whose vertices form the domain of the homomorphism, and \emph{hypergraph} to refer to the `host graph' $\cG$ whose vertices form the image of the homomorphism. 
The above definition of complex is standard, and if for some $k\ge 1$, the complex $H$ contains no edges of size greater than $k$, we say $H$ is a $k$-complex (we do not insist that $H$ contains edges of size exactly $k$). 

Contrasting with usual `uniform' usage, our definition of hypergraph allows for edges of each size from $0$ upwards, and we actually allow \emph{weighted hypergraphs}, but we are not interested in weights on $H$, and so do not refer to weighted complexes.
It is convenient to avoid the assumption that our (weighted) hypergraphs $\cG$ are down-closed\footnote{It is not entirely obvious what the weighted generalisation of down-closed should be, which is one good reason for avoiding the notion.}. 
As we will see when we come to the definition of a homomorphism, an edge of $\cG$ whose subsets are not all contained in $\cG$ cannot play a role in any homomorphisms from $H$ to $\cG$, but it will nevertheless be convenient in the proof to allow such edges.
 
Given a vertex set $V$, a \emph{weighted hypergraph} is a function from the power set of $V$ to the non-negative reals. 
We think of a normal, unweighted hypergraph as being equivalent to its characteristic function, but remind the reader that the more usual setting for the hypergraph regularity method is to embed $k$-complexes into $k$-complexes, so this characteristic function may be nonzero on edges of sizes $0,1,\dotsc,k$. 
Though it may seem odd to care about the weight of the empty edge, and in applications it will often simply be $1$ at the start of the proof, it turns out to be useful during the proof.
The extra generality of weights turns out not to complicate our methods, and to rather simplify the notation. 
It is not essential to our approach; if one starts with unweighted hypergraphs, the functions appearing throughout will take only values $\{0,1\}$; that is, they are unweighted hypergraphs.
We use the letter $\Gam$ and calligraphic letters $\cG$, $\cH$ for weighted hypergraphs, and the corresponding lower case letters $\gam$, $g$, and $h$ for the weight functions. 

A \emph{homomorphism} $\phi$ from a complex $H$ to a weighted hypergraph $\cG$ is a map $\phi:V(H)\to V(\cG)$ such that $\abs[\big]{\phi(e)}=\abs{e}$ for each $e\in H$, and the \emph{weight} of $\phi$ is
\[
\cG(\phi) := \prod_{e\in H}g\big(\phi(e)\big)\,.
\]
Note that this product does run over $e=\emptyset$ and edges of size $1$ in $H$.
If $\cG$ is an unweighted hypergraph, then the weight of $\phi$ is either $0$ or $1$, taking the latter value if and only if $\phi(e)$ is an edge of $\cH$ (in the usual unweighted sense) for each $e\in F$, including edges of size one (vertices), and the empty edge (the reason for including $e=\emptyset$ becomes clear later). 
In other words, this is if and only if $\phi$ is a homomorphism according to the usual unweighted definition from $H$ to $\cG$. 
We will be interested in summing the weights of homomorphisms, which is thus equivalent for unweighted hypergraphs to counting homomorphisms by the usual definition. 
Slightly abusing terminology for the sake of avoiding unwieldy phrases, we will talk about `counting homomorphisms' or `the number of homomorphisms' when what we really mean is `the sum of weights of homomorphisms'.

Bearing in mind that our weighted hypergraphs are inspired by $k$-uniform hypergraphs, we wish to consider weighted hypergraphs which contain edges of size $0,1,\dotsc, k$, but not of size $k+1$. 
If the weight function is to generalise the indicator function for edges in the unweighted setting then we should say that $\cG$ is a $k$-graph to mean that $g(e)=0$ for any edge $e$ of size at least $k+1$. 
We prefer an alternative definition for convenience of notation. 
If one is interested in weights in $\cG$ of edges up to size $k$, one can ask for a homomorphism from a $k$-complex into $\cG$, which naturally excludes any edges of size at least $k+1$. 

It is more convenient for our purposes to say that $\cG$ is a \emph{$k$-graph} to mean that $g(e)=1$ for all edges $e$ of size at least $k+1$, so that such edges do not affect the weight of any homomorphisms into $\cG$. 
This affords a certain amount of flexibility in the homomorphism counting methods we develop. 
For example, let $H$ be a $(k+1)$-simplex (the down-closure of a single edge of size $k+1$), and $H'$ be obtained from $H$ by removing the edge of size $k+1$.
If $\cG$ is a $k$-graph then homomorphisms from $H$ and $H'$ to $\cG$ receive the same weight and we do not need to distinguish between them. 

We are usually not interested in counting general homomorphisms from $H$ to $\cG$; for simplicity we reduce to a \emph{partite setting} where we have identified special image sets in $V(\cG)$ for each vertex of $H$. 
More formally, for the partite setting we will have a complex $H$ on vertex set $X$, a $k$-graph $\cG$ on vertex set $V$, a partition of $X$ into disjoint sets $\{X_j\}_{j\in J}$ indexed by $J$, and a partition of $V$ into disjoint sets $\{V_j\}_{j\in J}$ indexed by $J$.
The sets $X_j$ and $V_j$ are called \emph{parts}. 
We say a set of vertices (e.g.\ in $X$) is \emph{crossing}, or \emph{partite} if it contains at most one vertex from each part. 
As a shorthand, we say that $H$, $\cG$ are $J$-partite to mean we have this setting; partitions of $V(H)$ and $V(\cG)$ indexed by $J$. 
If only the number of indices matters, we sometimes write e.g.~$k$-partite to mean $J$-partite for some set $J$ of size $k$.

Given this partite setting, a \emph{partite homomorphism} from $H$ to $\cG$ is a homomorphism from $H$ to $\cG$ that maps each $X_j$ into $V_j$. 
That is, given an index set $J$ and partitions of $X$ and $V$ indexed by $J$, we consider special homomorphisms from $X$ to $V$ that `respect' the partition. 
Given $x\in X_j$ we sometimes write $V_x$ for the part $V_j$ into which we intend to embed $x$; and for a crossing subset $e$ of $X$ we write $V_e=\prod_{x\in e}V_x$ for the collection of crossing $\abs{e}$-sets with vertices in $\bigcup_{x\in e}V_x$. 

Given the partite, weighted setup above we write $\cG(H)$ for the expected weight of a uniformly random partite homomorphism from $H$ to $\cG$, that is, the normalised sum over all partite homomorphisms $\phi$ from $H$ to $\cG$ of the weight of $\phi$,
\[
\cG(H) := \Ex[\Big]{\prod_{e\in H}g\big(\phi(e)\big)} = \Big(\prod_{j\in J}\abs{V_j}^{-\abs{X_j}}\Big)\sum_{\phi}\prod_{e\in H}g\big(\phi(e)\big)\,.
\]
If $\cG$ is constant on the sets $V_e$ for crossing $e\subset V(H)$, then we obtain $\cG(H)=\cG(\phi)$ for any partite homomorphism $\phi:H\to\cG$, and a counting lemma states that $\cG(H)$ is close to this `expected value' where the constants taken on the sets $V_e$ are (close to) the density of $\cG$ on the appropriate $V_e$.

In this paper we will primarily work with partite homomorphisms which map exactly one vertex of $H$ into each part of $\cG$. 
We reduce the general setting to this one-vertex-per-part setting by the following somewhat standard `copying process'.

\begin{definition}[Standard construction]\label{def:standard}
Given an index set $J$, a $k$-complex $H$ with vertex set $X$ partitioned into $\{X_j\}_{j\in J}$ and a $k$-graph $\cG$ with vertex set $V$ partitioned into $\{V_j\}_{j\in J}$, the \emph{standard construction} is as follows. 
Let $\cG'$ be an $X$-partite $k$-graph with vertex sets $\{V'_x\}_{x\in X}$ where for each $x\in X$, the set $V'_x$ is a copy of the set $V_j$ such that $x\in X_j$, and where for each set $f\subset V(H)$ and each edge $e\in V'_f$ we define
 \[g'(e):=\begin{cases} 1 & \text{if $f\not\in H$}\,,\\ g(e') & \text{if $f\in H$}\,,\end{cases}\]
 where $e'$ is the natural projection of $e$ to $V(\cG)$.
\end{definition}

This construction defines a new $k$-graph $\cG'$ (together with a partition of its vertices indexed by $X$) whose vertices are all copies of vertices in $\cG$, with weights given precisely so that $J$-partite homomorphism counts from $H$ to $\cG$ correspond to $X$-partite homomorphism counts from $H$ to $\cG'$. 
One is forced to consider $H$ as $J$-partite for the former counts, and $X$-partite (with parts of size $1$) for the latter. 
That is, for $f\not\in H$ the edges $V'_f$ all have weight one in $\cG'$, so for each $\phi:V(H)\to V(\cG')$ we have
\[\cG(\phi)=\prod_{f\in H}g\big(\phi(f)\big)=\prod_{f\subset V(H)}g'\big(\phi(f)\big)=\cG'(\phi)\,,\]
where we abuse notation by identifying $\phi$ with its natural projection onto $V(\cG)$.

\subsection{Density and link graphs}

The notation for weighted hypergraphs gives us a rather compact way of expressing the `expected number' of homomorphisms from $H$ to $\cG$ in the partite setting. 
If $H$ and $\cG$ are $J$-partite, and for $f\subset J$ we have constants $d(f)$ which represent the average weight $\cG$ gives to edges in $V_f$, we can reuse the notation for $k$-graphs to represent the product of densities that form the expected value of $\cG(\phi)$. 
More formally, we have a $k$-graph $\cD$ on vertex set $J$ whose weight function $d$ maps $f\mapsto d(f)$. 
If $\cG$ is constant on each $V_f$, then (trivially) we have $\cG(H)=\cD(H)$. 
More generally, if $\cG$ is not constant, but the edges are well-distributed (in a sense we will make precise later) and the density on each $V_f$ is about $d(f)$, we will say $\cD$ is \emph{a density graph} for $\cG$. Note that we do not insist that densities are given exactly by $\cD$ (we allow a small error which we will specify later) and hence $\cD$ is not given uniquely by $\cG$. 
This turns out to be convenient for notation. 
Our definition of weighted hypergraph, in which the empty set is given a weight, is not always convenient. 
We will see that we cannot necessarily keep control of the weight of the empty set, and as a result we have to scale explicitly by it in many formulae. 
We will usually have $\cG$ and $\cD$ as above, except that $g(\emptyset)\ne d(\emptyset)$ and so we take care to scale by these values. 
In the model situation where $\cG$ is constant on each $V_f$ we would have $\cG(H)/g(\emptyset)=\cD(H)/d(\emptyset)$. 
This rather formal interpretation of our notation does serve a purpose, we will use $g(\emptyset)$ for keeping track of the embedded weight in a partial embedding, for which we would otherwise have to invent further notation.
The present choice of notation also avoids frequently having to explicitly exclude $\emptyset$ as a subset of some index set throughout the argument.

Before stating our results we also require a definition of the \emph{link graph} of a vertex $v$ in a weighted hypergraph $\cG$, which corresponds to our notion of the neighbourhood of $v$ in $\cG$. 
Let $J$ be an index set, $i\in J$, and let $\cG$ be a hypergraph with vertex sets $\{V_j\}_{j\in J}$. 
For a vertex $v\in V_i$, let $\cG_{v}$ be the graph on $\{V_j\}_{j\in J\setminus\{i\}}$ with weight function $g_v$ defined as follows. 
For $f\subset J\setminus\{i\}$ and $e\in V_f$, we set
\[
g_v(e) := g(e)\cdot g(v,e)\,.
\]
Note that we write $g(v,e)$ for the more cumbersome $g\big(\{v\}\cup e\big)$ and we do allow $e=\emptyset$ in this definition. 
The point of this definition is that a partite homomorphism from a complex $H$ to $\cG$ which maps vertex $i\in V(H)$ to $v\in V(G)$ corresponds (in terms of weight) to a homomorphism from $H-i$ to $\cG_v$.

In the weighted setting we are required to replace the notion of the size of a set of vertices with the sum of the weights of the vertices, and for convenience we work with the following normalised version of this idea. 
We are primarily interested in showing that a large fraction of a part $V_j$ in some partite $k$-graph $\cG$ has some `good' property. 
Given a subset $U\subset V_j$ we write $\vnorm{U}{\cG}:=\Ex{\indicator{v\in U}g(v)}$, where the expectation is over a uniform choice of $v\in V_j$, so that $\vnorm{V_j}{\cG}$ is the average weight of a vertex in $V_j$. 
Now if $U$ is the set of vertices satisfying some property, a statement of the form $\vnorm{U}{\cG}\ge (1-\eps)\vnorm{V_j}{\cG}$ for a small $\eps>0$ is the weighted generalisation of `a large fraction of the vertices in $V_j$ have the property'.

\subsection{Pseudorandomness for the majorising hypergraph}

Our first main result is to define the pseudorandomness condition typically hereditary counting (Definition~\ref{def:THC}), and show that it follows from certain counting conditions (Theorem~\ref{thm:GaTHC}). 
The definition and theorem are important in their own right for the following reason. 
Given a bounded-degree complex $H$, Theorem~\ref{thm:GaTHC} gives sufficient counting conditions for a good lower bound on the number of homomorphisms from $H$ into $\Gam$. 
Since we do not have a matching upper bound, this result is known as a \emph{one-sided counting lemma}.
Note that we have not made any attempt to optimise the dependence on $H$ of these counting conditions, the key innovation is that for bounded-degree $H$ the size of the graphs appearing in the conditions is bounded (i.e.\ does not grow with $v(H)$). 


\begin{definition}[Typically hereditary counting (THC)]\label{def:THC}
 Given $k\ge 1$, a vertex set $J$ endowed with a linear order, and a density $k$-graph $\cP$ on $J$, we say the $J$-partite $k$-graph $\Gam$ is an \emph{$(\eta,c^*)$-THC graph} if the following two properties hold.
 \begin{enumerate}[label=\textup{(THC\arabic*)}]
  \item\label{thc:count} For each $J$-partite $k$-complex $R$ with at most $4$ vertices in each part and at most $c^*$ vertices in total, we have
   \[\Gam(R)=\big(1\pm v(R)\eta\big)\tfrac{\gam(\emptyset)}{p(\emptyset)}\cP(R)\,.\]
  \item\label{thc:hered} If $|J|\ge 2$ and $x$ is the first vertex of $J$, there is a set $V_x'\subset V_x$ with $\vnorm{V_x'}{\Gam}\ge(1-\eta)\vnorm{V_x}{\Gam}$ such that for each $v\in V_x'$ the graph $\Gam_v$ is an $(\eta,c^*)$-THC graph on $J\setminus\{x\}$ with density graph $\cP_x$.
 \end{enumerate}
\end{definition}

Roughly, THC means that we can count accurately copies of small complexes (and the count corresponds to the expected number) and that this property is typically hereditary in the sense that for most vertices $v$ we can count in the link $\Gam_v$, and we can count in typical links of $\Gam_v$, and so on. The important point separating this definition from simply `we can count all small subgraphs accurately' is that we may take links a large (depending on the number of vertices of $\Gam$) number of times.

It is immediate that when $\Gam$ is the complete $J$-partite $k$-graph (that is, it assigns weight $1$ to all $J$-partite edges) then for any $c^*$ it is a $(0,c^*)$-THC graph, with density graph $\cP$ being the complete $k$-graph on $J$ (and the ordering on $J$ is irrelevant). This is the setting we obtain (from the standard construction) when we are interested in embedding a $k$-complex $H$ on $J$ into a dense partite $k$-graph $\cG$, which we think of as a relatively dense subgraph of the complete $J$-partite $k$-graph on $V(\cG)$.

More importantly for this paper, the following result shows that counting conditions in $\Gam$ of a type found frequently in the literature suffice for $\Gam$ to be the majorising hypergraph in our upcoming counting and embedding results (stated in Section~\ref{sec:countandembed}). 

\begin{theorem}\label{thm:GaTHC}
For all $\Delta,\,k\ge 2$, $c^*\ge \Delta+2$, and $0<\eta'<1/2$, there exists $\eta_0>0$ such that whenever $0<\eta<\eta_0$ the following holds. 

Let $J$ be a finite set and $H$ be a $J$-partite $k$-complex on $J$ with $\Delta(H^{(2)}) \le \Delta$. 
Suppose that $\Gam$ is a $J$-partite $k$-graph in vertex sets $\{V_j\}_{j\in J}$ which is identically $1$ on any $V_e$ such that $e\notin H$, and $\cP$ is a density graph on $J$ such that for all $J$-partite $k$-complexes $F$ on at most $(\Delta+2)c^*$ vertices we have 
\[
\Gam(F) = (1\pm\eta)\tfrac{\gam(\emptyset)}{p(\emptyset)}\cP(F)\,.
\]
Then $\Gam$ is an $(\eta', c^*)$-THC graph. 
\end{theorem}

The combination of Definition~\ref{def:THC} and Theorem~\ref{thm:GaTHC} have two major implications for this paper. 
Firstly they give a linear-forms type condition for a hypergraph $\Gam$ to be sufficiently pseudorandom for the GPE techniques that allow us to prove embedding and counting lemmas (for regular $\cG\subset\Gam$) which are discussed in later subsections. 
Secondly, given such a counting lemma (or similar results from the literature) one can verify that the subgraph $\cG$ itself satisfies THC via Theorem~\ref{thm:GaTHC}, and obtain one-sided counting in $\cG$ directly from THC. 
In the final section of this paper we discuss the merits of these approaches and motivate the presence of both of them.

We also show that THC holds with high probability when $\Gam$ is a random hypergraph (Lemma~\ref{lem:randomTHC}), which allows for the methods of this paper to be applied in subgraphs of random hypergraphs. One way of 
achieving this would be to verify that the conditions of Theorem~\ref{thm:GaTHC} hold in suitable random hypergraphs, but we prefer to give a direct verification as it yields a better dependence of the probability on the structure of $\Gam$, and shows that the THC property can be tractable in a direct manner. 

We view a random $k$-uniform hypergraph on $n$ vertices as a $k$-graph that is complete (i.e.\ weight $1$) on edges of size at most $k-1$, and for which weights of $k$-edges are independent Bernoulli random variables (taking values in $\{0,1\}$) with probability $p$. 
Let $\Gam=G\tk(n,p)$ be this $k$-graph.
For a finite set $J$, let $H$ be a $J$-partite $k$-complex on a vertex set $X$, and let $V(\Gam)$ be partitioned into $\{V_j\}_{j\in J}$. Suppose that $X$ comes with a linear order.
By the standard construction we obtain vertex sets $\{V_x'\}_{x\in X}$, and an $X$-partite $k$-graph $\Gam'$. Note that, as observed after Definition~\ref{def:standard}, partite homomorphism counts in $\Gam$ and in $\Gam'$ are in correspondence. In particular, the counting property~\ref{thc:count} is equivalent to asking for the same bounds on homomorphism counts in $\Gam$. Furthermore, if we embed an initial segment of $X$ and update $\Gam'$ by taking links of all the embedded vertices, then~\ref{thc:count} in the link graph is the same as asking for a count of \emph{rooted homomorphisms} in $\Gam$. 
There is a slight subtlety here, namely that if we embed two vertices of $X$ to (automatically) different vertices of $\Gam'$ which correspond to one vertex of $\Gam$, then the complex we count in a link of $\Gam'$ and that which we count rooted in $\Gam$ are not quite the same.

We would like to know that in this setup $\Gam'$ is well-behaved enough to apply the main results of this paper, which we might expect to be true provided $p$ is not too small. 
We prove that for $c^*\in\NATS$ and $\eta>0$, provided $p$ and the parts $V_j$ are large enough, with high probability $\Gam'$ is an $(\eta,c^*)$-THC graph.

To state the requirements on $p$ formally we give a definition of degeneracy.
Suppose that $X$ is equipped with a fixed ordering, and let
\[
\deg_k(H):=\max_{e\in H}\abs[\big]{\{f\in H\tk: e\subset f,\,f\setminus e \text{ precedes } e\}}\,,
\]
where $f\setminus e$ precedes $e$ if and only if each vertex of $f\setminus e$ comes before every vertex of $e$ in the order on $X$. 
Then, when embedding vertices in order, part-way through the process an edge $e$ can be the set of unembedded vertices for at most $\deg_k(H)$ edges of size $k$ in $F$.
We make no attempt to optimise the dependence of $p$ on the relevant parameters.

\begin{lemma}\label{lem:randomTHC}
Let $\eta>0$ be a real number, $c^*,\,\Delta,\, d,\,k\in\NATS$, and $J$ be a finite set.  
Suppose that $H$ is a $J$-partite $k$-complex of maximum degree $\Delta$ and degeneracy $\deg_k(H)\le d$ with vertex set $X$ equipped with some fixed ordering.
For some fixed $0<\eps<1$, let $\Gam=G\tk(n,p)$ be a random $k$-graph where $\min\big\{p^{4^kc^*d},\,p^{4^k\Delta+d}\big\}\ge (2\log n) n^{\eps-1}$. 
Suppose also that $(1-\eta)^\Delta \ge 1/2$ and $\abs{X}\le n$. 
Then with probability at least $1-o(1)$ the following holds.

For any partition $\{V_j\}_{j\in J}$ of $V(\Gam)$ into parts of size at least $n_0=n/\log n$, 
writing $\Gam'$ for the $X$-partite graph obtained by the standard construction to $H$, $\Gam$, and $\{V_j\}_{j\in J}$, we have that $\Gam'$ is an $(\eta, c^*)$-THC graph with density graph $\cQ$ that gives weight $p$ to edges of $H\tk$ and weight $1$ elsewhere.
%
%
\end{lemma}

\subsection{Regularity} 

As mentioned in the introduction, a dense bipartite graph is regular if and only if the number of copies of $C_4$ it contains is close to minimal for that density. To generalise this to hypergraphs, we need to define the \emph{octahedron graph}. We will need several related graphs later, so we give the general definition.

Given a vector $\vec{a}$ with $k$ nonnegative integer entries, we define $\oct{k}{\vec{a}}$ to be the $k$-partite complex whose $j$th part has $\vec{a}_j$ vertices, and which contains all crossing $i$-edges for each $1\le i\le k$. 
Let $\vec{1}^k$ and $\vec{2}^k$ denote the $k$-vectors all of whose entries are respectively $1$ and $2$. Then $\oct{k}{\vec{1}^k}$ is the complex generated by down-closure of a single $k$-uniform edge, while $\oct{k}{\vec{2}^k}$ is `the octahedron'. Note that $\oct{2}{\vec{2}^2}$ is the down-closure of the $2$-graph $C_4$.
Later we will also require notation for two copies of $\oct{k}{1,\vec{a}}$ which share the same first vertex but are otherwise disjoint, for which we write\footnote{The `$+2$' in $+2\oct{k}{\vec{a}}$ is supposed to represent adding a common `tail' to two disjoint copies of $\oct{k}{\vec{a}}$.} $+2\oct{k}{\vec{a}}$. 

We are now in a position to define regularity for hypergraphs. Even when we are working in the `dense case', that is, we are thinking of $\cG$ as a relatively dense subgraph of the complete hypergraph (as opposed to some much sparser `majorising hypergraph'), we will often need to introduce a graph $\Gam$ which is not complete and of which $\cG$ is a relatively dense subgraph. The reader should always think of $\Gam$ as being a hypergraph whose good behaviour we have already established (and we are trying to show that $\cG$ is also well behaved).

\begin{definition}[Regularity of hypergraphs]\label{def:reg}
Given $k\ge 1$ and nonnegative real numbers $\eps$, $d$, let $\cG$ and $\Gam$ be $k$-partite hypergraphs on the same vertex parts. Suppose that for each $e$ with $|e|<k$ we have $g(e)=\gam(e)$, and suppose that for each $e$ with $|e|=k$ we have $g(e)\le\gam(e)$. Then we say that $\cG$ is \emph{$(\eps,d)$-regular} with respect to $\Gam$ if
\begin{align}
\cG\big(\oct{k}{\vec1^k}\big)&=(d\pm\eps)\Gam\big(\oct{k}{\vec{1}^k}\big) &\text{and}&& \cG\big(\oct{k}{\vec2^k}\big)&\le\big(d^{2^k}+\eps\big)\Gam\big(\oct{k}{\vec2^k}\big)\,.
\end{align}
We say that $\cG$ is \emph{$\eps$-regular} with respect to $\Gam$ to mean that the corresponding $(\eps, d)$-regularity statement holds with $d= \cG(\oct{k}{\vec{1}^k})/\Gam(\oct{k}{\vec{1}^k})$.
\end{definition}

Note that in this definition we do not specify the octahedron density of $\cG$ but only give an upper bound. The definition is only useful for graphs $\Gam$ such that a matching lower bound holds for all $\cG$, which we will see (Corollary~\ref{cor:relCSoctlow}) is the case when $\Gam$ is sufficiently well behaved.

Regularity for $k$-graphs is not usually discussed for $k=1$, but we use the notion as a shorthand for relative density in this paper. 
The definition makes sense when $k=1$, but only the first part of the assertion, that $\cG$ has density close to $d$ with respect to $\Gam$, is important. For any $1$-graph $\cG$ on a vertex set $V$, we have 
\[
\cG\big(\oct{1}{\vec2^1}\big) = \Ex{g(u)g(v)}[u,v\in V] = \Ex{g(v)}[v\in V]^2 = \cG\big(\oct{1}{\vec1^1}\big)^2\,,
\] 
and so imposing the upper bound on octahedron count is superfluous, as essentially the same upper bound (the change in $\eps$ being unimportant) follows from the density.

\subsection{Counting and embedding results}\label{sec:countandembed}

We can now state counting and embedding lemmas of a rather standard type that follow from our methods.
In the dense case, that is when $\Gam$ is a complete $J$-partite $k$-graph (for which THC is trivial), our counting lemma (Theorem~\ref{thm:counting}) is more or less the same as that given in~\cite{NRScount}. 
The notion of regularity used there is that of the regularity lemma in~\cite{RSreg}, which is slightly stronger than the octahedron minimality we use (see \cite{DHNRcharacterising} for the 3-uniform case). 
The embedding lemma, Theorem~\ref{thm:embedding}, is (as far as we know) not found in this form in the literature, but it does follow fairly easily from~\cite{NRScount}.
A related but rather harder statement is found in~\cite{CFKO}. 
In the sparse case, our Theorem~\ref{thm:counting} essentially follows from the results of~\cite{NRScount} and of~\cite{CFZrelative}, though again it is not explicitly stated. 
We would like to stress that the main novelty here is that our proofs proceed via a vertex-by-vertex embedding. 
As is standard in this context, we write e.g.\ $0< \eta_0 \ll d_1,\dotsc,d_k$ to mean that there is an increasing function $f$ such that the argument is valid for $0 < \eta_0 \le f(d_1,\dotsc,d_k)$.

Finally, we will say $\Gam$, with density graph $\cP$, is a $(\eta_0,c^*)$-THC graph for $H$ (where both $\Gam$ and $H$ are $J$-partite) if applying the standard construction to $\Gam$ and its density graph $\cP$ yields a $(\eta_0,c^*)$-THC graph.

\begin{theorem}[Counting lemma for sparse hypergraphs]\label{thm:counting}
For all $k\ge 2$, finite sets $J$, and $J$-partite $k$-complexes $H$, given parameters $\eta_k$, $\eta_0$ and $\eps_\ell$, $d_\ell$ for $1\le\ell\le k$ such that  $0<\eta_0\ll d_1,\dotsc,d_k,\,\eta_k$, and for all $\ell$ we have $0<\eps_{\ell}\ll d_{\ell},\dotsc,d_k,\,\eta_k$,
the following holds.

Let $c^*=\max\{2v(H)-1, 4k^2+k\}$. 
Given any $J$-partite weighted $k$-graphs $\cG\subset\Gam$ and density graphs $\cD$, $\cP$, where $\Gam$ is an $\big(\eta_0,c^*\big)$-THC graph for $H$ with density graph $\cP$, and where for each $e\subset J$ of size $1\le\ell\le k$, the graph $\ind{\cG}{V_e}$ is $\eps_\ell$-regular with relative density $d(e)\ge d_\ell$ with respect to the graph obtained from $\ind{\cG}{V_e}$ by replacing layer $\ell$ with $\Gam$, we have
\[
\cG(H) = \big(1\pm v(H)\eta_k\big)\tfrac{g(\emptyset)}{d(\emptyset)p(\emptyset)}\cD(H)\cP(H)\,. 
\]
\end{theorem}

\begin{theorem}[Embedding lemma for sparse hypergraphs]\label{thm:embedding}
For all $k\ge2$ and $\Delta\ge1$, given parameters $\eta_k$, $\eta_0$, and $\eps_\ell$, $d_\ell$ for $1\le\ell\le k$ such that $0<\eta_0\ll d_1,\dotsc,d_k,\,\eta_k,\,\Delta$, and for all $\ell$ we have $0<\eps_{\ell}\ll d_{\ell},\dotsc,d_k,\,\eta_k,\,\Delta$, the following holds.

Let $\cG\subset\Gam$ be $J$-partite weighted $k$-graphs with associated density graphs $\cD$, $\cP$, where $\Gam$ is an $(\eta_0,4k^2+k)$-THC graph for $H$ with density graph $\cP$, and where for each $e\subset J$ of size $1\le\ell\le k$, the graph $\ind{\cG}{V_e}$ is $\eps_\ell$-regular with relative density $d(e)\ge d_\ell$ with respect to the graph obtained from $\ind{G}{V_e}$ by replacing layer $\ell$ with $\Gam$.
Then we have
\[
\cG(H) \ge (1-\eta_k)^{v(H)}\tfrac{g(\emptyset)}{d(\emptyset)p(\emptyset)}\cD(H)\cP(H)\,. 
\]
for all $J$-partite $k$-complexes $H$ of maximum degree $\Delta$.
\end{theorem}

\subsection{Regularity inheritance}

As stated in the introduction, we prove counting and embedding results via regularity inheritance. 
For sparse graphs, a regularity inheritance lemma states that, given vertex sets $X$, $Y$, and $Z$ such that on each pair we have a regular subgraph of a sufficiently well-behaved majorising graph, neighbourhoods of vertices $z\in Z$ on one or two sides of the pair $(X,Y)$ typically induce another regular subgraph of the majorising graph.
The cases `one side' and `two sides' (see~\cite{CFZextremal,ABSSregularity}) are usually stated as separate lemmas, and the quantitative requirement for `well-behaved' are a little different. 
In this paper, we will not try to optimise this quantitative requirement and so state one lemma which covers all cases.

In addition to a regularity inheritance lemma one usually needs to make use of the (trivial) observation that given a regular pair $(X,Y)$ in a graph $G$, if $Y'$ is a subset of $Y$ which is not too small then most vertices in $X$ have about the expected neighbourhood in $Y'$ (see Section~\ref{sec:sketch}). Another way of phrasing this is to define a partite weighted graph $\cG$ on $X\cup Y$, with weights on the crossing $2$-edges corresponding to edges of $G$ and weights on the vertices of $Y$ being the characteristic function of $Y'$; then for most $v\in X$ the link $1$-graph $\cG_v$ has about the expected density (recall that regularity is trivial for $1$-graphs). We will need a generalisation of this observation to graphs of higher uniformity, where we will need not only that the link graph typically has the right density but also that it is typically regular. It is convenient to state this too as part of our general regularity inheritance lemma.

Informally, the idea is the following. If $\cG\subset\Gam$ are $\{0,\dots,k\}$-partite weighted graphs, which are equal on all edges except those in $V_{[k]}$ and $V_{\{0,\dots,k\}}$, and we have that $\cG[V_{[k]}]$ and $\cG[V_{\{0,\dots,k\}}]$ are respectively $(\eps,d)$-regular and $(\eps,d')$-regular with respect to $\Gam$, and $\Gam$ is sufficiently well-behaved, then for most $v\in V_0$ the graph $\cG_v$ is $(\eps',dd')$-regular with respect to $\Gam_v$, where $\eps'$ is not too much larger than $\eps$.

Following the general setting of the results outlined so far, our notion of `well-behaved' for $\Gam$ is given in terms of small subgraph counts, exactly as one might have in typical links of a THC-graph.
In the statement of the lemma we use notation $\cH\tl$ to mean the $k$-graph which gives the same weight as $\cH$ to edges of size $\ell$ but weight $1$ to all other crossing edges, and $\cH\cdot\cH'$ to mean the $k$-graph whose weight function is the pointwise product $h\cdot h'$. 
Recall that $+2\oct{k}{\vec a}$ represents two copies of $\oct{k}{1,\vec{a}}$ which share the first vertex, but are otherwise disjoint.

\begin{lemma}\label{lem:k-inherit}
For all $k\ge 1$ and $\eps',\, d_0 > 0$, provided $\eps,\, \eta>0$ are small enough that 
\[
\min\{\eps', 2^{-k}\} \ge 2^{2^{k+6}}k^3\big(\eps^{1/16}+\eta^{1/32}\big)d_0^{-2^{k+1}}\,,
\]
the following holds for all $\cP$ and all $d,\,d'\ge d_0$.

Let $\{V_j\}_{0\le j\le k}$ be vertex sets, and $\cP$ be a density graph on $\{0,\dotsc,k\}$.
Let $\cG\le\Gam$ be $(k+1)$-partite $(k+1)$-graphs on $V_0,\dotsc,V_k$ with such that 
\begin{enumerate}[label=\textup{(INH\arabic*)}]
\item\label{inh:count} for all complexes $R$ of the form $+2\oct{k}{\vec a}$ or $\oct{k+1}{\vec b}$, where $\vec a\in\{0,1,2\}^k$ and $\vec{b}\in\{0,1,2\}^{k+1}$, we have 
\[ \Gam(R) = (1\pm \eta)\tfrac{\gam(\emptyset)}{p(\emptyset)}\cP(R)\,, \]
\item\label{inh:GGam} $\cG$ gives the same weight as $\Gam$ to every edge except those of size $k+1$ and those in $V_{[k]}$,
\item\label{inh:regJ} $\cG^{(k+1)}\cdot\Gam^{(\le k)}$ is $(\eps,  d')$-regular with respect to $\Gam$,
\item\label{inh:regf} $\ind{\cG}{V_1,\dotsc,V_k}$ is $(\eps, d)$-regular with respect to $\ind{\Gam}{V_1,\dotsc,V_k}$. 
\end{enumerate}
Then there exists a set $V_0'\subset V_0$ with $\vnorm{V_0'}{\Gam}\ge (1-\eps')\vnorm{V_0}{\Gam}$ such that for every $v\in V_0'$ the graph $\cG_v$ is $(\eps', dd')$-regular with respect to $\Gam_v$. 
\end{lemma}

This is the promised regularity inheritance lemma. 
The quantification of the constants is crucial for the definition of a good partial embedding in the following Section~\ref{subsec:gpe}; in order for a useful counting lemma to follow from our approach one needs to be able to control the regularity error parameters at every step of a vertex-by-vertex embedding, and work with underlying densities much smaller than these errors. Observe that in the statement above, the quantities $d$ and $d'$ are relative densities of parts of $\cG$ with respect to $\Gam$; they need to be large compared to $\eps'$ (and $\eps$) in order for the statement to be interesting, and $\eta$ also needs to be small compared to $\eps'$. But the densities $p(e)$ from $\cP$, which by \ref{inh:count} are approximately the absolute densities in $\Gam$, can be (and in applications usually will be) very small compared to all other quantities. In typical applications $d$, $d'$, $\eps$, $\eps'$, $\eta$ will be constants fixed in a proof and independent of $v(\cG)$, while the $p(e)$ may well tend to zero as $v(\cG)$ grows.

\subsection{Good partial embeddings and counting}\label{subsec:gpe}

When using regularity inheritance to prove counting and embedding results, it is natural to describe one step of the embedding, isolate some common structure from each step, and prove the full result by induction. 
The `common structure' that we define is that of a \emph{good partial embedding} (GPE). 
Theorems~\ref{thm:counting}, and~\ref{thm:embedding} follow from more general results for GPEs.
Here we give motivation and definitions of the necessary ideas, and state these results for GPEs. 
A comparison of our methods to the simplest example of this approach in graphs is given in Section~\ref{sec:sketch}.

We construct homomorphisms from $H$ to $\cG$ vertex-by-vertex and count the contribution to the total number of homomorphisms from each choice of image as the construction proceeds. 
Given a complex $H$ and a $V(H)$-partite setup where we want to find a partite homomorphism from $H$ to a weighted graph $\cG$, we start with a trivial partial embedding $\phi_0$ from $H_0:=H$ to $\cG_0:=\cG$ in which no vertices are embedded. 
Now for each $t=1,\dots,v(H)$ in succession, we choose a vertex $x_t$ of $H_{t-1}$ and a vertex $v_t$ of $V_{x_t}$. 
We set $\phi_t:=\phi_{t-1}\cup\{x_t\to v_t\}$, and $H_t:=H_{t-1}\setminus\{x_t\}$, and we write $\cG_t:=(\cG_{t-1})_{v_t}$, that is we take the link graph. 
The graph $\cG_{v(H)}$ is an empty weighted graph with weight function $g_{v(H)}$: the only edge it contains is the empty set, and its weight is
\[g_{v(H)}(\emptyset)=\prod_{e\subseteq V(H)}g\big(\phi(e)\big)=\cG(\phi)\,.\]
Obviously, in general the final value $\cG(\phi)$ depends on the choices of the $v_t$ made along the way, but in the model case when for each $f\subset V(H)$ the function $g$ is constant, say equal to $d(f)$, on $V_f$ we obtain the same answer which ever choices we make. 
Furthermore, (trivially) at each step $t$, when we are to choose $v_t$ the average weight in $\cG_{t-1}$ of vertices in $V_{x_t}$ depends only on the values $d(f)$ and not on the choices made; and a similar statement is true for the edges in each $V_f$. 

When the (strong) hypergraph regularity lemma is applied to a $k$-uniform subgraph of $\Gam$, one ends up working with a subgraph $\cG$ of $\Gam$ which has the following properties. First, there is a vertex partition $\{V_j\}_{j\in J}$ indexed by $J$ of $V(\cG)=V(\Gam)$. Second, for each $f\subset J$ with $2\le|f|\le k$, the graph $\cG[V_f]$ is $\big(\eps_{|f|},d(f)\big)$-regular with respect to the graph whose weight function is equal to that of $\cG$ on edges of size at most $|f|-1$ and to $\Gam$ on edges of size $|f|$. Here one should think of the edges of $\cG$ of size $k-1$ and less as being output by the regularity lemma, and the $k$-edges as being the subgraph of $\Gam$ which we are regularising. The difficulty is that, while we always have $\eps_{|f|}\ll d(f)$, and indeed $\eps_\ell\ll d(f)$ for any $f$ with $|f|\ge\ell$, it may be the case that $\eps_\ell$ is large compared to the $d(f)$ with $|f|<\ell$. 

The solution to this is to separate counting and embedding into several steps. To begin with, we can count any small hypergraph to high precision in the ambient $\Gamma$ by assumption. We define a hypergraph whose edges are given weight equal to $\Gamma$ on edges of size $3$ and above, but equal to $\cG$ on edges of size two (and one). We can think of this hypergraph as being very regular and dense relative to $\Gam$: the relative density parameters are $d(e)$ for $|e|=2$ which are much larger than the regularity parameter $\eps_2$. Using our regularity inheritance lemma, we show that we can count any small hypergraph to high precision in this new hypergraph. This means we can now think of our new hypergraph as a well-behaved ambient hypergraph, and consider the hypergraph whose edges have weight equal to $\Gamma$ on edges of size $4$ and above, but equal to $\cG$ on edges of size $3$ and below. The same argument shows we can count small hypergraphs to high precision in this hypergraph too, and so on. Our approach thus keeps track of a \emph{stack} of hypergraphs, where we assume that we can count in the bottom \emph{level} $\Gamma$ and inductively bootstrap our way to counting in the top level $\cG$ by using the fact that each level is relatively dense and very regular with respect to the level below.

In general, we may have a more complicated setup because we have embedded some vertices. We begin by definining abstractly the structure we consider, and will then move on to giving the conditions it must satisfy in order that we can work with it. It is convenient to introduce a complex $H$ and a partial embedding of that complex in order to define the update rule; we do not need to specify the graph into which $H$ is partially embedded. 

\begin{definition}[Stack of candidate graphs, update rule]
 Let $k\ge 2$, and suppose that a $k$-complex $H$, a partial embedding $\phi$ of $H$, and disjoint vertex sets $V_x$ for each $x\in V(F)$ which is unembedded (that is, $x\not\in\dom\phi$) are given. Suppose that for each $0\le\ell\le k$ and each $e\subset V(H)\setminus\dom\phi$ we are given a subgraph $\cC^{(\ell)}(e)$ of $V_e$. We write $\cC^{(\ell)}$ for the union of the $\cC^{(\ell)}(e)$; that is, the graph with parts $\{V_x\}_{x\in V(H)\setminus\dom\phi}$ whose weight function is equal to that of $\cC^{(\ell)}(e)$ on $V_e$ for each $e\subset V(H)\setminus\dom\phi$. If $\cC^{(0)}\ge\cC^{(1)}\ge\dotsb\ge\cC^{(k)}$ then we call the collection of $k+1$ graphs a \emph{stack of candidate graphs}, and $\cC^{(\ell)}$ is the \emph{level $\ell$ candidate graph}.
 
 Given $x\in V(H)\setminus\dom\phi$ and $v\in V_x$, we form a stack of candidate graphs corresponding to the partial embedding $\phi\cup\{x\mapsto v\}$ according to the following \emph{update rule}. For each $0\le\ell\le k$, we let $\cC^{(\ell)}_{x\mapsto v}:=\cC^{(\ell)}_v$ be the link graph of $v$ in $\cC^{(\ell)}$. Note that trivially since $\cC^{(\ell)}\le\cC^{(\ell-1)}$ we have $\cC^{(\ell)}_{x\mapsto v}\le \cC^{(\ell-1)}_{x\mapsto v}$ for each $1\le\ell\le k$, so that this indeed gives a stack of candidate graphs.
\end{definition}

It will be important in what follows that we think of each $\cC^{(\ell)}$ both as specifying weights for an ongoing embedding of $H$, and also as a partite graph into which we expect to know the number of embeddings of some (small, not necessarily related to $H$) complex $R$.

We are now in a position to define a \emph{good partial embedding} (GPE). Informally, this is a partial embedding of $H$ together with a stack of candidate graphs, such that for each $1\le\ell\le k$ the graph $\cC^{(\ell)}$ is relatively dense and regular with respect to $\cC^{(\ell-1)}$. We specify the relative density of each $\cC\tl(e)$ explicitly in terms of numbers a density $k$-graph $\cD\tl$ with densities $d\tl(f)\in[0,1]$ for each $1\le\ell\le k$ and $f\subset V(H)$, which we think of as being the relative densities in the trivial GPE. 
We denote by $\cD\tl_\phi$ the density $k$-graph obtained from $\cD\tl$ by repeatedly taking the neighbourhood of vertices $x\in\dom\phi$, so that $\cD\tl_\phi$ gives the `current' relative densities of $\cC\tl$.

We will need a collection of parameters which describe, respectively, the minimum relative densities in each level of the stack (with respect to the level below) at any step of the embedding (denoted $\delta_\ell$), the required accuracy of counting in each level (denoted $\eta_\ell$), and the regularity required in each level. The regularity parameters are somewhat complicated. In general, one should focus on the best- and worst-case regularity; it is necessary to have the other parameters, but one only needs the extra granularity they offer in certain parts of the argument. 
Briefly, when we say $\cC^{(\ell)}(e)$ is $\eps_{\ell,r,h}$-regular, the $\ell$ indicates the level in the stack, $r=|e|$ gives the uniformity, and $h$ is the number of \emph{hits}, that is, how many times in creating $\phi$ we previously degraded the regularity of $\cC^{(\ell)}(e)$. This will turn out to be equal to
\[\pi_{\phi}(e):=\big|\{x\in\dom\phi:\{x\}\cup e'\in H\text{ for some $\emptyset\neq e'\subset e$}\}\big|\,.\] 
The maximum value of $\pi_\phi(e)$ that we could observe in the proof is related to a degeneracy-like property of $H$ which we now define. 
Given a fixed linear order on $V(H)$, write $\vdeg(H)$ for the \emph{vertex-degeneracy} of $H$, which is
\[
\vdeg(H):= \smash[b]{\max_{e\in H}}\abs[\Big]{\big\{x\in V(H) :\text{$x\le y$ for all $y\in e$, and $\{x\}\cup e'\in H$ for some $\emptyset\neq e'\subset e$}\big\}}\,.
\]
The definition of vertex-degeneracy was chosen precisely to make $\pi_{\phi}(e)\le \vdeg(H)$ hold for all unembedded $e\in H$ whenever $\phi$ is a partial embedding of $H$ with $\dom\phi$ an initial segment of $V(H)$.

\begin{definition}[Ensemble of parameters, valid ensemble]\label{def:ensemble}
 Given integers $k$, $c^*$, $h^*$, and $\Delta$, an \emph{ensemble of parameters} is a collection $\delta_1,\dotsc,\delta_k$ of \emph{minimum relative densities}, $\eta_0,\dotsc,\eta_k$ of \emph{counting accuracy parameters}, and $\big(\eps_{\ell,r,h}\big)_{\ell,r\in[k],\,h\in\{0,\dotsc,h^*\}}$ of \emph{regularity parameters}. For each $\ell\in[k]$ we define the \emph{best-case regularity} $\eps_\ell:=\min_{r\in[k],h\in\{0,\dotsc,h^*\}}\eps_{\ell,r,h}$ and the \emph{worst-case regularity} $\eps'_\ell:=\max_{r\in[k],h\in\{0,\dotsc,h^*\}}\eps_{\ell,r,h}$.
 
 An ensemble of parameters is \emph{valid} if the following statements all hold for each $1\le\ell\le k$.
 \begin{enumerate}[label=\textup{(VE\arabic*)}]
  \item\label{ve:worst}
  $\eta_0\ll \delta_1,\dotsc,\delta_\ell,\,\eta_\ell,\,k,\,c^*$ and for all $\ell'\in[\ell]$ we have $\eps'_{\ell'}\ll\delta_{\ell'},\dotsc,\delta_\ell,\, \eta_\ell,\, k,\,c^*,\,\Delta$ such that the following hold: 
  \begin{align}
  \eta_0 &\le \frac{\eta_\ell}{72(k+1)c^*} \prod_{0<\ell''\le\ell}\delta_{\ell''}^{c^*}\,,
\\\eps_{\ell'}' &\le \frac{\eta_\ell\delta_{\ell'}}{72k(k+1)\Delta^2} \prod_{\ell'<\ell''\le\ell} \delta_{\ell''}^{c^*} \,.  \end{align}
  \item\label{ve:reginh} For each $r\in[k]$ and $0\le h\le h^*-1$, we have $\eps_{\ell,r,h}\ll\eps_{\ell,r,h+1},\,\delta_\ell$ small enough for Lemma~\ref{lem:k-inherit} (inheritance and link regularity) with input $\delta_\ell$ and $\eps_{\ell,r,h+1}$. In particular $\eps_{\ell,r,h}$ increases with $h$.
  \item\label{ve:sizeorder} For each $r\in[k-1]$ we have $\eps_{\ell,r+1,h^*}\le\eps_{\ell,r,0}$.
  \item\label{ve:count} The counting accuracy $(4k+1)\eta_{\ell-1}$ is good enough for each application of Lemma~\ref{lem:k-inherit} as above. 
  That is, with inputs $\delta_\ell$ and $\eps_{\ell,r,h}$ for $1\le h\le h^*$ we have $(4k+1)\eta_{\ell-1}$ small enough to apply Lemma~\ref{lem:k-inherit}.
 \end{enumerate}
\end{definition}
By this definition, we always have $\eps_\ell=\eps_{\ell,k,1}$ and $\eps'_\ell=\eps_{\ell,1,h^*}$. It is important to observe that we can obtain a valid ensemble of parameters by starting with $\delta_k$ and $\eta_k$, choosing $\eps_{k,1,h^*}=\eps'_k$ to satisfy
\[
\eps_{k,1,h^*} \le \frac{\eta_k\delta_{k}}{72k(k+1)\Delta^2}\,,
\]
then choosing in order
\[
\eps_{k,1,h^*-1}\gg\dots\gg\eps_{k,1,0}\gg\eps_{k,2,h^*}\gg\dots\gg\eps_{k,2,1}\gg\dots\gg\eps_{k,k,0}=\eps_k\,,
\]
at which point we can calculate the required accuracy of counting $\eta_{k-1}$ and given $\delta_{k-1}$ choose $\eps'_{k-1}$ to match it, and repeat this process down the stack. In particular, this order of choosing constants is compatible with the strong hypergraph regularity lemma (see Section~\ref{sec:reg}), to which we would first input $\eps_k$, be given a $d_{k-1}$ which means we can specify $\delta_{k-1}$, then choose $\eps_{k-1}$, and be able to calculate $\delta_{k-2}$, and so on. 

Given a partial embedding $\phi$ of $H$, a stack of candidate graphs, $1\le\ell\le k$ and $e\subset V(H)\setminus\dom\phi$ with $|e|\ge 1$, we let $\ocC^{(\ell-1)}(e)$ denote the subgraph of $\cC^{(\ell-1)}$ induced by $\bigcup_{x\in e}V_x$. We let $\tcC^{\ell}(e)$ denote the graph obtained from $\ocC^{(\ell-1)}(e)$ by replacing the weights of edges in $V_e$ with the weights of $\cC^{\ell}(e)$. We will always consider regularity of $\tcC^{(\ell)}(e)$ with respect to $\ocC^{(\ell-1)}(e)$. This may seem strange; if we are working with unweighted graphs then there may be edges at all levels of the complex $\ocC^{(\ell-1)}(e)$ which are not in $\ocC^{(\ell)}(e)$, and so we are insisting on a regularity involving some edges of $\cC^{(\ell)}(e)$ which do not contribute to the count of embeddings into $\cC^{(\ell)}$. But it turns out to be necessary.

\begin{definition}[Good partial embedding]\label{def:gpe}
 Given $k\ge 2$, a $k$-complex $H$ of maximum degree $\Delta$, integers $c^*$, $h^*$, and for each $0\le\ell\le k$ a density $k$-graph $\cD\tl$ on $V(H)$, let $\delta_1,\dots,\delta_k$, $\eta_0,\dotsc,\eta_k$, and $\big(\eps_{\ell,r,h}\big)_{\ell,r\in[k],h\in[h^*]}$ be a valid ensemble of parameters. 
 Given $1\le\ell\le k$, we say that a partial embedding $\phi$ of $H$ together with a stack of candidate graphs $\cC^{(0)},\dots,\cC^{(\ell)}$ is an \emph{$\ell$-good partial embedding} ($\ell$-GPE) if
 \begin{enumerate}[label=\textup{(GPE\arabic*)}]
  \item\label{gpe:c0}
   The graph $\cC\tz$ is an $\big(\eta_0,c^*\big)$-THC graph with density graph $\cD\tz_\phi$.
  \item\label{gpe:reg}
  For each $1\le\ell'\le \ell$ and $\emptyset\neq e\subset V(H)\setminus\dom\phi$, the graph $\tcC^{(\ell')}(e)$ is $(\eps,d)$-regular with respect to $\ocC^{(\ell'-1)}(e)$, where
  \[\eps=\eps_{\ell',|e|,\pi_\phi(e)}\quad\text{ and }\quad d=d^{(\ell')}_\phi(e)=\prod_{\substack{f\subset V(H),\\e\subset f,\,f\setminus e\subset\dom\phi}}d^{(\ell')}(f)\,.\]
  \item\label{gpe:dens}
  The parameters $\delta_1,\dotsc,\delta_\ell$ are `global' lower bounds on the relative density terms in the sense that for each $1\le\ell'\le\ell$ and $\emptyset\neq e\subset V(H)\setminus\dom\phi$, we have
  \[
  \delta_{\ell'}\le \prod_{\substack{f\subset V(H),\\ e\subset f}}d\tlp(f)\,.
  \]
 \end{enumerate}
 When we have a $k$-good partial embedding, we will usually simply say \emph{good partial embedding} (GPE).
\end{definition}

If we were told that the trivial partial embedding was good, and that for every $x$ and $v\in V_x$, extending a good partial embedding $\phi$ of $H$ to $\phi\cup\{x\mapsto v\}$ and using the update rule to obtain a new stack of candidate graphs would result in a good partial embedding, then we would rather trivially conclude the desired counting lemma. 
We would simply count the number of ways to complete the embedding: when we come to embed some $x$ to $\cC^{(k)}(x)$ (with respect to the current GPE $\phi$) the density of $\cC^{(k)}(x)$ would be
\[
\prod_{\ell=0}^k d\tl_\phi(x) = \prod_{\ell=0}^k\prod_{\substack{f\subset V(H),\,x\in f, \\ f\setminus\{x\}\subset\dom\phi}}d\tl(f)
\]
up to a relative error which is small provided that for each $\ell$, all the $\eps_{\ell,r,h}$ are small enough compared to the $d\tl(f)$. Since this formula does not depend on a specific $\phi$ but only on $\dom\phi$ (so, on the order we embed the vertices) we conclude that the total weight of embeddings of $H$ is
\[
\tfrac{c\tk(\emptyset)}{\prod\limits_{0\le\ell\le k}d\tl(\emptyset)}\prod_{\ell=0}^k\cD\tl(H) = \tfrac{c\tk(\emptyset)}{\prod\limits_{0\le\ell\le k}d\tl(\emptyset)}\prod_{\ell=0}^k\prod_{f\subset V(H)}d\tl(f)
\]
up to a relative error which is small provided that for each $\ell$, all the $\eps_{\ell,r,h}$ are small enough given the $d\tl(f)$ and $v(H)^{-1}$. 
This is the statement we would like to prove. Of course, it is unrealistic to expect that we always get a good partial embedding when we extend a good partial embedding. However, it is enough if we typically get a good partial embedding, and the next lemma states that this is the case. 

\begin{lemma}[One-step Lemma]\label{lem:onestep} Given $k\ge 2$, a $k$-complex $H$ of maximum degree $\Delta$ and vertex-degeneracy $\vdeg(H)\le \Delta'$, positive integers $c^*$ and $h^*$, a valid ensemble of parameters, a partial embedding $\phi$ and stack of candidate graphs $\cC^{(0)},\dots,\cC^{(k)}$ giving a GPE, let $B_0(x)$ denote the set of vertices $v\in V_x$ such that condition~\ref{gpe:c0} does not hold for the extension $\phi\cup\{x\mapsto v\}$ and the updated candidate graph $\cC^{(0)}_{x\mapsto v}$. For $1\le\ell\le k$, let $B_\ell(x)$ denote the set of vertices $v\in V_x$ such that $\phi\cup\{x\mapsto v\}$ and the updated candidate graphs do not form an $\ell$-GPE. 

Then for every $1\le\ell\le k$ such that $\ell(4k+1)\le c^*$ and $\ell(4k+1)+k\Delta'\le h^*$, we have
\[
\vnorm{B_\ell(x)\setminus B_{\ell-1}(x)}{\cC^{(\ell-1)}(x)} \le k\Delta^2\eps'_\ell\vnorm{V_x}{\cC^{(\ell-1)}(x)}\,.
\]
\end{lemma}

The point of this collection of bounds on atypical vertices is that if a vertex $v$ is in $B_{\ell}(x)\setminus B_{\ell-1}(x)$ for some $\ell$, then we will be able to upper bound the count of $H$-copies extending $\phi\cup\{x\mapsto v\}$ in terms of the count of those $H$-copies in $\cC^{(\ell-1)}$ (which we show we can estimate accurately). 
This upper bound is bigger than the number we would like to get (the count in $\cC^{(k)}$) by the reciprocal of a product of some $d^{(\ell')}(f)$ terms, for various edges $f$ but only for $\ell'\ge\ell$. In particular, if $v(H)-\abs{\dom\phi}$ is not too large then this product is much larger than $\eps'_\ell$, so that the vertices of $B_{\ell}(x)\setminus B_{\ell-1}(x)$ in total do not contribute much to the overall count.

The corresponding counting lemma is then the following.

\begin{lemma}[Counting Lemma for GPEs]\label{lem:GPEcount} Given $k\ge 2$, positive integers $\Delta$, $c^*$, $h^*$, and a valid ensemble of parameters, let $\phi$ be a partial embedding of a $k$-complex $H$ of maximum degree $\Delta$, and suppose that for some $1\le\ell\le k$, the stack of candidate graphs $\cC\tz,\dots,\cC\tl$ gives an $\ell$-GPE. 
Write $r=v(H)-\abs{\dom\phi}$ and suppose that we have $c^*\ge\max\{2r-1, \ell(4k+1)\}$, $h^*\ge \ell(4k+1)+\vdeg(H)$, and $r\eta_\ell\le 1/2$. 
Then 
\begin{align}
\cC^{(\ell)}(H-\dom\phi)
  &=(1\pm r\eta_\ell)\tfrac{c\tl(\emptyset)}{\prod\limits_{0\le\ell'\le\ell}d_\phi^{(\ell')}(\emptyset)}\prod_{0\le\ell'\le\ell}\cD^{(\ell')}_\phi(H-\dom\phi)\,.
\end{align}
\end{lemma}
The right-hand side consists of a relative error term and a product of densities, where the $\emptyset$ terms correspond to edges of $F$ which are fully embedded by $\phi$, and the remaining terms correspond to the expected weight of edges not yet fully embedded by $\phi$.  

The proofs of Lemmas~\ref{lem:onestep} and~\ref{lem:GPEcount} are an intertwined induction, which we give in the following Section~\ref{sec:count}. Specifically, to prove Lemma~\ref{lem:onestep} for some $\ell\ge 1$ we assume Lemma~\ref{lem:GPEcount} for $\ell'<\ell$, and to prove Lemma~\ref{lem:GPEcount} for $\ell\ge 1$ we assume Lemma~\ref{lem:onestep} for $\ell'\le\ell$. 
The base case is provided by the observation that the counting conditions we require to prove Lemma~\ref{lem:onestep} for $\ell=1$, in $\cC^{(0)}$, hold because~\ref{gpe:c0} states that $\cC^{(0)}$ is a THC graph.

If one is only interested in a lower bound for the purpose of embedding, our methods are significantly simpler because we trivially have zero as a lower bound for the total weight of embeddings using bad vertices, and one can afford the luxury of ignoring levels below $k$ of the stack. 
Controlling this error is what requires $c^*\ge 2r-1$ in Lemma~\ref{lem:GPEcount}, but we would like to depend less on the global structure of $H$ in an embedding result, stated as Lemma~\ref{lem:GPEemb} below. 

\begin{lemma}[Embedding lemma for GPEs]\label{lem:GPEemb}
Given $k\ge 2$, positive integers $\Delta$, $\Delta'$, $c^*\ge k(4k+1)$, $h^*\ge k(4k+1)+k\Delta'$, and a valid ensemble of parameters, let $\phi$ be a partial embedding of a $k$-complex $H$ of maximum degree $\Delta$ and vertex-degeneracy at most $\Delta'$, and suppose that the stack of candidate graphs $\cC\tz,\dots,\cC\tk$ gives a $k$-GPE. 
Write $r=v(H)-\abs{\dom\phi}$. 

Then we have
\begin{align}
\cC^{(k)}(H-\dom\phi)
  &\ge (1- \eta_k)^r\tfrac{c\tk(\emptyset)}{\prod\limits_{0\le\ell\le k}d_\phi^{(\ell)}(\emptyset)}\prod_{0\le\ell\le k}\cD\tl_\phi(H-\dom\phi)\,.
\end{align}
\end{lemma}

Note that although Lemmas~\ref{lem:GPEcount} and~\ref{lem:GPEemb} only explicitly allow for counting embeddings in a partite graph where one vertex is embedded to each part, it is easy to deduce versions where multiple vertices may be embedded into each part by applying the standard construction at each level of the stack. 
It is trivial to check that for levels $1$ to $k$ the required regularity is carried over, and the homomorphism counts imposed on the bottom level are similarly preserved by the construction.

\section{A sketch of counting in dense graphs}\label{sec:sketch}

The following sketch proves what is perhaps the simplest non-trivial example of a counting lemma, and forms the basis of our methods. 
In a few places we use simple facts about dense, regular graphs that are much more difficult to prove in more general settings. 
Most of the technical work in this paper is dedicated to proving such facts in the setting of sparse hypergraphs. 

Let $X$, $Y$, and $Z$ be disjoint vertex sets in a graph $G$, each of size $n$, such that each pair of sets induces an $\eps$-regular bipartite graph of density $d$ with $\eps<d/2$. 
It is usual for us to use Szemerédi's original definition of regularity, which means here that for subsets $X'\subset Y$ and $Y'\subset Y$ each of size at least $\eps n$, there are $(d\pm\eps)\abs{X'}\abs{Y'}$ edges between $X'$ and $Y'$, and similar for the pairs $(X,Z)$ and $(Y,Z)$.
We sketch a proof that the number of triangles with one vertex in each set is $(d^3\pm \xi)n^3$ with an error $\xi$ polynomial in $\eps$ and $\eps /d$. 

The proof requires standard properties of a regular pair and some standard notation: $\deg(u,S)$ is the number of edges $\{u,v\}\in E(G)$ with $v\in S$, and $N(u)$ is the set $\big\{v \in V(G) : \{u,v\}\in E(G)\big\}$.
We first observe that, by regularity, for all but at most $4\eps n$ vertices $x\in X$ we have both $\deg(x,Y)$ and $\deg(x,Z)$ in the range $(d\pm\eps)n$. 
We also note that any vertex is in at most $n^2$ triangles. 
Another standard consequence of regularity is that for any typical $x\in X$ (that is, with $\deg(x,Y)$ and $\deg(x,Z)$ in the range $(d\pm\eps)n$), the pair $\big(N(x)\cap Y,N(x)\cap Z\big)$ \emph{inherits regularity} and is $\big(\frac{\eps}{d-\eps}, d\big)$-regular. 
We now consider $N(x)\cap Y$ and note that similarly, for all but at most 
\[
\tfrac{2\eps}{d-\eps}\abs{N(x)\cap Y}\le 2\eps\tfrac{d+\eps}{d-\eps}n\le 6\eps n
\]
vertices $y\in N(x)\cap Y$, the vertices $x$ and $y$ have
\[
\big(d\pm\tfrac{\eps}{d-\eps}\big)\abs{N(x)\cap Z}=\big(d\pm\tfrac{\eps}{d-\eps}\big)(d\pm\eps)n
\]
common neighbours in $Z$ (and so are in that many triangles); and the atypical vertices in $N(x)\cap Y$ contribute at most $(d+\eps)n\le n$ triangles each.

Pulling together these bounds there are at least zero and at most $4\eps n^3+6\eps n^3$ triangles using an atypical $x\in X$, and using a typical $x$ but an atypical $y\in Y$ respectively.
We also have the lower bound
\[
(1-4\eps)n\cdot \big(1-\tfrac{2\eps}{d-\eps}\big)(d-\eps)n \cdot \big(d-\tfrac{\eps}{d-\eps}\big)(d-\eps)n\,,
\]
and the upper bound
\[
n\cdot (d+\eps)n \cdot \big(d+\tfrac{\eps}{d-\eps}\big)(d+\eps)n\,,
\]
on the number of triangles using typical $x$ and $y$. 
Given $\eps<d/2$ we can bound $\eps/(d-\eps)\le 2\eps/d$, and hence the above sketch indeed shows that there are $(d^3\pm\xi)n^3$ triangles with $\xi$ polynomial in $\eps$ and $\eps/d$. 

For a more general version where one counts copies of some small graph $H$, one considers embedding $H$ into $G$ one vertex at a time, keeping track at each step of the number of ways to extend the next embedding. 
For the general argument we do two things: argue that most ways of continuing the embedding are `typical', and that `atypical' choices do not contribute much. 
More generally, `typical' simply means that neighbourhoods (and common neighbourhoods) of embedded vertices are about the size one would expect from the densities of the regular pairs, and that most vertices are typical is a simple consequence of regularity; and the atypical choices do not contribute much because they are so few. 

Our methods for sparse hypergraphs follow the same lines as this sketch, but some of the steps are significantly more involved. 
Adapting the sketch to sparse graphs requires some similar modifications (see~\cite{ABHKPblow}), and our terminology is chosen to follow these developments, but hypergraphs present their own technical challenges.

If it is given that $G$ has the THC property, then `every choice of image made so far in the partial embedding is typical' simply means that we always choose a vertex whose link gives a THC graph. By definition there are few vertices which are atypical. This is technically easy to work with (indeed, THC was designed to make this so).

If we are using GPEs, then our definition of a GPE (Definition~\ref{def:gpe}) is what it means for every choice of image made so far in the partial embedding to have been `typical'. The fact that most ways of continuing the embedding are typical is our Lemma~\ref{lem:onestep}, and the control of embeddings using atypical vertices appears in Lemma~\ref{lem:GPEcount}. This requires technically more work---mainly because the definitions are complicated---but has the advantage that one has access to the properties of the majorising graph $\Gam$, which (for example if $\Gam$ is a random graph) can be useful.

We finally point out that, in contrast to the above sketch where inheritance of regularity follows immediately from the definition, in sparse (hyper)graphs inheritance is not automatic. Our regularity inheritance lemma (Lemma~\ref{lem:k-inherit}) requires careful applications of the Cauchy--Schwarz inequality, and is crucial for proving Lemma~\ref{lem:onestep}.

\section{Counting and embedding in \texorpdfstring{$\Gamma$}{Gamma}}\label{sec:Gacount}

In this section we prove Theorem~\ref{thm:GaTHC} and Lemma~\ref{lem:randomTHC}. 
Both results give sufficient conditions for a $k$-graph $\Gam$ to be a THC-graph in a way that is compatible with the hypotheses of our counting and embedding results (Theorems~\ref{thm:counting} and~\ref{thm:embedding}). 

\subsection{Counting implies THC}

To prove Theorem~\ref{thm:GaTHC} we show the following. 
Suppose $J$ is a vertex set with a linear order. For convenience we will usually take $J=[m]$ with the natural order. Suppose $\Gam$ is an $[m]$-partite $k$-graph with density $k$-graph $\cD$. Suppose that $H$ is a $k$-complex on $[m]$ with $\Delta(H^{(2)})\le\Delta$, and that $\Gam$ is identically equal to $1$ on $e$-partite edges for each $e\not\in H$. Suppose that $c^*\ge 2\Delta+2$ and $\eps^*>0$ are given, and that that counts of all small (depending on $c^*$) subgraphs in $\Gam$ match those in $\cD$ to high accuracy (depending on $\eps^*$ and $c^*$). Then $\Gam$ is a $(c^*,\eps^*)$-THC graph.

A difficulty with proving this is that the definition of $(c^*,\eps)$-THC is recursive; it is not easy to verify whether a given graph satisfies the definition. So we will begin by defining a graph with some additional structure which helps us to perform the verification.

We say a set $X\subset [m]$ is a \emph{counting place} if $H^{(2)}[X]$ is a connected graph with at most $c^*$ vertices. We claim that it is enough to know accurate counts in $\Gam$ of small $X$-partite graphs for all counting places $X$. Making this precise, we have
\begin{proposition}\label{prop:countplace}
Given $k$, $c^*$, and $\eps>0$, let $J$ be a vertex set. Suppose that $\Gam$ is a $J$-partite $k$-graph, and $\cD$ is a density $k$-graph on $J$. Suppose that $H$ is a $k$-complex on $J$, suppose that if $e\not\in E(H)$ then $\Gam$ is identically $1$ on $V_e$ and $d(e)=1$, and suppose that for each counting place $X\subset J$ and each $X$-partite $k$-complex $F$ with at most $c^*$ vertices we have
\[\Gam(F)=(1\pm\eps)^{v(F)}\tfrac{\gam(\emptyset)}{d(\emptyset)}\cD(F)\,.\]
Then for each $X\subset J$ with $|X|\le c^*$ and each $X$-partite $k$-complex $F$ with at most $c^*$ vertices we have
\[\Gam(F)=(1\pm\eps)^{v(F)}\tfrac{\gam(\emptyset)}{d(\emptyset)}\cD(F)\,.\]
\end{proposition}
\begin{proof}
 Given $X$ and $F$, we are done if $X$ is a counting place, so suppose that $X$ is not a counting place. Then $H^{(2)}[X]$ has components on vertex sets $X_1,\dots,X_\ell$, and let for each $i$ the $k$-complex $F_i$ consist of the $X_i$-partite edges of $F$. By definition each $X_i$ is a counting place, so we have for each $i$
 \[\Gam(F_i)=(1\pm\eps)^{v(F_i)}\tfrac{\gam(\emptyset)}{d(\emptyset)}\cD(F_i)\,.\]
 Now since edges of $\Gam$ which are on sets $V_e$ for $e\not\in H$ are identically $1$, and since $H$ is a complex and hence down-closed, we have
 \[\Gam(F)=\gam(\emptyset)^{1-\ell}\prod_{i=1}^\ell\Gam(F_i)=\gam(\emptyset)^{1-\ell}\prod_{i=1}^\ell\Big((1\pm\eps)^{v(F_i)}\tfrac{\gam(\emptyset)}{d(\emptyset)}\cD(F_i)\Big)=(1\pm\eps)^{v(F)}\tfrac{\gam(\emptyset)}{d(\emptyset)}\cD(F)\,,\]
 as desired.
\end{proof}

Given a counting place $X$, and some $t<\min(X)$, we say a set $Y$ is \emph{of interest for $(t,X)$} if $Y$ consists of all vertices $1\le y\le t$ such that there exists $x\in X$ with $xy\in H^{(2)}$. 
We say $Y$ is \emph{of interest} if there exist $(t,X)$ such that $Y$ is of interest for $(t,X)$. 
Given a set $Y$ of interest, we let
\begin{align*}
 s_H(Y)&:=\big|\{z\in[m]:zy\in H^{(2)}\text{ for some }y\in Y \text{ with }y>z\}\big|+|Y|\,,\text{  and}\\
 p_H(Y)&:=\big|\{Y':Y\subsetneq Y'\,,\,Y'\text{ is of interest, and }\max(Y')=\max(Y)\}\big|\,.
\end{align*}
When $H$ is clear from the context we omit it. Finally, given a set $Y$, we say $e=\{e_1,\dots,e_{|Y|}\}\in V_Y$ is a \emph{fail set for $Y$} if there is some $(t,X)$ such that $Y$ is of interest for $(t,X)$ and such that the link graph $\Gam_e$ obtained by taking links of $\Gam$ with successively $e_1,e_2,\dots,e_{|Y|}$, and the similarly defined density graph $\cD_Y$, satisfy
\[\Gam_e(F)\neq(1\pm\eps)\tfrac{\gam_e(\emptyset)}{d_Y(\emptyset)}\cD_Y(F)\,,\]
for some $X$-partite $k$-complex $F$ with at most $c^*$ vertices.

We \emph{decorate} the $[m]$-partite $k$-graph $\Gam$ by choosing (possibly empty) sets $B^Y$ for each $Y$ which is of interest, where each $B^Y$ consists of edges in $V_{Y'}$ for some $Y'$ such that either $1\le |Y'|\le\Delta$ or $Y'=Y$. We will think of the $B^Y$ as \emph{bad sets}; we will say sets in $B^Y$ of size $|Y|$ are \emph{large bad sets} and the rest are \emph{small bad sets}.

We say a decorated graph $\Gam$ with bad sets $B^Y$ for $Y$ of interest, and density graph $\cD$, is \emph{$(H,\eps,\delta,c^*)$-safe} if the following are true. 
\begin{enumerate}[label=(S\arabic*)]
\item\label{safe:count} for each counting place $X$ and each $X$-partite graph $F$ with $v(F)\le c^*$ we have
\[\Gam(F)=(1\pm\eps)\tfrac{\gam(\emptyset)}{d(\emptyset)}\cD(F)\,,\]
\item\label{safe:fail} for each $Y$ of interest and each fail set $e$ for $Y$, there is a set of $B^Y$ contained in $e$, and
\item\label{safe:notbig} for each $Y$ of interest we have
\begin{equation}\label{eq:safe:notbig}
 \sum_{Z\subset Y}\frac{d(\emptyset)\sum_{e\in V_Z\cap B^Y}\prod_{e'\subset e}\gam(e')}{\gam(\emptyset)|V_Z|\prod_{Z'\subset Z}d(Z')}\le\delta^{s(Y)}2^{-p(Y)}\,.
\end{equation}
\end{enumerate}

If $\eps$ and $\delta$ are chosen appropriately, and $c^*\ge 2\Delta(H)+2$, then a $(H,\eps,\delta,c^*)$-safe graph will turn out to be an $(\eps^*,c^*)$-THC graph. The next lemma formalises this.
\begin{lemma}\label{lem:safe}
 Given integers $\Delta$ and $c^*$ such that $c^*\ge \Delta+2$, and $\delta>0$, if $\eps>0$ is sufficiently small the following holds. Given a $k$-complex $H$ on $[m]$ such that $H^{(2)}$ has maximum degree at most $\Delta$, let $\Gam$ be an $[m]$-partite $k$-graph with density graph $\cD$. Suppose that whenever $e\not\in E(H)$ we have $\Gam[V_e]$ identically equal to one. Suppose that for each $Y$ of interest we are given a set $B^Y$, decorating $\Gam$, such that the decorated graph is $(H,\eps,\delta,c^*)$-safe.
 
 Let $H'$ be the $k$-complex on $\{2,3,\dots,m\}$ with edges $E(H')=\{e\in E(H):1\not\in e\}\cup\{e\setminus\{1\}:e\in E(H),1\in e\}$. Given $v\in V_1$, let $\Gam_v$ be the link graph of $v$, and let $\cD_1$ be the link graph of $1$ in $\cD$. For all but at most $10\cdot 2^{(c^*+3)\Delta^{c^*+3}}\delta d(1)|V_1|$ total weight of vertices $v$ in $V_1$, there is a decoration of $\Gam_v$ with respect to which $\Gam_v$ is $(H',\eps,\delta,c^*)$-safe with density graph $\cD_1$.
\end{lemma}
The reader might wonder at this point why we do not combine~\ref{safe:fail} and~\ref{safe:notbig} (and avoid having decorations at all) by simply summing over fail sets for $Y$ rather than sets of $B^Y$. The reason is that we are not able to show the required bound~\ref{safe:notbig} typically continues to hold in $\Gam_v$; we have trouble if exceptionally many fail sets all contain one subset. However we can solve this problem by declaring such a subset to be itself bad, as we then do not have to control the fail sets which contain it.
\begin{proof}[Proof of Lemma~\ref{lem:safe}]
Given $\Delta$, $c^*\ge\Delta+2$, and $\delta>0$, we require $\eps>0$ to be small enough that
\[\delta^{\Delta^2c^*}2^{-2c^*\Delta^{c^*+1}}>2\Delta(c^*\Delta)^\Delta(4\Delta+8)\eps\]
 Since any $Y$ of interest in $H$ is a collection of neighbours in $H^{(2)}$ of some counting place $X$, which by definition is a set of size at most $c^*$, it follows that $|Y|\le\Delta c^*$ and that $s(Y)\le\Delta^2 c^*$. Since $X$ is connected in $H^{(2)}$, it follows that any two vertices of $Y$ are at distance at most $c^*+1$ in $H^{(2)}$, so the number of sets of interest containing any given vertex of $H$ is at most $1+\Delta+\dots+\Delta^{c^*+1}\le 2c^*\Delta^{c^*+1}$. In particular $p(Y)\le 2c^*\Delta^{c^*+1}$. It follows that
\[\delta^{s(Y)}2^{-p(Y)}>2\Delta|Y|^\Delta(4\Delta+8)\eps\]
holds for any $Y$ of interest in $H$.

 We begin by altering the sets $B^Y$ in order to avoid exceptionally many large bad sets all containing a small subset. Specifically, suppose $Y$ is of interest, and suppose $1$ is a neighbour in $H^{(2)}$ of at least one vertex of $Y$. We let $\bar{B}^Y$ be obtained as follows. We start with $\bar{B}^Y$ empty. Let $Z=Y\cap N_H^{(2)}(1)$. Whenever $e\in V_Z$ satisfies
\[\frac{1}{|V_{Y\setminus Z}|\prod_{W\subset Y:W\not\subset Z}d(W)}\sum_{\substack{e'\in V_{Y}\cap B^Y\\e\subset e'}}\prod_{f\subset e':f\not\subset e}\gam(f) > 1\]
we add $e$ to $\bar{B}^Y$. We then add all sets in $B^Y$ which are not supersets of any set in $\bar{B}^Y$. By construction, every set of $B^Y$ has a subset in $\bar{B}^Y$, so by~\ref{safe:fail} every fail set for $Y$ has a subset in $\bar{B}^Y$.

 Observe that by~\ref{safe:notbig}, we have
 \begin{equation}\label{eq:barnotbig}
 \sum_{Z\subset Y}\frac{d(\emptyset)\sum_{e\in V_Z\cap \bar{B}^Y}\prod_{e'\subset e}\gam(e)}{\gam(\emptyset)|V_Z|\prod_{Z'\subset Z}d(Z')}\le\delta^{s(Y)}2^{-p(Y)}\,.
\end{equation}
 Given $v\in V_1$ and any $Y$ of interest, let
 \[\bar{B}^Y_v:=\big\{e\in \bar{B}^Y:e\cap V_1=\emptyset\big\}\cup\big\{e\setminus\{v\}:e\in\bar{B}^Y,v\in e\big\}\,.\]
 
 If $v\in V_1$ has $\gam(v)=0$ we say $v$ is disallowed for $Y$. If $\gam(v)>0$, we say that $v$ is \emph{disallowed for $Y$} if either $1\in Y$ or $1y\in H^{(2)}$ and we have
\begin{equation}\label{eq:THC:dis}
 \sum_{\substack{Z\subset Y\\1\not\in Z}}\frac{d(\emptyset)d(1)\sum_{e\in V_Z\cap \bar{B}^Y_v}\prod_{e'\subset e\cup\{v\}}\gam(e')}{\gam(\emptyset)\gam(v)|V_Z|\prod_{Z'\subset Z\cup\{1\}}d(Z')}>\delta^{s(Y)-1}2^{-p(Y)-1}\,.
\end{equation}

 \begin{claim}\label{cl:fewdis} For each $Y$ such that either $1\in Y$ or $1y\in H^{(2)}$ for some $y\in Y$, the set of vertices $v\in V_1$ which are disallowed for $Y$ has total weight at most $10\delta d(1)|V_1|$.
 \end{claim}
 \begin{proof}
  We define
  \[S_1:=\big\{e\in\bar{B}^Y:\abs{e\cap V_1}>0\big\}\quad\text{and}\quad S_2:=\big\{e\cup\{v\}:e\in\bar{B}^Y,v\in V_1, e\cap V_1=\emptyset\big\}\,.\]
  We first aim to estimate
  \begin{equation}\label{eq:fewdis:bigsum}
 \sum_{W\subset Y\cup\{1\}}\frac{d(\emptyset)\sum_{e\in V_W\cap (S_1\cup S_2)}\prod_{e'\subset e}\gam(e')}{\gam(\emptyset)|V_W|\prod_{W'\subset W}d(W')}\,.
\end{equation}
 By~\eqref{eq:barnotbig}, the contribution to~\eqref{eq:fewdis:bigsum} made by edges $e\in S_1$ is at most $\delta^{s(Y)}2^{-p(Y)}$. It remains to estimate the contribution made by edges in $S_2$. These edges split up into the edges $S_L$ which come from large bad sets in $\bar{B}^Y$ and the edges $S_Z$ which come from small bad sets contained in $V_Z$ for some $Z\subset Y\setminus\{1\}$ of size at most $\Delta$. For a given $Z\subset Y\setminus\{1\}$ of size at most $\Delta$, we work as follows. First, we define three functions $\tX,\tY,\tW$ from $V_Z$ to $\mathbb{R}^+_0$, as follows. We set
 \[\tX(e):=\prod_{e'\subset e}\gam(e')\,,\quad\tY(e):=\frac{1}{|V_1|}\sum_{v\in V_1}\prod_{e'\subset e}\gam\big(e'\cup\{v\}\big)\quad\text{and}\quad\tW(e)=\begin{cases}1&\text{ if }e\in S_Z\\0&\text{ otherwise}\end{cases}\,.\]
 Letting $e$ be chosen uniformly at random in $V_Z$, these functions become random variables. Observe that $\Ex{\tX}$, $\Ex{\tX\tY}$ and $\Ex{\tX\tY^2}$ are, respectively, equal to $\Gam(K_Z)$, $\Gam(K_{Z\cup\{1\}})$ and $\Gam(K_{Z,1,1})$, where the $k$-complex $K_{Z,1,1}$ is obtained from $K_{Z\cup\{1\}}$ by duplicating the vertex $1$ (with the $(Z\cup\{1\})$-partition in which both copies of $1$ are assigned to part $1$). Since these graphs have at most $\Delta+2$ vertices, by Proposition~\ref{prop:countplace} we have
 \begin{align*}
  \Ex{\tX}&=(1\pm\eps)^{|Z|}\tfrac{\gam(\emptyset)}{d(\emptyset)}\prod_{Z'\subset Z}d(Z')\,,\\
  \Ex{\tX\tY}&=(1\pm\eps)^{|Z|+1}\tfrac{\gam(\emptyset)}{d(\emptyset)}\prod_{Z'\subset Z}d(Z')d(Z'\cup\{1\})\,,\quad\text{and}\\
  \Ex{\tX\tY^2}&=(1\pm\eps)^{|Z|+2}\tfrac{\gam(\emptyset)}{d(\emptyset)}\prod_{Z'\subset Z}d(Z')d(Z'\cup\{1\})^2\,.
 \end{align*}
 Thus the conditions of Lemma~\ref{lem:ECSdist} are met, with $d_{\subref{lem:ECSdist}}:=\prod_{Z'\subset Z}d(Z'\cup\{1\})$ and with $\eps_{\subref{lem:ECSdist}}:=(2\Delta+4)\eps$. We conclude from Lemma~\ref{lem:ECSdist} that
  \[\Ex{\tW\tX\tY}=\Big(1-\eps_{\subref{lem:ECSdist}}\pm 2\sqrt{\tfrac{\eps_{\subref{lem:ECSdist}}\Ex{\tX}}{\Ex{\tW\tX}}}\Big)\cdot d_{\subref{lem:ECSdist}}\Ex{\tW\tX}\le d_{\subref{lem:ECSdist}}\Ex{\tW\tX}+2\eps_{\subref{lem:ECSdist}}d_{\subref{lem:ECSdist}}\Ex{\tX}\,.\]
  Substituting back the definitions, and using the above estimate for $\Ex{\tX}$,  we have
  \[\frac{\sum_{e\in S_Z}\prod_{e'\subset e}\gam(e')}{|V_{Z\cup\{1\}}|}\le \prod_{Z'\subset Z}d(Z'\cup\{1\})\cdot\sum_{e\in S_Z}\frac{\prod_{e'\subset e}\gam(e')}{|V_Z|}+(4\Delta+8)\eps\tfrac{\gam(\emptyset)}{d(\emptyset)}\prod_{Z'\subset Z}d(Z')d(Z'\cup\{1\})\,,\]
  and rearranging we get
  \[\ \frac{d(\emptyset)\sum_{e\in S_Z}\prod_{e'\subset e}\gam(e')}{\gam(\emptyset)|V_{Z\cup\{1\}}|\prod_{Z'\subset Z\cup\{1\}}d(Z')}\le \frac{d(\emptyset)\sum_{e\in V_Z\cap\bar{B}^Y}\prod_{e'\subset e}\gam(e')}{\gam(\emptyset)|V_Z|\prod_{Z'\subset Z}d(Z')}+(4\Delta+8)\eps\,.\]
  
  We can obtain a similar estimate for the large bad sets. Note that $S_L$ is non-empty only if $Y$ does not contain $\{1\}$. We let $Z$ be the neighbours in $H^{(2)}$ of $1$. We define $\tX$ and $\tY$ exactly as above (with this $Z$), and for $e\in V_Z$ set
  \[\tW(e):=\frac{1}{|V_{Y\setminus Z}|\prod_{W\subset Y:W\not\subset Z}d(W)}\sum_{\substack{e'\in V_{Y}\cap B^Y\\e\subset e'}}\prod_{f\subset e':f\not\subset e}\gam(f) \,.\]
  By definition of $\bar{B}^Y$, we have $0\le\tW(e)\le 1$ for each $e\in V_Z$, so that (by exactly the same calculation as above) we can apply Lemma~\ref{lem:ECSdist} to estimate $\Ex{\tW\tX\tY}$. Observe that $d(Z')=\gam(e')=1$ if $e\in V_{Z'}$ and $Z'$ is a subset of $Y\cup\{1\}$ which is not contained in either $Z\cup\{1\}$ or $Y$, because such a $Z'$ must contain $1$ and some $y\in Y\setminus Z$ which is not in $H^{(2)}$, so that $Z\not\in H$. By the same calculation as above, and using this observation, we obtain
  \[\ \frac{d(\emptyset)\sum_{e\in S_L}\prod_{e'\subset e}\gam(e')}{\gam(\emptyset)|V_{Y\cup\{1\}}|\prod_{Z'\subset Y\cup\{1\}}d(Z')}\le \frac{d(\emptyset)\sum_{e\in V_Y\cap\bar{B}^Y}\prod_{e'\subset e}\gam(e')}{\gam(\emptyset)|V_Y|\prod_{Z'\subset Y\cup\{1\}}d(Z')}+(4\Delta+8)\eps\,.\]
  
  Now summing these bounds for all the $S_Z$ and for $S_L$, we have
  \[\sum_{Z\subset Y}\frac{d(\emptyset)\sum_{e\in S_Z}\prod_{e'\subset e}\gam(e')}{\gam(\emptyset)|V_{Z\cup\{1\}}|\prod_{Z'\subset Z\cup\{1\}}d(Z')}\le \sum_{Z\subset Y}\frac{d(\emptyset)\sum_{e\in V_Z\cap\bar{B}^Y}\prod_{e'\subset e}\gam(e')}{\gam(\emptyset)|V_Z|\prod_{Z'\subset Z}d(Z')}+2\Delta|Y|^{\Delta}(4\Delta+8)\eps\,,\]
  where the edges of $S_L$ are considered in the term $Z=Y$ and where we take $S_Z=\emptyset$ whenever $|Z|=0$ or $\Delta<|Z|<|Y|$ (so that the sum runs over less than $2\Delta|Y|^{\Delta}$ non-zero terms). Using~\eqref{eq:barnotbig}, we can substitute $\delta^{s(Y)}2^{-p(Y)}$ as an upper bound for the sum on the right hand side, and we have $S_Z=V_{Z\cup\{1\}}\cap S_2$, so we obtain
  \[\sum_{Z\subset Y}\frac{d(\emptyset)\sum_{e\in V_{Z\cup\{1\}}\cap S_2}\prod_{e'\subset e}\gam(e')}{\gam(\emptyset)|V_{Z\cup\{1\}}|\prod_{Z'\subset Z\cup\{1\}}d(Z')}\le \delta^{s(Y)}2^{-p(Y)}+2\Delta|Y|^{\Delta}(4\Delta+8)\eps\,.\]
  Putting this together with the already calculated contribution to~\eqref{eq:fewdis:bigsum} from $S_1$, we get
  \[
 \sum_{W\subset Y\cup\{1\}}\frac{d(\emptyset)\sum_{e\in V_W\cap (S_1\cup S_2)}\prod_{e'\subset e}\gam(e')}{\gam(\emptyset)|V_W|\prod_{W'\subset W}d(W')}\le 2\delta^{s(Y)}2^{-p(Y)}+2\Delta|Y|^{\Delta}(4\Delta+8)\eps\le 3\delta^{s(Y)}2^{-p(Y)}\,,
\]
 where the final inequality is by choice of $\eps$.
 \smallskip
 On the other hand, we can provide a lower bound on~\eqref{eq:fewdis:bigsum} by considering disallowed vertices for $Y$. By~\eqref{eq:THC:dis}, if $v\in V_1$ is disallowed for $Y$ then we have
 \[\sum_{\substack{Z\subset Y\\1\not\in Z}}\frac{d(\emptyset)\sum_{e\in V_Z\cap \bar{B}^Y_v}\prod_{e'\subset e\cup\{v\}}\gam(e')}{\gam(\emptyset)|V_{Z\cup\{1\}}|\prod_{Z'\subset Z\cup\{1\}}d(Z')}>\tfrac{\gam(v)}{d(1)|V_1|}\delta^{s(Y)-1}2^{-p(Y)-1}\,.\]
 and hence
 \[\sum_{\substack{W\subset Y\cup \{1\}\\1\in W}}\frac{d(\emptyset)\sum_{e\in V_{W\setminus\{1\}}\cap \bar{B}^Y_v}\prod_{e'\subset e\cup\{v\}}\gam(e')}{\gam(\emptyset)|V_{W}|\prod_{W'\subset W}d(W')}>\tfrac{\gam(v)}{d(1)|V_1|}\delta^{s(Y)-1}2^{-p(Y)-1}\,.\]
 Now for a given $W\subset Y\cup\{1\}$ such that $1\in W$, if $e\in V_W\setminus\{1\}$ then $e$ is in $\bar{B}^Y_v$ if and only if $e\cup\{v\}$ is in $S_1\cup S_2$. So the left hand side of the last equation is part of~\eqref{eq:fewdis:bigsum}, and furthermore as $v$ varies over $V_1$ we obtain all the terms in~\eqref{eq:fewdis:bigsum} exactly once. In particular, if the total weight of vertices in $V_1$ which are disallowed for $Y$ exceeds $10\delta d(1)|V_1|$, then summing over the disallowed vertices we get
 \[\tfrac{10\delta d(1)|V_1|}{d(1)|V_1|}\delta^{s(Y)-1}2^{-p(Y)-1}<3\delta^{s(Y)}2^{-p(Y)}\,,\]
 which is a contradiction. We conclude that the total weight of vertices in $V_1$ which are disallowed for $Y$ is at most $10\delta d(1)|V_1|$, as desired.
 \end{proof}
 
 Now suppose $v\in V_1$ is not disallowed for any $Y$ of interest in $H$. Let $Y\subset V(H')$ be of interest in $H'$. Observe that at least one of $Y$ and $Y\cup\{1\}$ must be of interest in $H$. We define
 \[\hat{B}^{Y}:=\begin{cases} \bar{B}^{Y}_v\quad&\text{if $Y$ is of interest in $H$ and $Y\cup\{1\}$ is not}\,,\\
  \bar{B}^{Y\cup\{1\}}_v\quad&\text{if $Y\cup\{1\}$ is of interest in $H$ and $Y$ is not}\,,\text{ and}\\
  \bar{B}^{Y}_v\cup\bar{B}^{Y\cup\{1\}}_v\quad&\text{if $Y$ and $Y\cup\{1\}$ are of interest in $H$}\,.\end{cases}\]
 We claim that $\Gam_v$, with density graph $\cD_1$, and decorations $\hat{B}^Y$ for each $Y\subset V(H')$ of interest in $H'$, is $(H',\eps,\delta,c^*)$-safe. To verify~\ref{safe:count}, let $X$ be a counting place in $H'$. By definition $X$ is also a counting place in $H$. If $\{1\}$ is not of interest for $(1,X)$, then $\Gam_v[V_X]=\Gam[V_X]$ and $\cD_1[X]=\cD[X]$, so~\ref{safe:count} holds for $X$. If $\{1\}$ is of interest for $(1,X)$, then considering the (only) $Z=\emptyset$ term of~\eqref{eq:THC:dis}, we see that if $\emptyset\in \hat{B}^{\{1\}}_v$ then the (only) term on the left hand side of~\eqref{eq:THC:dis} evaluates to $1$. Since $p(\{1\})=0$, we have $\delta^{s(\{1\})}2^{-p(\{1\})-1}\le\tfrac12<1$, which contradicts our assumption that $v$ is not disallowed for $\{1\}$. It follows that $\emptyset\not\in \hat{B}^{\{1\}}_v$, so $\{v\}\not\in\bar{B}^{\{1\}}$. In particular $\{v\}$ is not a fail set for $\{1\}$, so by definition~\ref{safe:count} holds for $X$.
 
 To check~\ref{safe:fail}, let $Y$ be of interest in $H'$ and let $e$ be a fail set in $\Gam_v$ for $Y$. Then either $\{v\}\cup e$ is a fail set in $\Gam$ for $\{1\}\cup Y$ which is of interest in $H$, or $e$ is a fail set in $\Gam$ for $Y$ which is of interest in $H$, or both. In any case, either $\bar{B}^{\{1\}\cup Y}$ contains a subset of $\{v\}\cup e$, in which case $\hat{B}^Y$ contains a subset of $e$, or $\bar{B}^Y$ contains a subset of $e$, in which case the same subset is in $\hat{B}^Y$.

 Finally, consider~\ref{safe:notbig}. Given $Y$ which is of interest in $H'$, there are three cases to consider.
 
 To begin with, suppose $Y$ is of interest in $H$ but $Y\cup\{1\}$ is not. If $1$ is not adjacent to any member of $Y$, then neither side of~\eqref{eq:safe:notbig} changes when we change $\Gam$ to $\Gam_v$ and $\cD$ to $\cD_1$ (Note that the quantities $\gam(\emptyset)$ and $d(\emptyset)$, which are not the same as $\gam_v(\emptyset)$ and $d_1(\emptyset)$ respectively, both cancel out), so we are done. If $1$ is adjacent to some vertices in $Y$, then by~\eqref{eq:THC:dis}, since $v$ is not disallowed,~\eqref{eq:safe:notbig} holds with a factor of $2$ to spare: the left hand side is exactly the left hand side of~\eqref{eq:THC:dis}, while on the right hand side we have $s_H(Y)=s_{H'}(Y)+1$ and $p_H(Y)=p_{H'}(Y)$.
 
 Now suppose $\{1\}\cup Y$ is of interest in $H$ but $Y$ is not. Then we have $s_{H'}(Y)=s_H(Y)-1$ and $p_{H'}(Y)=p_H(Y)$, so as above, by~\eqref{eq:THC:dis} we see that~\eqref{eq:safe:notbig} holds with a factor of $2$ to spare.
 
Finally, suppose both $Y$ and $\{1\}\cup Y$ are of interest in $H$. Then $p_{H'}(Y)=p_H(Y)-1$, so whether or not $1$ is adjacent to a member of $Y$, the contribution of sets in $\bar{B}^Y$ to the left hand side of~\eqref{eq:safe:notbig} is by~\eqref{eq:THC:dis} at most half of the right hand side bound. Furthermore, exactly as above, the contribution of the sets in $\bar{B}^{Y\cup\{1\}}$ is at most half of the right hand side; so~\eqref{eq:safe:notbig} holds as desired.
 
 To complete the proof, we just need to show that the total weight of disallowed vertices in $V_1$ is small. Given $Y$ which is of interest in $H$, there are two possibilities. First, $1$ is neither in $Y$ nor adjacent to any member of $Y$ in $H^{(2)}$. In this case only vertices $v\in V_1$ with $\gam(v)=0$ are disallowed for $Y$, so the total weight of vertices in $V_1$ disallowed for $Y$ is $0$. Second, $1$ is either in or adjacent to $Y$. Since $Y$ is of interest, in particular there is some counting place $X$ such that $Y$ is of interest for $\big(\max(Y),X\big)$. Now $X$ is connected in $H^{(2)}$ and has at most $c^*$ vertices, so it has diameter at most $c^*-1$. All members of $Y$ are adjacent to at least one member of $X$ in $H^{(2)}$, and $1$ is either in $Y$ or adjacent in $H^{(2)}$ to at least one member of $Y$. It follows that all members of $Y$ are at distance at most $c^*+3$ in $H^{(2)}$ from $1$. There are at most $(c^*+3)\Delta^{c^*+3}$ such vertices, so the number of possibilities for $Y$ is at most $2^{(c^*+3)\Delta^{c^*+3}}$. By Claim~\ref{cl:fewdis}, for each such $Y$ the total weight of disallowed vertices is at most $10\delta d(1)|V_1|$, giving the claimed bound.
\end{proof}

With Lemma~\ref{lem:safe} it is now easy to prove Theorem~\ref{thm:GaTHC}. We simply need to show that if $\Gam$ is a graph in which we can count (rather larger than $c^*$-vertex) graphs to sufficiently high accuracy, then there is a decoration of $\Gam$ which is safe. The decoration we use is trivial---for each $Y$ of interest we let $B^Y$ be the fail sets for $Y$. The only condition which is not trivially satisfied is~\ref{safe:notbig}, and we show that this condition holds by another application of the Cauchy--Schwarz inequality.

\begin{proof}[Proof of Theorem~\ref{thm:GaTHC}]
It suffices to give a decoration for $\Gam$ which yields a $(H,\eps, \delta, c^*)$-safe graph where
\begin{enumerate}
\item
$\eps\le\eta'$ so that \ref{safe:count} gives the required counting for \ref{thc:count} in any such safe graph,
\item $\delta^{\Delta^2c^*}2^{-2c^*\Delta^{c^*+1}}>2\Delta(c^*\Delta)^\Delta(4\Delta+8)\eps$
so that we may apply Lemma~\ref{lem:safe}, and
\item
$20\cdot 2^{(c^*+3)\Delta^{c^*+3}}\delta \le \eta'$
so that the lemma guarantees that the total weight of vertices in $V_1$ whose link we do not know how to decorate to form another $(H,\eps, \delta, c^*)$-safe graph is at most a fraction $\eta'$ of the total weight in $V_1$.
\end{enumerate}
To verify the third condition note that with $\eps<\eta'$ and~\ref{safe:count} we have
$(1-\eta')d(1) \le \vnorm{V_1}{\Gam}$, and hence it suffices to have 
\[
10\cdot 2^{(c^*+3)\Delta^{c^*+3}}\delta d(1)\abs{V_1} \le \frac{\eta'}{2}d(1)\abs{V_1} < \eta'(1-\eta')d(1)\abs{V_1} \le \eta'\vnorm{V_1}{\Gam}\cdot |V_1|\,.
\]
It is easy to see that this relationship between the constants is possible to satisfy; we have $\delta\ll\eta',\,\Delta,\,c^*$ so can pick $\delta$ small enough first, and then $\eps\ll\delta,\,\eta',\,\Delta,\,c^*$ can be chosen small enough. 
For the remainder of the proof, fix $\delta$ and $\eps$ such that the above inequalities hold.
We also state now that taking $\eta$ small enough that $3\eta^{14}\le \eps$ and 
\[
2^{4+\Delta^2+2^{1+c^*}}(c^*)^{c^*} \eta^{1/4} \le \delta^{\Delta^2c^*}2^{-2c^*\Delta^{c^*+1}}
\]
suffices.

We now construct the decoration. 
The statement of the theorem gives us a density graph, hence we need only supply the sets $B^Y$ for each $Y$ of interest. 
We show that the trivial decoration suffices, i.e.\ for each $Y\subset J$ of interest, let $B^Y\subset V_Y$ be the fail sets for $Y$. 
Given this decoration, the fact that $\eta\le\eps$ and the hypotheses of the theorem give conditions~\ref{safe:count} and~\ref{safe:fail} immediately; it remains to check that~\ref{safe:notbig} holds.  

Given our trivial definition of $B^Y$, note that the only $Z\subset Y$ which can contribute to the sum in~\eqref{eq:safe:notbig} is $Y$ itself, hence it suffices to show that for each $Y$ of interest, we have
\begin{equation}\label{eq:safe:notbig:trivial}
 \frac{d(\emptyset)\sum_{e\in B^Y}\prod_{e'\subset e}\gam(e')}{\gam(\emptyset)|V_Y|\prod_{Z'\subset Y}d(Z')}\le\delta^{s(Y)}2^{-p(Y)}\,.
\end{equation}

It is a simple application of Corollary~\ref{cor:ECSconc} to bound the contribution to the left-hand side of~\eqref{eq:safe:notbig:trivial} that arises from all fail sets for a particular $X$ and $F$, which we now give.

Given a fixed counting place $X$, $t<\min(X)$, a set $Y$ of interest for $(t,X)$, and an $X$-partite $k$-complex $F$ on at most $c^*$ vertices, let $B^{Y,X,F}\subset B^Y$ be the fail sets of $Y$ that fail because 
\[
\Gam_e(F)\neq(1\pm\eta)\tfrac{\gam_e(\emptyset)}{d_Y(\emptyset)}\cD_Y(F)
\]
holds. 
Then $B^Y=\bigcup_{X,F}B^{Y,X,F}$ where the union is over all $X$ such that there exists $t<\min(X)$ such that $Y$ is of interest for $(t,X)$ and all $X$-partite $k$-complexes $F$ on at most $c^*$ vertices. 
Let $\Gam^{Y,X,F}$ be equal to $\Gam$ except on $V_Y$ where we set $\gam^{Y,X,F}(e)=\gam(e)\indicator{e\in B^{Y,X,F}}$. 
Now the left-hand side of~\eqref{eq:safe:notbig:trivial} is equal to 
\begin{equation}\label{eq:lhs:XFsum}
\sum_{X,F}\frac{d(\emptyset)}{\gam(\emptyset)} \frac{\Gam^{Y,X,F}(H[Y])}{\cD(H[Y])}\,,
\end{equation}
and we bound each term in the sum.

Consider the $k$-complex $2F$ formed of two vertex-disjoint copies of $F$, and let $\pi:V(2F)\to X$ be the projection from vertices of $2F$ to their part.
We can view $F$ as a subgraph of $2F$ hence $\pi$ also serves as the analogous projection for $F$, and we also extend $\pi$ to be the identity function on $Y$.
Let $F'$ and $F''$ be obtained from $F$ and $2F$ respectively by adding the vertices $Y$ and any edge $f$ such that $\pi(f)\in H[X\cup Y]$. 
Note that $v(F')\le v(F'')\le 2c^*+\abs{Y} \le (\Delta+2)c^*$.
Define the functions $\tX,\tY:V_Y\to \nnREALS$ by
\begin{align*}
  \tX(x_Y)&= \prod_{e\subset Y}\gam(x_e)  \,,\quad\text{and}\\
  \tY(x_Y)&=\Ex[\Big]{\prod_{e\in F,\, e\not\subset Y}\gam(x_e)}[x_j\in V_{\pi(j)} \text{ for } j\in V(F)\setminus Y]\,,
\end{align*}
which become random variables when $x_Y\in V_Y$ is chosen uniformly at random, and let $d\subref{cor:ECSconc} := \cD(F')/\cD(H[Y])$. 
Then by the hypotheses of the theorem we have
\begin{align*}
  \Ex{\tX}      &= \Gam(H[Y]) = (1\pm\eta)\cD(H[Y])\,,\\
  \Ex{\tX\tY}   &= \Gam(F')   = (1\pm\eta)\cD(F') = (1\pm 4\eta)d\subref{cor:ECSconc}\cdot\Ex{\tX}\,,\\
  \Ex{\tX\tY^2} &= \Gam(F'')  = (1\pm\eta)\cD(F'') = (1\pm 4\eta)d\subref{cor:ECSconc}^2\cdot\Ex{\tX}\,,
\end{align*}
and hence by Corollary~\ref{cor:ECSconc} with $\eps\subref{cor:ECSconc}:= 4\eta$, the random variable $\tW$ which indicates the event $\tY = (1\pm 3\eta^{1/4})d\subref{cor:ECSconc}$ satisfies $\Ex{\tW\tX} \ge (1-6\eta^{1/4})\Ex{\tX}$.
Considering the complementary event, and rewriting this in terms of $\Gam$ and $\cD$ via the above estimates, we have
\[
\Ex{(1-\tW)\tX} \le 9\eta^{1/4}\cD(H[Y])\,.
\]
But observe that $x_Y\in B^{Y,X,F}$ if and only if 
\[
\Gam_{x_Y}(F)\neq(1\pm\eps)\frac{\gam_{x_Y}(\emptyset)}{d_Y(\emptyset)}\cD_Y(F)\,,
\]
which occurs only if $\tW=0$ since $\tY = \Gam_{x_Y}(F)/\gam_{x_Y}(\emptyset)$, $d\subref{cor:ECSconc}=\cD(F')/\cD(H[Y]) = \cD_Y(F)/d_Y(\emptyset)$, and $3\eta^{1/4}<\eps$. 
We conclude that 
\[
\frac{d(\emptyset)}{\gam(\emptyset)} \frac{\Gam^{Y,X,F}(H[Y])}{\cD(H[Y])} \le 9\eta^{1/4}\,,
\]
which gives a bound on each term in the sum~\eqref{eq:lhs:XFsum}.
To bound the number of terms, we use that $X$ is a counting place, $Y$ is a set of neighbours of $X$, and $F$ is an $X$-partite $k$-complex on at most $c^*$ vertices. 
These facts give the following crude estimates. 
Firstly, for any $Y$ of interest $\abs{Y}\le\Delta c^*$, and there are at most $2^{\Delta^2c^*}$ sets $X$ for which there exists $t$ such that $Y$ is of interest to $(t,X)$. 
Secondly, any such $X$ has size at most $c^*$ and $F$ has at most $c^*$ vertices, so there are at most $(c^*)^{c^*}$ ways of choosing the vertex partition of $F$ indexed by $X$, and at most $2^{2^{c^*}}$ choices for the edges of $F$.
Then there are at most $2^{\Delta^2c^*}\cdot(c^*)^{c^*}\cdot2^{2^{c^*}}\le 2^{\Delta^2+2^{1+c^*}}(c^*)^{c^*}$ terms in the sum~\eqref{eq:lhs:XFsum}, and our calculations show that the trivial decoration we consider has
\[
 \sum_{Z\subset Y}\frac{d(\emptyset)\sum_{e\in V_Z\cap B^Y}\prod_{e'\subset e}\gam(e')}{\gam(\emptyset)|V_Z|\prod_{Z'\subset Z}d(Z')}\le 2^{4+\Delta^2+2^{1+c^*}}(c^*)^{c^*} \eta^{1/4}\,,
\]
and $\eta$ was chosen small enough that this gives~\ref{safe:notbig}, because (as in the proof of Lemma~\ref{lem:safe}) we have $s(Y)\le\Delta^2c^*$ and $p(Y)\le 2c^*\Delta^{c^*+1}$.
\end{proof}

\subsection{Random hypergraphs have THC}

We now turn to the proof of Lemma~\ref{lem:randomTHC}, where we recall that $\Gam$ is a random $k$-uniform hypergraph and we obtain $\Gam'$ from $\Gam$ by partitioning $V(\Gam)$ into $\abs{J}$ parts and applying the standard construction with a $J$-partite $k$-complex $H$ of maximum degree $\Delta$. 
To show that $\Gam'$ is a THC-graph involves showing that counts of complexes $R$ in $\Gam'$ and in related $k$-graphs obtained by taking links are close to their expectation, and such counts will correspond to counts of weighted homomorphism-like objects in $\Gam$. 
A difficulty arises here because $\Gam'$ may contain multiple copies of a single edge of $\Gam$ and so the $k$-edge weights in $\Gam'$ are not necessarily independent Bernoulli random variables. 
In order to avoid trying to deal with $\Gam$ and $\Gam'$ simultaneously, we first state and prove the required property of $\Gam$, which is rather technical.

Let $Y$ be an initial segment of $X$ and $\phi:Y\to V(\Gam)$ be a partite map. 
Let $Z$ be a vertex set disjoint from $Y$, equipped with a map $\rho$ that associates each $z\in Z$ to some $\rho(z)\in X\setminus Y$. 
For convenience we extend $\rho$ to be the identity map on $Y$. 
Let $R$ be a $J$-partite $k$-complex on $Z$, and let $\Rphi$ be the hypergraph with vertex set $\im\phi\cup Z$, and edge set
\[
E(\Rphi) := E(R)\cup\{f\subset \dom\phi\cup Z : f\cap Z\ne\emptyset,\, \rho(f)\in H \}\,.
\]
We view $\Rphi$ as $J$-partite in the following way. 
Each vertex in $\im\phi$ is in $V_j$ for some $j\in J$, which naturally gives an association to the index $j$, and vertices in $Z$ are related to indices $j$ through the map $\rho:Z\to X$ and the partition of $X$ into parts indexed by $J$. 
We write $V_z$ for the $V_j$ to which $z\in Z$ is associated in this way.

Then the homomorphism-like objects we consider in $\Gam$ are partite maps $\psi$ from $\Rphi$ to $\Gam$, where we insist that $\psi$ extends the identity map on $\im\phi$. 
This definition is rather difficult to parse, but a certain amount of complexity is necessary to deal with the case that $\phi$ is not injective. 
In any case, the idea is that $\psi$ signifies a copy of $R_\phi$ in $\Gam$ `rooted' at some fixed vertices specified by $\im\phi$. 
We are interested in weighting such $\psi$ according to the subset $\Rphi^{(\ge2)}\subset\Rphi$ of edges of size at least $2$, preferring to deal separately with the empty set (which has weight $1$ in this setup) and vertex weights. 
For $z\in Z$, let $U_z\subset V_z$ be a set of exactly $n_1:= n p^d/(2\log n)$ vertices. 
We define 
\begin{equation}\label{eq:Ndef}
N(\phi, R, U_Z) := \sum_{\psi}\prod_{e\in\Rphi^{(\ge2)}}\gam\big(\psi(e)\big)\,,
\end{equation}
where the sum is over all maps $\psi:\im\phi\cup Z\to V(\Gam)$ such that $\psi(w)=w$ for any $w\in\im\phi$ and $\psi(z)\in U_z$ for all $z\in Z$. 
Note that with $Y=\phi=\emptyset$ we have $\Rphi=R$, and since $\Gam$ is complete on edges of size at most $1$, $n_1^{-\abs{Z}}N(\phi,R,U_Z)$ is then the partite count of copies of $R$ in $\Gam$ that lie on $U_Z$.

The main probabilistic tool we require for counting in $\Gam'$ is a statement that for any suitable $R$, the count $N(\phi, R, U_Z)$ is close to its expectation with very high probability. 
It turns out that we are interesting in $R$ of the following form. 
Given $Y$, let $H^4$ be the complex $H$ with each vertex blown up into $4$ copies. 
A \emph{suitable} $R$ is any subcomplex of $H^4$ on at most $c^*$ vertices which uses no copies of vertices in $Y$. 
Considering suitable $R$ is what requires us to work with $4^k\Delta$ and $4^kd$ in what follows.

\begin{claim}\label{claim:Nconc}
Consider the setup of Lemma~\ref{lem:randomTHC} and a suitable $R$ (according to the above definitions). 
Then with probability at least $1-\exp\big(-O(n^{1+\eps})\big)$ we have
\begin{equation}\label{eq:Nconc}
N(\phi, R, U_Z) = \big(1\pm\tfrac{1}{\log n}\big)n_1^{\abs{Z}}\prod_{e\in \Rphi^{(\ge2)}}q(e)
\,.
\end{equation}
\end{claim}

\begin{claimproof}
Formally, we proceed by induction on $\abs{Z}$. The claim is trivial if $\abs{Z}\le 1$, as the product over $E_2(\Rphi)$ is empty. 

If $\abs{Z}\ge 2$, note that it suffices to consider injective maps $\psi$ in~\eqref{eq:Ndef}. 
Any non-injective partite map $\psi':\im\phi\cup Z\to V(\Gam)$ of the form considered in~\eqref{eq:Ndef} is an injective partite map into $V(\Gam)$ from the complex $R'$ on a vertex set $Z'$ formed from $R$ by identifying any vertices of $z$ with the same images under $\psi'$. 
Applying the claim to $R'$ (which is on fewer vertices), we see that with probability at least $1-\exp\big(-O(n^{1+\eps})\big)$, these non-injective maps contribute an amount at most twice expectation of $N(\psi,R', U_{Z'})$. 
Comparing the expectations of $N(\psi,R, U_{Z})$ and $N(\psi,R', U_{Z'})$, identifying a pair $z,z'$ of vertices in $Z$ `costs' a factor $n_1$ but can gain a factor up to $p^{-4^k\Delta}$ since the edges involving $z$, of which there are at most $4^k\Delta$, are now coupled in $\Gam$ with those containing $z'$. 
Then the assumptions on $n$ and $p$ imply that the contribution to $N(\phi,R,U_Z)$ from non-injective homomorphisms is at most a factor $O(n^{-\eps})$ times the expected contribution from injective homomorphisms. 
Write $N^*$ for the contribution to $N(\phi, R, U_Z)$ from injective $\psi$, noting that the above argument shows that $N(\phi,R, U_Z)=\big(1\pm O(n^{-\eps})\big)N^*$.

For each injective $\psi$, the term $\mathbf{X}_\psi^*:=\prod_{e\in \Rphi^{(\ge2)}}\gam\big(\psi(e)\big)$ appearing in $N^*$ is a product of independent Bernoulli random variables with probabilities given by $q(e)$. 
The $\mathbf{X}_\psi^*$ themselves are therefore `partly dependent' Bernoulli random variables, each with the same probability $p^*=\prod_{e\in \Rphi^{(\ge2)}}q(e)$. 
Since we consider only the edges $\Rphi^{(\ge2)}$, if $\mathbf{X}_\psi^*$ and $\mathbf{X}_{\psi'}^*$ are dependent it must be because they agree on at least two vertices of $Z$. 
Then each $\mathbf{X}_\psi^*$ can be dependent on at most $\binom{\abs{Z}}{2}n_1^{\abs{Z}-2}$ other variables $\mathbf{X}_{\psi'}^*$.
We apply a theorem of Janson~\cite[Corollary 2.6]{Jlargedeviations} which bounds the probability of large deviations in sums of partly dependent random variables. 

\begin{theorem}[Janson~\cite{Jlargedeviations}]\label{thm:janson}
Let $\Psi$ be an index set, and $N^*=\sum_{\psi\in\Psi}\mathbf{X}^*_\psi$, such that each $\mathbf{X}^*_\psi$ is a Bernoulli random variable with probability $p^*\in(0,1)$. 
Let $\Delta_1^*$ be one more than the maximum degree of the graph on vertex set $\Psi$ such that $\psi$ and $\psi'$ are adjacent if and only if $\mathbf{X}_\psi^*$ and $\mathbf{X}_{\psi'}^*$ are dependent. 
Then for any $\delta>0$,
\begin{equation}\label{eq:janson}
\Pr\big[N^* = (1\pm\delta)\Ee N^*\big]\ge 1 - 2\exp\Bigg(-\frac{3\delta^2\abs{\Psi}p^*\big(1-\Delta_1^*/\abs{\Psi}\big)}{8\Delta_1^*}\Bigg)\,.
\end{equation}
\end{theorem}

In the setup above, we have $n_1^{\abs{Z}}(1-\abs{Z}/n_1)^{\abs{Z}} \le \abs{\Psi} \le n_1^{\abs{Z}}$ and $\Delta_1^*\le \abs{Z}^2 n_1^{\abs{Z}-2}$. 
Since $\abs{Z}\le c^*$ is bounded by a constant, this means $\abs{\Psi}=\big(1\pm O(n_1^{-1})\big)n_1^{\abs{Z}}$,
\begin{align}
\frac{\abs{\Psi}}{\Delta_1^*} &=\Omega(n_1^2)\,,&
&\text{and}
&\frac{\Delta_1^*}{\abs{\Psi}} & =O(n_1^{-2})\,.
\end{align}
Moreover, we know that $p^*\le p^{(\abs{Z}-1)4^kd}$ since embedding the first vertex of $Z$ is `free', and each remaining vertex can be the last vertex of at most $4^kd$ edges which occur with probability $p$ each. 
Then for $\delta=1/(2\log n)$, the exponent on the right-hand side of~\eqref{eq:janson} is 
\[
-\Omega\big(p^{4^kd\abs{Z}}n^2\big) = -\Omega\big(n^{1+\eps}\big)\,,
\]
by the assumptions on $n$ and $p$. 
The claim follows since the event that \eqref{eq:Nconc} that we wish to control occurs with probability $1$ provided $N^*=(1\pm\delta)\Ee N^*$ and $n$ is large enough. 
We have 
\[
\Ee N^* = \abs{\Psi}\prod_{e\in \Rphi^{(\ge2)}}q(e) = \big(1\pm O(n_1^{-1})\big)n_1^{\abs{Z}}\prod_{e\in \Rphi^{(\ge2)}}q(e)\,,
\] 
and hence for large enough $n$, with probability at least $1-\exp\big(-O(n^{1+\eps})\big)$,
\[
N(\phi,R)=\big(1\pm O(n^{-\eps})\big)N^* = \big(1\pm O(n^{-\eps})\big)\Big(1\pm \tfrac{2}{\log n}\Big)\Ee N^* = \big(1\pm \tfrac{1}{\log n}\big)n_1^{\abs{Z}}\prod_{e\in \Rphi^{(\ge2)}}q(e)\,.\qedhere
\]
\end{claimproof}

With the main probabilistic argument complete, we can now apply Claim~\ref{claim:Nconc} to the problem of showing $\Gam'$ is an $(\eta,c^*)$-THC graph.

\begin{proof}[Proof of Lemma~\ref{lem:randomTHC}]

We start with a sketch of the proof. 
Given a fixed partition $\{V_j\}_{j\in J}$, to verify $\Gam'$ is an $(\eta,c^*)$-THC graph we must count suitable $X$-partite complexes $R$ in graphs obtained from $\Gam'$ by embedding vertices of $H$. 
We are not required to consider arbitrary embeddings, at each step we are permitted by~\ref{thc:hered} to avoid a `bad set' of potential images, which we will exploit in due course. 

At first, no vertex of $H$ has been embedded and we count $R$ in $\Gam'$, which by the standard construction is the same as counting $R$ in $\Gam$.
By Claim~\ref{claim:Nconc} with $Y=\phi=\emptyset$ and a union bound over suitable $R$, with high probability we have the required accurate counts of $R$ in $\Gam$.
These counts are accurate enough to imply deterministically that there is a small `bad set' which, if avoided, allows us to embed the next vertex $x$ and continue the argument with `well-behaved' vertex weights in $\Gam'_x$. 

When some initial segment $Y$ of $X$ has been embedded, say by a map $\phi':Y\to V(\Gam')$, 
we always have an associated map $\phi:Y\to V(\Gam)$ obtained by identifying the copies of parts $V_j$ made in the standard construction. 
Write $\Gam'_{\phi'}$ for the $k$-graph obtained from $\Gam'$ by taking the link of vertices in $\im\phi'$.
By construction, the required counts of complexes $R$ in $\Gam'_{\phi'}$ correspond to counts of $\Rphi$ in $\Gam$, which we can control with Claim~\ref{claim:Nconc}. 
We handle vertex weights separately, and apply Claim~\ref{claim:Nconc} with subsets $U_z\subset V_z$ of vertices that receive weight $1$ in $\Gam'_{\phi'}$.
We then take a union bound over choices of partition to complete the lemma.

The notion of `well behaved' for vertex weights in $\Gam'_{\phi'}$ that we maintain is as follows. 
Recall that given a partition $\{V_j\}_{j\in J}$, we have $V(\Gam')$ partitioned into $\{V'_x\}_{x\in X}$ where $V'_x$ is a copy of the $V_j$ into which $x$ will be embedded. 
Since we view vertices in $Z$ as copies of vertices in $X$, we also write $V'_z$ for the part of $\Gam'$ into which $z$ should be embedded. 
Given $Y\subset X$ and $\phi$, $\phi'$ as above, let $\cQ_\phi$ be the density $k$-graph obtained from $\cQ$ by taking links of vertices in $\im\phi$. 
For a fixed suitable $R$ on vertex set $Z$, let $\cA_{Y,z}$ be the event that $\vnorm{V_z'}{\Gam_W}\ge (1-\eta)^{\pi(z)}q_Y(z)$, where $\pi(z) = \abs[\big]{\{y\in Y : \{y,z\}\in H\}}\le \Delta(H)$, and let $\cA_Y$ be the intersection of $\cA_{Y,z}$ for all $z\in Z$. 
The event $\cA_\emptyset$ holds with probability $1$ because $\Gam$ gives weight $1$ to all vertices, and by avoiding bad vertices we will maintain $\cA_Y$ as we embed.

We are now ready to give the main proof, supposing that the initial segment $Y\subset X$ has been embedded, we have the associated partite maps $\psi$ and $\psi'$ from $Y$ to $V(\Gam)$ and $V(\Gam')$ respectively, we count copies of suitable $R$ in $\Gam'$. 

Given $\cA_Y$, since we have by assumption $(1-\eta)^\Delta \ge 1/2$, $\abs{V_z}\ge n/\log n$, and $q_Y(z)\ge p^d$, we can apply Claim~\ref{claim:Nconc} for every collection of $U_z$ such that $U_z\subset V_z$ is of size exactly $n_1$. 
There are $\prod_{z\in Z}\binom{\abs{V_z}}{n_1} = e^{O(n)}$ choices of collection, hence by Claim~\ref{claim:Nconc} and a union bound over collections, conditioned on $\cA_Y$ for all $z\in Z$, with probability at least $1-\exp\big(-O(n^{1+\eps})\big)$, the $N(\phi,R,U_Z)$ counts are close to their expectation for all such $U_Z$. 
In particular, the $N(\phi,R,U_Z)$ are `correct' for the collections where every $u\in U_z$ receives weight $1$ as a vertex in $\Gam'_{\phi'}$.
The count $N(\phi,R,U_Z)$ deals with edges of $\Rphi$ of size at least $2$, hence by the above argument and averaging over sets $U_z$ of vertices that receive weight $1$ in $\Gam'_{\phi'}$, we obtain that with high probability the count $\Gam'_{\phi'}(R)$ is close to its expectation. 

More precisely, by a union bound over the constant number of complexes $R$ to consider, and by averaging over the choice of collection $\{U_z\}_{z\in Z}$, we have, conditioned on $\cA_Y$, with probability at least $1-\exp\big(-O(n^{1+\eps})\big)$,
\begin{equation}\label{eq:Gam'Wconc}
\Gam'_{\phi'}(R) = \big(1\pm\tfrac{1}{\log n}\big)\Big(\prod_{e\in \Rphi^{(\ge2)}}q(e)\Big)\prod_{z\in Z}\vnorm{V_z'}{\Gam'_{\phi'}}\,,
\end{equation}
for all suitable $k$-complexes $R$.

Let $x$ be the next vertex to embed. 
To prove the lemma it now suffices to show that there is a subset $\tilde{V_x'}\subset V_x'$ with $\vnorm{\tilde{V_x'}}{\Gam'_{\phi'}}\ge (1-\eta)\vnorm{V_x'}{\Gam'_{\phi'}}$ such that $\cA_{Y\cup\{x\}}$ holds. 
Then the above argument after $x$ has been embedded, and a union bound over the number of vertices to embed (at most $n$) gives the result.

Suppose that we embed $x$ to $w\in V_x'$. 
Since $H$ has maximum degree $\Delta$, there are at most $\Delta$ vertices $z\in \rho(Z)$ with $\ind{\Gam'_{\phi'\cup\{x\mapsto w\}}}{V_z'}\ne \ind{\Gam'_{\phi'}}{V_z'}$. 
Let $Z'$ be the set of these vertices. 
The counts~\eqref{eq:Gam'Wconc} imply that for $z\in Z'$,
\begin{align}
\abs{V_x'}^{-1}\sum_{u\in V_x}\gam'_{\phi'}(u) &= \big(1\pm\tfrac{1}{\log n}\big)\vnorm{V_x'}{\Gam'_{\phi'}}\,,\\
\abs{V_x'}^{-1}\abs{V_z'}^{-1}\sum_{uv\in V_{xz}'}\gam'_{\phi'}(u)\gam'_{\phi'}(v)\gam'_{\phi'}(u,v) &= \big(1\pm\tfrac{1}{\log n}\big)q_Y(x,z)\vnorm{V_z}{\Gam'_{\phi'}}\cdot\vnorm{V_x'}{\Gam'_{\phi'}}\,,\\
\abs{V_x'}^{-1}\sum_{u\in V_x'}\gam'_{\phi'}(u)\Big(\abs{V_z'}^{-1}\sum_{v\in V_z'}\gam'_{\phi'}(v)\gam'_{\phi'}(u,v)\Big)^2 &= \big(1\pm\tfrac{1}{\log n}\big)q_Y(x,z)^2\vnorm{V_z'}{\Gam'_{\phi'}}^2\cdot\vnorm{V_x'}{\Gam'_{\phi'}}\,
\end{align}
hence we may apply Corollary~\ref{cor:ECSconc} with $\eps\subref{cor:ECSconc}=4/\log n$ and $d\subref{cor:ECSconc}=q_Y(x,z)\vnorm{V_z}{\Gam'_{\phi'}}$ to obtain the following.

For each $z\in Z'$ there is a set $B_z\subset V_x'$ with $\vnorm{B_z}{\Gam'_{\phi'}} \le 8(\log n)^{-1/4}\vnorm{V_z'}{\Gam'_{\phi'}}$ such that for all $w\in V_x'\setminus B_z$, if $x$ is embedded to $w$ we have 
\[
\vnorm{V_z'}{\Gam'_{\phi'\cup\{x\mapsto w\}}} = \Big(1\pm\frac{4}{(\log n)^{1/4}}\Big)q_Y(x,z)\vnorm{V_z'}{\Gam'_{\phi'}}\,.
\]
Set $\tilde V_x'=V_x'\setminus \bigcup_{z\in Z'}B_z$, so that
\[
\vnorm{\tilde{V_x'}}{\Gam'_{\phi'}} \ge \Big(1-\frac{8\Delta}{(\log n)^{1/4}}\Big)\vnorm{V_x'}{\Gam'_W}\,,
\]
which is at least $(1-\eta)\vnorm{V_x'}{\Gam'_{\phi'}}$ for large enough $n$.
Then given $\cA_Y$ and the counts~\eqref{eq:Gam'Wconc}, we have a small `bad set' which, if avoided when embedding $x$, implies $\cA_{Y\cup\{x\}}$ holds deterministically. 
So we can maintain well-behaved vertex weights throughout the embedding, and we may repeat the probabilistic argument above to control the counting properties~\eqref{eq:Gam'Wconc} after each vertex is embedded. 
There are at most $n$ embeddings, and hence with probability at least $1-\exp\big(-O(n^{1+\eps})\big)$ the partition $\{V_j\}_{j\in J}$ yields a $\Gam'$ with the required properties. 
To complete the proof we take a union bound over the $e^{O(n)}$ possible partitions.
\end{proof}

\section{A sparse hypergraph regularity lemma}\label{sec:reg}

There are several approaches to generalising Szemer\'edi's regularity lemma to hypergraphs (e.g.\ \cite{RSreg,Ghypergraph}). 
Recall that the main idea is to partition a hypergraph into a bounded number pieces, almost all of which are regular. 
Difficulties arise in giving a precise formulation of regularity that is both weak enough to be found by a regularity lemma and strong enough to support a counting lemma. 
We use a notion of octahedron minimality as our regularity condition (Definition~\ref{def:reg}), and in this section we describe how existing results imply that we can partition arbitrary hypergraphs into pieces which have the necessary structure. 

In dense hypergraphs, the combined use of (strong) regularity lemmas with compatible counting lemmas constitute the standard hypergraph regularity method~\cite{RNSSKregmethod,RSregmethod}. 
Our Theorem~\ref{thm:counting} is essentially a version of the counting lemma of~\cite{NRScount} for use with our definition of regularity. 
In the following subsection we show how to derive the setup of Theorem~\ref{thm:counting} from the regularity lemma of~\cite{RSreg}, allowing our Theorem~\ref{thm:counting} to be a drop-in replacement in many applications of the standard hypergraph regularity method. 

Versions of these tools for sparse graphs are less well-developed, but notably the weak regularity lemma and accompanying counting lemma of Conlon, Fox, and Zhao~\cite{CFZrelative} give a general technique for transferring results for dense hypergraphs to a sparse setting. 
We show how combined use of the regularity methods of~\cite{CFZrelative} and~\cite{RSreg,NRScount} can yield the setup of Theorems~\ref{thm:counting} and~\ref{thm:embedding}. In particular, we will state a sparse hypergraph regularity lemma, namely a sparse version of the R\"odl-Schacht regularity lemma~\cite{RSreg2}. We derive this from the dense version using the Conlon-Fox-Zhao weak regularity lemma.

We should point out that the proof strategy works with only trivial changes to obtain sparse versions of the R\"odl--Skokan regularity lemma~\cite{RSreg}, and the regular slice lemma~\cite[Lemma~10, parts (a) and (b)]{ABCM}\footnote{The dense version has a part (c), but this part is false in the sparse setting.}. We should point out that it has been well known in the area for some years that these regularity lemmas hold, but to the best of our knowledge no-one actually wrote them down with proofs.

\subsection{Sparse hypergraphs after Conlon, Fox, and Zhao}
We say, following~\cite{CFZrelative} (but with slightly different notation) that two weighted $k$-uniform hypergraphs $G$ and $H$ on a vertex set $V$ are a \emph{$\gamma$-discrepancy pair} in the following situation. For any $(k-1)$-uniform unweighted graphs $F_1,\dots,F_k$ on $V$, let $S$ be the collection of $k$-sets in $V$ whose $(k-1)$-sets can be labelled using each label $1$,\dots, $k$ exactly once, such that the label $i$ subset is in $F_i$. We say the edges of $S$ are \emph{rainbow for $F_1,\dots,F_k$}. If for any choice of the $F_i$, we have
\[\Big|\sum_{e\in S}\big(g(e)-h(e)\big)\Big|\le\gamma|V|^k\]
then $(G,G')$ is a $\gamma$-discrepancy pair. Note that this concept is interesting only when $\sum_{e\in\binom{V}{k}}g(e)$ is much larger than $\gamma |V|^k$. Since we want to work with sparse hypergraphs, in order to talk about discrepancy pairs (and in general to apply the machinery of~\cite{CFZrelative}) we will need to scale our weight functions in order that the majorising hypergraph has density about $1$. Going with this, we say a $k$-uniform hypergraph $G$ on $V$ is \emph{upper $\eta$-regular} if for any $(k-1)$-uniform unweighted graphs $F_1,\dots,F_k$ on $V$, letting $S$ be the set of rainbow edges we have
\[\sum_{e\in S}(g(e)-1)\le\eta|V|^k\,.\]

Conlon, Fox, and Zhao~\cite[Lemma~2.15]{CFZrelative} proved that if $\Gam$ satisfies the `linear forms condition' then it (and trivially all its subgraphs) are upper $o(1)$-regular. More concretely, they proved the following.
\begin{lemma}\label{lem:CFZupper}
 Given $\eta>0$ and $k$, there exists $\eta'>0$ such that the following holds. Suppose that $\Gam$ is a $k$-uniform $n$-vertex weighted hypergraph, and for each unweighted $k$-uniform hypergraph $H$ on at most $2k$ vertices we have
$G(H)=1\pm\eta'$.
 Then $\Gam$, and all its subgraphs, are upper $\eta$-regular.\hfill\qed
\end{lemma}

Conlon, Fox, and Zhao also proved the following sparse `weak regularity lemma' which applies to upper-regular graphs. We should point out that it is well known that from such a regularity lemma one can fairly easily, by iteration, prove a sparse strong regularity lemma, so in some sense this lemma does all the work. We state a somewhat weaker version (the original allows for directed hypergraphs and parts of different sizes, and gives a bound on the `complexity' of $\tilde{G}$ as one would need for iteration to a strong regularity lemma). 
\begin{theorem}[{\cite[Theorem~2.16]{CFZrelative}}]\label{thm:CFZweak}
 For any $\gamma>0$ and $k$-uniform weighted hypergraph $G$ on $V$ which is upper $\eta$-regular with $\eta\le2^{-80k/\gamma^2}$, there exists a $k$-uniform weighted hypergraph $\tilde{G}$ on $V$, such that $0\le\tilde{g}(e)\le 1$ for each $e\in\binom{[n]}{k}$ and such that $G$ and $\tilde{G}$ form a $\gamma$-discrepancy pair.
\end{theorem}

Finally, Conlon, Fox, and Zhao proved a counting lemma going with this concept of regularity. We will need only the special case of their result counting octahedra, which is the following.
\begin{theorem}[{\cite[Theorem~2.17]{CFZrelative}}]\label{thm:CFZcount}
 For every $\delta>0$ and $k$ there exist $\eps>0$ and $\eta>0$ such that the following holds. Suppose $\Gam$ is a $k$-uniform $n$-vertex weighted hypergraph, and for each unweighted $k$-uniform hypergraph $H$ on at most $4k$ vertices we have
$\Gam(H)=1\pm\eta$.
 Suppose that $G$ is a subgraph of $\Gam$, and $\tilde{G}$ is a $k$-uniform weighted hypergraph on $n$ such that $0\le\tilde{g}(e)\le 1$ for each $e\in\binom{[n]}{k}$, and suppose that $G$ and $\tilde{G}$ form an $\eps$-discrepancy pair. Then we have
 \[G\big(\oct{k}{\vec{2}^k}\big)=\tilde{G}\big(\oct{k}{\vec{2}^k}\big)\pm\delta\,.\]
\end{theorem}

We should note that the requirement of~\cite{CFZrelative} for a large number of vertices disappears in our setting since we made constant choices explicit rather than using $o(1)$ notation.

\subsection{Hypergraph regularity lemmas}

In this subsection we will derive sparse hypergraph regularity lemmas in which the concept of regularity is octahedron-minimality from existing regularity lemmas in which the concept of regularity is different. For clarity, in this section we use the word \emph{oct-regular} for the concept we earlier defined as simply \emph{regular}, and we refer to the different concept of the original forms as \emph{disc-regular}. We give only the bare bones definitions needed to state our lemmas; for more intuition or for notation which is required to actually work with the resulting regular partitions, the reader should consult~\cite{RSreg} or~\cite{RSreg2}.

Given $k$ and a vertex set $V$ with a \emph{ground partition} $\mathcal{P}$, we say the parts of $\mathcal{P}$ are \emph{$1$-cells}. A \emph{$(k-1)$-family of partitions} $\mathcal{P}^*$ on $V$ consists of $\mathcal{P}$ of $V$, together with, for each $2\le j\le k-1$ and each $j$-set $J$ of $1$-cells, a \emph{supporting} partition of the $j$-sets of $V$ with one vertex in each member of $J$ into \emph{$j$-cells}. We say a partition is supporting if, for each $j$-cell of $\mathcal{P}^*$, there are $j$ $(j-1)$-cells such that each edge of the given $j$-cell is rainbow for the chosen $(j-1)$-cells.

We will talk about a \emph{$j$-polyad} in such a partition, by which we mean a choice of $j$ $1$-cells, $\binom{j}{2}$ $2$-cells, and so on up to $\binom{j}{j-1}$ $(j-1)$-cells, which are supporting in the above sense.

Given a $j$-polyad $\mathcal{J}$ and a weighted $j$-uniform hypergraph $G$, we say $G$ is \emph{$(\eps,d,r)$-disc-regular} with respect to $\mathcal{J}$ if there is some $d$ such that the following holds. Let $\mathbf{P}$ denote the $j$-sets supported by the union of the $(j-1)$-cells of $j$. Let $H_1,\dots,H_r$ be any unweighted sub-$(j-1)$-graphs of the union of the $(j-1)$-cells in $\mathcal{J}$. Let $Q_i$ denote the $j$-sets supported by $H_i$, for each $1\le i\le r$, and let $\mathbf{Q}$ be the union of the $Q_i$. Suppose that $|\mathbf{Q}|\ge\eps|\mathbf{P}|$. Then we have
\[\sum_{e\in\binom{V(G)}{j}}g(e)\mathbf{q}(e)=(d\pm\eps)|\mathbf{Q}|\,.\]
If there exists $d$ such that $G$ is $(\eps,d,r)$-disc-regular with respect to $\mathcal{J}$, we say $G$ is $(\eps,r)$-disc-regular with respect to $\mathcal{J}$.

We say a $k$-uniform hypergraph $G$ is \emph{$(\eps,r)$-disc-regular} with respect to a $(k-1)$-family of partitions $\mathcal{P}^*$ if the following holds. Choose uniformly at random a $k$-set of vertices intersecting each part of the ground partition in at most one vertex, and choose the $k$-polyad containing this set. Then with probability at least $1-\eps$, $G$ is $(\eps,r)$-disc-regular with respect to the chosen polyad. We say similarly that $G$ is \emph{$\eps$-oct-regular} with respect to $\mathcal{P}^*$ if the same holds replacing $(\eps,r)$-disc-regularity with $\eps$-oct-regularity.

Finally, we say a $(k-1)$-family of partitions is \emph{$(t_0,t_1,\eps)$-oct-equitable} if the following hold. There are between $t_0$ and $t_1$ parts of the ground partition, and these parts differ in size by at most one. There are numbers $d_2,\dots,d_{k-1}$ such that $1/d_i$ is an integer at most $t_1$ for each $2\le i\le k-1$, and for each $2\le j\le k-1$, each $j$-polyad in $\mathcal{P}^*$, and each $j$-cell supported by that polyad, the $j$-cell is $d_j$-oct-regular with respect to the polyad. We say $\mathcal{P}^*$ is \emph{$(t_0,t_1,\eps)$-disc-equitable} if we replace $\eps$-oct-regularity with $(\eps,1)$-disc-regularity, and impose the stronger condition that every part of the ground partition has exactly the same size.

Probably the most commonly used form of hypergraph regularity is the following, due to R\"odl and Schacht.

\begin{lemma}[{\cite[Lemma~23]{RSreg2}}]\label{lem:RSchRL}
Let $k \geq 3$ be a fixed integer. For all positive integers $q$, $t_0$ and $s$,
positive~$\eps_k$ and functions $r: \NATS \rightarrow \NATS$ 
and $\eps: \NATS \rightarrow (0,1]$, there exist integers~$t_1$ 
and~$n_0$ such that the following holds for all $n \ge n_0$ which are divisible 
by~$t_1!$. Let $V$ be a vertex set of size $n$, and suppose that~$G_1,
\dots, G_s$ are edge-disjoint $k$-uniform hypergraphs on $V$, and that $\mathcal{Q}$ is a partition
of $V$ into at most $q$ parts of equal size. Then there exists a
$(k-1)$-family of partitions $\mathcal{P}^*$ on~$V$ such that
\begin{enumerate}[label=\abc]
\item the ground partition of $\mathcal{P}^*$ refines $\mathcal{Q}$,
\item $\mathcal{P}^*$ is $(t_0,t_1, \eps(t_1))$-disc-equitable, and 
\item for each $1 \leq i \leq s$, $G_i$ is $(\eps_k,r(t_1))$-disc-regular with
respect to~$\mathcal{P}^*$.
\end{enumerate}
\end{lemma}

It is often convenient to have the initial partition $\mathcal{Q}$ of the vertex set which is refined. In fact R\"odl and Schacht even allow for an initial family of partitions which is refined (in an appropriate sense), and we could similarly allow for such an initial family of partitions in our forthcoming sparse version Lemma~\ref{lem:RSchreglem}; but for clarity we prefer this version.

A counting lemma going with this, taken from~\cite[Lemma~27]{ABCM} (but which is derived from~\cite{RSchCount} via~\cite{CFKO}), is the following (which we state specifically for counting octahedra in one polyad; this follows from the version of~\cite[Lemma~27]{ABCM} by the standard construction).

\begin{lemma}\label{lem:RegCount}
  Let $k,s,r,m_0$ be positive integers, and let
  $d_2,\ldots,d_{k-1}$, $\eps$, $\eps_k$, $\beta$ be positive constants such that $1/d_i
  \in\mathbb{N}$ for any $2 \leq i \leq k-1$ and
  \[\frac{1}{m_0}\ll \frac{1}{r}, \eps \ll
  \eps_k,d_2,\ldots,d_{k-1}\quad\text{and}\quad \eps_k \ll
  \beta\,.\]
  Then the following holds for all
  integers $m\ge m_0$. 
Let $\mathcal{J}$ be a $k$-polyad with $k$
clusters $V_1, \dots, V_k$ each of size $m$, and suppose that for each $2\le j\le k-1$, each $j$-cell of $\mathcal{J}$ is $(\eps,d_j,1)$-disc-regular with respect to its supporting $(j-1)$-polyad. Let $G$ be an unweighted $k$-uniform hypergraph on $\bigcup_{i \in [k]} V_i$ which is supported
on $\mathcal{J}$ and is $(\eps_k,d,r)$-disc-regular with respect to $\mathcal{J}$, and write $\cG$ for the unweighted $k$-complex obtained from $\mathcal{J}$ by adding all edges of $G$.
Then we have
\[\cG\big(\oct{k}{\vec{2}^k}\big)=\left(d^{2^k} \pm \beta \right)\prod_{j
= 2}^{k-1} d_j^{2^j\binom{k}{j}}\,.\]
\end{lemma}

We now state our sparse regularity lemma for hypergraphs. For convenience of use, we remove the divisibility condition on $n$ (which is trivial); this motivates the slight difference between the definitions of disc-equitable and oct-equitable.

\begin{lemma}\label{lem:RSchreglem}
Let $k \geq 3$ be a fixed integer. For all positive integers $q$, $t_0$, and $s$,
positive~$\eps_k$ and functions $\eps: \NATS \rightarrow (0,1]$, there exist integers~$t_1$ 
and~$n_0$, and an $\eta^*>0$, such that the following holds for all $n \ge n_0$. Let $V$ be a vertex set of size $n$, let $\Gam$ be a $k$-uniform hypergraph on $[n]$, and suppose that~$G_1,
\dots, G_s$ are edge-disjoint $k$-uniform subgraphs of $\Gam$, and that $\mathcal{Q}$ is a partition
of $V$ into at most $q$ parts whose sizes differ by at most one. Suppose furthermore that there is some $p>0$ such that for any $k$-uniform unweighted hypergraph $H$ on at most $4k$ vertices we have $\Gam(H)=(1\pm \eta^*)p^{e(H)}$. Then there exists a
$(k-1)$-family of partitions $\mathcal{P}^*$ on~$V$ such that
\begin{enumerate}[label=\abc]
\item\label{RL:a} the ground partition of $\mathcal{P}^*$ refines $\mathcal{Q}$,
\item\label{RL:b} $\mathcal{P}^*$ is $(t_0,t_1, \eps(t_1))$-oct-equitable, and 
\item\label{RL:c} for each $1 \leq i \leq s$, $G_i$ is $\eps_k$-oct-regular with
respect to~$\mathcal{P}^*$.
\end{enumerate}
\end{lemma}

Lemma~\ref{lem:RegCount} shows that Lemma~\ref{lem:RSchRL} immediately implies the dense, unweighted case of this statement. Specifically, what we now prove is the case that $p=1$ and $\Gam$ is the complete unweighted $k$-uniform hypergraph on $[n]$, and each $G_i$ is an unweighted $k$-uniform hypergraph on $n$ (that is, its weight function has range $\{0,1\}$ ). We defer the general case to later.  We should note that, as proved in~\cite{DHNRcharacterising}, the dense unweighted case of Lemma~\ref{lem:RSchreglem} is formally weaker than Lemma~\ref{lem:RSchRL}.

\begin{proof}[Proof of Lemma~\ref{lem:RSchreglem}, dense unweighted case]
 Given $k\ge 3$, positive integers $q$, $t_0$ and $s$, positive $\eps_k$ and a function $\eps:\NATS\rightarrow(0,1]$, we let $\eps'_k<\eps_k$ be small enough for Lemma~\ref{lem:RegCount} with input $\beta=\tfrac12\eps_k$. We choose functions $r,\eps'$ of $t_1$ which tend to infinity and zero respectively fast enough for Lemma~\ref{lem:RegCount} to apply provided all densities $d_2,\dots,d_{k-1}$ are at least $1/t_1$. In addition we insist $\eps'(t_1)$ is small enough for application of Lemma~\ref{lem:RegCount} with input $\beta=\tfrac12\eps(t_1)$ to count octahedra of uniformity between $2$ and $k-1$ inclusive. Let $m_0$ be large enough for all these applications of Lemma~\ref{lem:RegCount}.
 
 Let $t_1$ and $n'_0$ be returned by Lemma~\ref{lem:RSchRL} for input $k,q,t_0,s,\eps'_k,r,\eps'$. If necessary, we increase $t_1$ such that $q$ divides $t_1!$. Let $n_0\ge n'_0$ be such that $n_0\ge m_0t_1$ is sufficiently large for the following calculations. Given any $n\ge n_0$, let $G_1,\dots,G_s$ be edge-disjoint $k$-uniform unweighted hypergraphs on $[n]$. We add a set $N$ of $\lceil\tfrac{n}{t_1!}\rceil t_1!-n$ new vertices to the vertex set of each $G_i$ (and no new edges) to obtain $k$-graphs $G'_1,\dots,G'_s$ on $n'$ vertices, where by construction $n'$ is divisible by $t_1!$. We extend $\mathcal{Q}$ to a partition $\mathcal{Q}'$ on $[n']$ by adding the new vertices to the parts of $\mathcal{Q}$ such that the part sizes of $\mathcal{Q}'$ are equal (which is possible since $q$ divides $t_1!$ divides $n'$). We now apply Lemma~\ref{lem:RSchRL}, with the given inputs, to the $G'_1,\dots,G'_s$. The result is a family of partitions $\mathcal{R}^*$ which satisfies the conclusion of Lemma~\ref{lem:RSchRL}. Removing the added vertices $N$, we obtain a family of partitions $\mathcal{P}^*$, which we claim has the desired properties.
 
 The property~\ref{RL:a} follows from the construction and the fact that the ground partition of $\mathcal{R}^*$ refines $\mathcal{Q}'$. The partition $\mathcal{R}^*$ is $(t_0,t_1,\eps'(t_1))$-disc-equitable, so by Lemma~\ref{lem:RegCount} it follows that it is also $\mathcal{R}^*$ is $(t_0,t_1,\tfrac12\eps(t_1))$-oct-equitable, and whenever some $G'_i$ is $(\eps'_k,r)$-disc-regular with respect to a $k$-polyad of $\mathcal{R}^*$, it is also $\tfrac12\eps_k$-oct-regular with respect to that polyad. Removing the set $N$ of at most $t_1!$ extra vertices reduces the size of any part by at most $t_1!$, and hence changes the number of $j$-edges using that part by at most $t_1!n^{j-1}$, and the number of copies of $\oct{j}{\vec{2}^j}$ using that part by at most $(t_1!)^2n^{2j-2}$. By choice of $n_0$ and by Lemma~\ref{lem:RegCount}, these numbers are tiny compared to respectively the number of edges and octahedra supported in any cell or polyad of $\mathcal{R}^*$. Consequently $\mathcal{P}^*$ is $\eps_k$-oct-equitable, giving~\ref{RL:b}, and whenever $G'_i$ is $(\eps'_k,r)$-disc-regular with respect to some polyad of $\mathcal{R}^*$ also $G_i$ is $\eps_k$-oct-regular with respect to the corresponding polyad of $\mathcal{P}^*$, giving~\ref{RL:c}.
\end{proof}

We next use Theorems~\ref{thm:CFZweak} and~\ref{thm:CFZcount} to derive the general case of Lemma~\ref{lem:RSchreglem} from the dense unweighted case proved above. We should note that, although this proof involves two applications of a regularity lemma, the bounds on constants we obtain are essentially the same as in Lemma~\ref{lem:RSchRL}. The application of weak regularity causes all these bounds to increase by less than a double exponential; this is inconsequential given the bounds of Lemma~\ref{lem:RSchRL} cannot (even for $2$-graphs) be better than tower-type.

\begin{proof}[Proof of Lemma~\ref{lem:RSchreglem}, general case]
Given $k$, $q$, $t_0$, $s$ and $\eps_k$, and a function $\eps:\mathbb{N}\to(0,1]$, let $t_1$ and $n_0$ be returned by the dense unweighted case of Lemma~\ref{lem:RSchreglem} for input as above but with $\tfrac1{2s}\eps_k$ replacing $\eps_k$.

We choose $\gamma$ such that $2k!\gamma$ is small enough for Theorem~\ref{thm:CFZcount} with input $\delta=\tfrac1{4s}t_1^{-2^k}\eps_k$. We let $\eta>0$ be small enough for Theorem~\ref{thm:CFZweak} with input $\gamma$. We let $\eta'>0$ be small enough for Lemma~\ref{lem:CFZupper} with input $\eta$, and we let $\eta^*$ be small enough for the applications of Theorem~\ref{thm:CFZcount} and Lemma~\ref{lem:CFZupper}.
Given an initial partition $\mathcal{Q}$ and $k$-uniform hypergraphs $\Gam$, $G_1,\dots,G_s$ satisfying the conditions of the lemma, we proceed as follows.
 
 In order to apply the machinery of Conlon, Fox, and Zhao~\cite{CFZrelative} we need to scale the weight function of our majorising $k$-uniform hypergraph $\Gam$ by $p^{-1}$. We write $p^{-1}\Gam$ for the $k$-uniform hypergraph we obtain by this scaling, and similarly for the $G_i$. Given an unweighted $k$-uniform hypergraph $H$, we (slightly abusing notation) think of $\Gam$ and $H$ as complexes where all edges of size less than $k$ are present with weight $1$, and write $\Gam(H)$ for the corresponding homomorphism density. Our counting condition states that for any $H$ with at most $4k$ vertices, we have $\Gam(H)=(1\pm\eta')p^{e(H)}$, and thus $(p^{-1}\Gam)(H)=1\pm\eta'$. By Lemma~\ref{lem:CFZupper} it follows that $p^{-1}\Gam$, and its subgraphs $p^{-1}G_i$, are upper $\eta$-regular. Applying Theorem~\ref{thm:CFZweak}, with input $\gamma$ separately to each $p^{-1}G_i$, we obtain weighted $k$-graphs $G'_i$ on $[n]$, whose weights are in $[0,1]$, such that $(p^{-1}G_i,G'_i)$ is a $\gamma$-discrepancy pair for each $i$. It follows that $(p^{-1}s^{-1}G_i,s^{-1}G'_i)$ is also a $\gamma$-discrepancy pair for each $i$.
 
 We now create unweighted $k$-graphs $G''_i$ by, for each $e\in\binom{[n]}{k}$ independently, choosing to put $e$ into either exactly one of the $G''_i$, or into none of them, choosing to put $e$ in $G''_i$ with probability $s^{-1}g'_i(e)$. Since $0\le g'_i(e)\le 1$ for each $i$, we have $0\le \sum_{i\in[s]}g'_i(e)\le s$, so that the distribution we just described is as required a probability distribution.
 
 We claim that with high probability $(G''_i,G'_i)$ is a $\gamma$-discrepancy pair for each $i$. Indeed, suppose $i$ and unweighted $(k-1)$-graphs $F_1,\dots,F_k$ on $[n]$ are fixed before the sampling of the $G''_i$. The expected number of edges of $G''_i$ which are rainbow for $F_1,\dots,F_k$ is exactly equal to the sum of $g'_i(e)$ over $e$ rainbow for $F_1,\dots,F_k$. By the Chernoff bound, the probability of an additive error of $\gamma n^k$ is $o(2^{-kn^{k-1}})$. In other words, a given $F_1,\dots,F_k$ and $i$ witness the failure of our claim with probability $o(2^{-kn^{k-1}})$. Taking the union bound over the at most $s2^{kn^{k-1}}$ choices of $F_1,\dots,F_k$ and $i$, our claim fails with probability $o(1)$ as desired.
 
 Putting this together, we see that $(p^{-1}s^{-1}G_i,G''_i)$ is a $2\gamma$-discrepancy pair for each $i$. We now apply the dense unweighted case of Lemma~\ref{lem:RSchreglem}, with input $\tfrac{1}{2s}\eps_k$ replacing $\eps_k$ but otherwise with the same inputs as given to the general case, to the graphs $G''_1,\dots,G''_s$, obtaining a family $\mathcal{P}^*$ of partitions. We claim that this is the desired family of partitions, for which we only need check condition~\ref{RL:c} holds for the $G_i$ as well as the $G''_i$.
 
 To that end, suppose we have a polyad $\mathcal{J}$ of $\mathcal{P}^*$ and an $i$ such that $G''_i$ is $\tfrac{1}{2s}\eps_k$-oct-regular with respect to $\mathcal{J}$. It is enough to show that $G_i$ is also $\eps_k$-oct-regular with respect to $\mathcal{J}$. We obtain graphs $H$ and $\tilde{H}$ by deleting (i.e.\ setting to weight zero) all edges of respectively $G_i$ and $G''_i$ which are not supported by $\mathcal{J}$. Trivially $H$ is still a subgraph of $\Gam$, and we claim that $(p^{-1}s^{-1}H,\tilde{H})$ is a $2k!\gam$-discrepancy pair. Indeed, given any $(k-1)$-uniform unweighted hypergraphs $F_1,\dots,F_k$, we consider the intersections of each $F_j$ with the $(k-1)$-cells of $\mathcal{J}$. A $k$-set which is rainbow for the $F_i$ can only have non-zero weight in either $H$ or $\tilde{H}$ if its $(k-1)$-subsets are also contained in, and rainbow for, the $(k-1)$-cells of $\mathcal{J}$, so it suffices to look at the $k!$ different rainbow intersections of the $F_i$ with the $(k-1)$-cells of $\mathcal{J}$. For each such intersection, the weights of rainbow $k$-sets in $H$ and in $\tilde{H}$ are unchanged from those in $G_i$ and $G''_i$. Since the latter form a $2\gamma$-discrepancy pair, the contribution to the discrepancy of $p^{-1}s^{-1}H$ and $\tilde{H}$ by any given rainbow intersection is at most $2\gamma$, as required.
 
 Applying Theorem~\ref{thm:CFZcount}, we have
 \[(p^{-1}s^{-1}H)(\oct{k}{\vec{2}^k})=\tilde{H}(\oct{k}{\vec{2}^k})\pm\delta \,,\]
 and by choice of $\delta$ we conclude that $G_i$ is $\eps_k$-oct-regular with respect to $\mathcal{J}$, as desired.
\end{proof}

\subsection{Relating GPEs and regularity}

As usual, when one has a family of partitions and wishes to embed a complex $H$ into it, it is necessary to choose the cells of the family of partitions to which we embed edges of $H$ of all sizes, and (if the embedding is to be done by regularity) we need that the $k$-polyads to which we want to embed $k$-edges are regular (and relatively dense if we want to obtain an embedding with large weight). Making this choice can be quite technically difficult; the Regular Slice Lemma of~\cite{ABCM}, which as mentioned has a sparse counterpart in our setting, can help.

Once the choice is made, Theorems~\ref{thm:counting} and~\ref{thm:embedding} give respectively a two-sided counting lemma for small hypergraphs, and a one-sided counting lemma for potentially large hypergraphs. We now explain how to prove these two theorems. What we need to do is explain how to construct a stack of candidate graphs such that the trivial partial embedding, in which no vertices are embedded, is a good partial embedding. One should usually think of $\Gam$ as being a random graph, so that its density graph has all edges of uniformity less than $k$ of weight $1$, and all edges of uniformity $k$ of weight $p$ for some $p\in(0,1]$, and $\cG$ as being the subgraph consisting of the chosen cells of the family of partitions into which we want to embed, together with the supported $k$-edges in the resulting $k$-polyads (though we will not actually use this idea, and the following lemma is true in more generality).

\begin{lemma}\label{lem:getGPE}
 For all $k\ge 2$, finite sets $J$, and $J$-partite $k$-complexes $H$, given parameters $\eta_k$, $\eta_0$ and $\eps_\ell$, $d_\ell$ for $1\le\ell\le k$ such that  $0<\eta_0\ll d_1,\dotsc,d_k,\,\eta_k$, and for all $\ell$ we have $0<\eps_{\ell}\ll d_{\ell},\dotsc,d_k,\,\eta_k$,
the following holds.

Let $c^*=\max\{2v(H)-1, 4k^2+k\}$. 
Suppose we are given any $J$-partite weighted $k$-graphs $\cG\subset\Gam$ and density graphs $\cD$, $\cP$, where $\Gam$ is an $(\eta_0,c^*)$-THC graph for $H$ with density graph $\cP$, and where for each $e\subset J$ of size $1\le\ell\le k$, the graph $\ind{\cG}{V_e}$ is $\eps_\ell$-regular with relative density $d(e)\ge d_\ell$ with respect to the graph obtained from $\ind{\cG}{V_e}$ by replacing layer $\ell$ with $\Gam$. We define a stack of candidate graphs $\cC^{(0)},\dots,\cC^{(k)}$ as follows. To begin with, we apply the standard construction to $\Gam$ and $\cP$ in order to obtain a $v(H)$-partite graph $\cC^{(0)}$ with density graph $\cD^{(0)}$, and to $\cG$ to obtain a $v(H)$-partite graph $\cC^{(k)}$. We then let $\cC^{(\ell)}$ consist of all edges of $\cC^{(k)}$ of uniformity less than or equal to $\ell$, and all edges of $\cC^{(k)}$ of uniformity greater than $\ell$. For each $1\le\ell\le k$, we let $\cD^{(\ell)}$ have weight one on all edges of uniformity not equal to $\ell$, and weight equal to that of $\cD$ on edges of uniformity $\ell$. Then the trivial partial embedding of $H$ is a $k$-GPE.
\end{lemma}
\begin{proof}
 We set $h^*=k(4k+1)+\vdeg(H)$. Recall (after Definition~\ref{def:ensemble} that we can construct a valid ensemble of parameters by choosing them in the order given there, using minimum relative densities $\delta_1,\dots,\delta_k$ calculated to satisfy~\ref{gpe:dens} (using the supplied $d_1,\dots,d_k$), and the counting accuracy $\eta_k$; and we can further place, if necessary, upper bounds on the $\eta_\ell$ for $1\le\ell\le k-1$ in terms of the parameters for larger $\ell$. Suppose that such a valid ensemble of parameters is given, with $\eps_\ell$ being the best-case regularity at each level $\ell$.
 
 By assumption on $\Gam$, we have~\ref{gpe:c0}, and by assumption~\ref{gpe:dens}. It remains to verify~\ref{gpe:reg}. For a given $1\le\ell\le k$, it is trivial to verify the regularity statement for an edge $e$ of uniformity less than $k$: either $\cC^{(\ell)}(e)$ is equal to $\cC^{(\ell-1)}(e)$ (in which case regularity is automatic) or $|e|=\ell$ in which case by assumption $\cC^{(\ell)}(e)$ is a sufficiently regular subgraph of $\cC^{(\ell-1)}(e)$.
 
 It remains to verify~\ref{gpe:reg} for edges of size $k$. This is not trivially true; we need to show that a subgraph of $\Gam$ is regular with respect to a certain regular subgraph. However it follows immediately from Lemma~\ref{lem:slicing}.
\end{proof}

Given Lemma~\ref{lem:getGPE}, Theorems~\ref{thm:counting} and~\ref{thm:embedding} are immediate corollaries of Lemmas~\ref{lem:GPEcount} and~\ref{lem:GPEemb} respectively.

\section{Counting and embedding for GPEs}\label{sec:count}

In this section we prove Lemmas~\ref{lem:onestep}, \ref{lem:GPEcount}, and~\ref{lem:GPEemb}. 
As mentioned above, we prove the first two lemmas together, by induction on $\ell$ in each lemma. We begin by assuming Lemma~\ref{lem:GPEcount} for $\ell'<\ell$ in order to prove the bound on $B_\ell(x)$ claimed in Lemma~\ref{lem:onestep}. 
We will use Lemma~\ref{lem:GPEcount} to show that the various counting conditions for Lemma~\ref{lem:k-inherit} are met; the rest is simply bookkeeping.

\begin{proof}[Proof of Lemma~\ref{lem:onestep} for $\ell\ge 1$]
If a vertex $v\in V_x$ is in $B_\ell(x)\setminus B_{\ell-1}(x)$, then by definition there is a failure of regularity in the graph $\cC^{(\ell)}_{x\mapsto v}$ (obtained by applying the update rule to $\cC^{(\ell)}$ ). 
Specifically, there is some $e\subset H\setminus\big(\dom(\phi)\cup\{x\}\big)$ such that, although $\tcC^{(\ell)}(e)$ is $\big(\eps_{\ell,|e|,\pi_{\phi}(e)},d\big)$-regular (with $d$ as given in~\ref{gpe:reg}) with respect to $\ocC^{(\ell-1)}(e)$, the graph $\tcC^{(\ell)}_{x\mapsto v}(e)$ is not $\big(\eps_{\ell,|e|,\pi_{\phi\cup\{x\to v\}}(e)},d_x\big)$-regular (with $d_x$ as given in~\ref{gpe:reg}) with respect to $\ocC^{(\ell-1)}_{x\mapsto v}(e)$.

First, observe that if $\pi_{\phi}(e)=\pi_{\phi\cup\{x\mapsto v\}}(e)$, then this failure of regularity is impossible: 
we have $\tcC^{(\ell)}(e)=\tcC^{(\ell)}_{x\mapsto v}(e)$ and $\ocC^{(\ell-1)}(e)=\ocC^{(\ell-1)}_{x\mapsto v}(e)$. 
Thus there is an edge of $H$ which both contains $x$ and some non-empty subset of $e$; since there are at most $\Delta$ edges of $F$ containing $x$, each of whose at most $k-1$ other vertices are in at most $\Delta-1$ different edges of $H$, there are in total at most $\Delta+(k-1)\Delta(\Delta-1)\le k\Delta^2$ choices of $e$.

Thus, in order to prove the $\ell$ case of Lemma~\ref{lem:onestep}, it suffices to show that for any given non-empty $e\subset H\setminus\big(\dom(\phi)\cup\{x\}\big)$, the total weight of vertices $v$ in $\cC^{(\ell)}(x)$ such that $\tcC^{(\ell)}_{x\mapsto v}(e)$ is not $\big(\eps_{\ell,|e|,\pi_{\phi}(e)+1},d_x\big)$-regular, with respect to $\ocC^{(\ell-1)}_{x\mapsto v}(e)$, is at most $\eps'_\ell\vnorm{V_x}{\cC^{(\ell-1)}(x)}$. The idea is that Lemma~\ref{lem:k-inherit} should provide this desired bound.

To that end, let $V=V_x\cup\bigcup_{y\in e}V_y$, let $\Gam:=\cC^{(\ell-1)}[V]$ with the inherited vertex partition, and let $\cG$ be obtained from $\Gam$ by replacing the edges in $V_e$ and $V_{\{x\}\cup e}$ with those from $\cC^{(\ell)}$. 
Now by~\ref{gpe:reg} the graphs $\cG\cap V_e$ and $\cG\cap V_{e\cup\{x\}}$ are both $\eps:=\eps_{\ell,|e|,\pi_\phi(e)}$-regular with densities $d,d'\ge \delta_\ell$ with respect to $\Gam$. 
With $\eps':=\eps_{\ell,|e|,\pi_\phi(e)+1}$, $d_x=dd'$, and the update rule; by definition failure of regularity in the sense of an $\ell$-GPE coincides with failure to inherit regularity in Lemma~\ref{lem:k-inherit}. 
The conditions~\ref{ve:reginh} and~\ref{ve:sizeorder} state that the constants above are compatible with Lemma~\ref{lem:k-inherit} in this case, and the conclusion of Lemma~\ref{lem:k-inherit} is the desired bound. 
It only remains to show that all the conditions of Lemma~\ref{lem:k-inherit} are met. By construction, we have~\ref{inh:GGam}, while~\ref{inh:regJ} and~\ref{inh:regf} are given by~\ref{gpe:reg}. Thus to complete the proof of Lemma~\ref{lem:onestep} we only need to show that the counting condition~\ref{inh:count} holds.

Write $s=\abs{e}$ and $e'=\{x\}\cup e$. Note that we have $0<s\le k$. 
We now justify that for any given $k$-complex $R$ of the form $\oct{s+1}{\vec a}$ or $+2\oct{s}{0,\vec b}$ with $\vec a\in\{0,1,2\}^{e'}$ and $\vec b\in\{0,1,2\}^e$, we can accurately count $R$ in $\Gam$. This verifies~\ref{inh:count}.

We separate two cases. First, if $\ell=1$ then $\Gam$ is an induced subgraph of $\cC^{(0)}$. By~\ref{gpe:c0}, $\cC^{(0)}$ is an $(\eta_0, c^*)$-THC graph, and thus by~\ref{thc:count}, $c^*\ge 4k+1$, and~\ref{ve:count}, we obtain the required count immediately.

The second, slightly more difficult case is $\ell>1$. Here we aim to deduce the required count of $R$ from the $\ell-1$ case of Lemma~\ref{lem:GPEcount} (which is valid by induction). We obtain a stack of candidate graphs by applying the standard construction (with the $e'$-partite $k$-complex $R$) to the graphs $\cC^{(i)}[V]$ for $i=0,\dots,k$. Now the required count follows immediately from Lemma~\ref{lem:GPEcount} and condition~\ref{ve:count} on $\eta_{\ell-1}$, provided that we can justify that the trivial partial embedding of $R$ (in which no vertices are embedded) together with this stack of candidate graphs forms an $(\ell-1)$-GPE. To do this we need to specify the valid ensemble of parameters we use. These are identical to the valid ensemble we are provided with, \emph{except} that we shift the indices for hits in the regularity parameters, that is, we use $\eps_{\ell',r,h}$ with $h_0\le h\le h^*$ where $h_0 = \max\{\pi_\phi(f): \emptyset\ne f\subset e'\}$. 
Recall that we have $h_0 \le \vdeg(H)$.

By construction the property~\ref{gpe:c0} for the trivial embedding of $R$ is implied by \ref{gpe:c0} for $\phi$ and $H$. 
For~\ref{gpe:reg}, we use the assumption that $\phi$ is an $(\ell-1)$-GPE, so for each edge $f$ and each $1\le\ell'\le k-1$, we have that $\tcC^{(\ell')}(f)$ is $\big(\eps_{\ell',|f|,\pi_{\phi}(f)},d_f\big)$-regular (with $d_f$ as given in~\ref{gpe:reg}) with respect to $\ocC^{(\ell'-1)}(f)$. Since $\pi_{\phi}(f)\le h_0$ we have $\eps_{\ell',|f|,\pi_{\phi}(f)}\le\eps_{\ell',|f|,h_0}$, and so indeed the trivial partial embedding of $R$ with the given stack of candidate graphs satisfies the conditions~\ref{gpe:reg} for $1\le\ell'\le\ell-1$ and the shifted regularity parameters. 

Finally we must verify that the shifted ensemble is valid and suitable for use in Lemma~\ref{lem:GPEcount}.
The `length' of the sequences of shifted regularity parameters is $h_0^*:=h^*-h_0\le h^*$, hence~\ref{ve:worst} and~\ref{ve:count} are implied by the same conditions for the unshifted ensemble. 
The property~\ref{ve:reginh} is unchanged by shifting, and~\ref{ve:sizeorder} holds because we have $\eps_{\ell,r+1,h_0^*}\le\eps_{\ell,r+1,h_0}\le\eps_{\ell,r,0}\le\eps_{\ell,r,h_0}$.
For counting $R$ with the height $\ell-1$ case of Lemma~\ref{lem:GPEcount} we need $c^*\ge\max\{8k+1, (\ell-1)(4k+1)\}$, 
\[
h_0^*\ge h^*-\Delta'\ge \ell(4k+1) \ge (\ell-1)(4k+1)+\vdeg(R)\,,
\]
and $(4k+1)\eta_\ell\le 1/2$, which hold for this case by the assumptions of Lemma~\ref{lem:onestep} because $\ell\ge 2$.
\end{proof}

The second part of the intertwined induction is a proof of Lemma~\ref{lem:GPEcount}. 
We first give a proof of Lemma~\ref{lem:GPEemb} which assumes Lemma~\ref{lem:onestep} (for $\ell\le k$), because it serves as a good introduction to aspects of the method without the notation necessary for the induction on $\ell$, or calculations involving bad vertices. 

\begin{proof}[Proof of Lemma~\ref{lem:GPEemb}]
We prove Lemma~\ref{lem:GPEemb} by induction on $r=v(H)-\abs{\dom\phi}$, assuming the $\ell\le k$ cases of Lemma~\ref{lem:onestep}.

The statement for $r=0$ is a tautology, since then $F-\dom\phi=\emptyset$ and the empty set appears identically on both sides of the required count.

For $r=1$, the statement follows directly from the definition of a GPE, without the need to apply Lemma~\ref{lem:onestep}. 
The empty set is dealt with explicitly, so here we consider consider the weights $\cC\tl(x)$ as functions on $V_x$, and by the \emph{density} of $\cC\tl(x)$ we mean $\vnorm{V_x}{\cC\tl(x)}$. 
Let $V(H)\setminus\dom\phi=\{x\}$, and note that by \ref{gpe:c0} we know that $\cC\tz(x)$ has density 
  \[
  \vnorm{V_x}{\cC\tz(x)}=(1\pm\eta_0)d\tz_\phi(x)\,,
  \]
   and by \ref{gpe:reg}, for each $1\le\ell'\le\ell$, the graph $\cC^{(\ell')}(x)$ is a subgraph (in the sense of a weighted $1$-graph) of $\cC^{(\ell'-1)}(x)$ of relative density
   \[
   d^{(\ell')}_\phi(x)\pm\eps_{\ell'}'\,.
   \] 
   Thus $\cC\tk$ has density
  \begin{align}\label{eq:GPEemb:cldens}
  (1\pm\eta_0)d\tz_\phi(x)\prod_{\ell\in[k]}\left(1\pm\frac{\eps'_{\ell}}{\delta_{\ell}}\right)d^{(\ell)}_\phi(x)
&= (1\pm\eta_k)\prod_{0\le\ell\le k} d^{(\ell)}_\phi(x)\,,
  \end{align}
  because we have a valid ensemble of parameters ensuring for $\ell\in[k]$ that
  $
  \eta_0,\, \eps_{\ell}' \ll \delta_{\ell},\, \eta_{k},\,k
  $
  by \ref{ve:worst}.  
  Multiplied by the weight $c\tk(\emptyset)$, this is the desired expression for $\cC\tl(F-\dom\phi)$ in the case $r=1$.

  For $r\ge2$, fix any $x\in V(H)\setminus\dom\phi$. 
  We will consider embedding $x$ to some $v\in V_x$ and use induction on $r$ to count the contribution from good choices of $v$. 
  The key observation is that the update rule implies
  \[
  \cC\tl(H-\dom\phi) = \Ex[\big]{\cC\tl_{x\mapsto v}\big(H-\dom\phi-\{x\}\big)}\,,
  \]
  where the expectation is over a uniformly random choice of $v\in V_x$.
  We separate three types of density term in the desired counting statement: $d\tl(\emptyset)$ terms, $d\tl(x)$ terms, and the remaining terms for which we write
\begin{equation}\label{eq:xi}
\xi(\ell)
  :=\frac{\cD\tl_{\phi}\big(H-\dom\phi\big)}{d\tl_{\phi}(\emptyset)d\tl_{\phi}(x)}
=\frac{\cD\tl_{\phi\cup\{x\mapsto v\}}\big(H-\dom\phi-x\big)}{d\tl_{\phi\cup\{x\mapsto v\}}(\emptyset)}\,,
\end{equation} 
where the second expression for $\xi(\ell)$ comes from the update rule. Note that despite the appearance of $v$ in the notation on the right-hand side, as an expected density $\xi$ does not depend on the choice of $v$.
The weight of the empty set is dealt with explicitly, analysing the choice of $v\in V_x$ gives the $d\tl_\phi(x)$ terms, and the $\xi(\ell)$ terms are found by induction on $r$.
  We can afford to ignore bad vertices for a lower bound, we merely need to estimate $\vnorm{V_x\setminus B_k(x)}{\cC\tk(x)}$ with Lemma~\ref{lem:onestep}.
  
  For $B_0(x)$ we have 
  \begin{equation}\label{eq:GPEemb:B0}
  \vnorm{B_0(x)}{\cC\tk(x)}\le\vnorm{B_0(x)}{\cC\tz(x)}\le\eta_0\vnorm{V_x}{\cC\tz(x)}\le2\eta_0d\tz_\phi(x)
  \end{equation}
  by \ref{gpe:c0} and the condition~\ref{thc:hered} of a THC-graph.
  The same argument as for~\eqref{eq:GPEemb:cldens} gives that the density of $\cC^{(\ell')}(x)$ satisfies
  \begin{equation}\label{eq:GPEemb:sizeVx}
  \vnorm{V_x}{\cC^{(\ell')}(x)}=(1\pm\eta_0)d\tz_\phi(x)\prod_{\ell''\in[\ell']}\Big(1\pm\frac{\eps'_{\ell''}}{\delta_{\ell''}}\Big)d^{(\ell'')}_\phi(x)
= (1\pm\eta_{\ell'})\prod_{0\le\ell''\le\ell} d^{(\ell'')}_\phi(x)\,.
  \end{equation}
  Then by Lemma~\ref{lem:onestep} and \eqref{eq:GPEemb:sizeVx}, we calculate for $1\le\ell\le k$ the bound
  \begin{align}
   \vnorm{B_{\ell}(x)\setminus B_{\ell-1}(x)}{\cC^{(k)}(x)} &\le \vnorm{B_{\ell}(x)\setminus B_{\ell-1}(x)}{\cC^{(\ell-1)}(x)} 
   \\&\le k\Delta^2\eps'_{\ell}\cdot\vnorm{V_x}{\cC^{(\ell-1)}(x)}
   \\&\le 2k\Delta^2\eps'_{\ell}\prod_{0\le\ell'<\ell} d^{(\ell')}_\phi(x)\,.\label{eq:GPEemb:badx}
  \end{align}
  We next give a short calculation which shows that
  \begin{equation}\label{eq:GPEemb:goodvs}
  \vnorm{V_x\setminus B_k(x)}{\cC\tk(x)} \ge (1-\eta_k)\prod_{0\le\ell\le k} d^{(\ell)}_\phi(x)\,,
  \end{equation}
  by a careful collection of density terms and `compensating' error terms from lower levels of the stack. 
  We have
  \begin{align}
  \vnorm{V_x\setminus B_k(x)}{\cC\tk(x)}
    &\ge \vnorm{V_x}{\cC\tk(x)} - \vnorm{B_0(x)}{\cC\tz(x)} - \sum_{\ell\in[k]} \vnorm{B_{\ell}(x)\setminus B_{\ell-1}(x)}{\cC^{(k)}(x)}
  \\&\ge \left( (1-\eta_0)\prod_{\ell'\in[k]}\left(1-\frac{\eps'_{\ell'}}{\delta_{\ell'}}\right)-\frac{2\eta_0}{\prod_{\ell'\in[k]}\delta_{\ell'}}-\sum_{\ell\in[k]}\frac{2k\Delta^2\eps_\ell'}{\prod_{\ell'=\ell}^k\delta_{\ell'}}\right)\prod_{0\le\ell\le k}d\tl_\phi(x)
  \\&\ge (1-\eta_k)\prod_{0\le\ell\le k}d\tl_\phi(x)\,.
  \end{align}
  The $\delta_{\ell'}$ terms in the denominators of the second line correspond to `missing densities' lost because we can only account for failure of a regularity condition in level $\ell'$ of the stack with the regularity properties of that level. 
  We can afford to write $\delta_{\ell'}$ terms instead of $d_\phi\tlp(x)$ because we have $\delta_{\ell'}\le d\tlp_\phi(x)$ by~\ref{gpe:dens}. 
  For $B_0(x)$ the missing densities are for levels $\ell'\in[k]$ but there is a very small $\eta_0$ to compensate, and for $B_{\ell}(x)\setminus B_{\ell-1}(x)$ we have a product of missing densities from levels $\ell$ to $k$ of the stack, but a comparatively small $\eps_\ell'$ to compensate. 
  
  With \eqref{eq:GPEemb:goodvs} in hand, we finish the proof with the induction on $r$. 
  For any $v\in V_x\setminus B_k(x)$, note the applying the induction hypothesis is valid as the required lower bounds on $c^*$, $h^*$ still hold, and we have 
  \begin{align}
  \cC\tk_{x\mapsto v}(H-\dom\phi-x) 
    &\ge (1-\eta_k)^{r-1}\tfrac{c\tk_{x\mapsto v}(\emptyset)}{\prod_{0\le\ell\le k}d_{\phi\cup\{x\mapsto v\}}^{(\ell)}(\emptyset)}\prod_{0\le\ell\le k}\cD\tl_{\phi\cup\{x\mapsto v\}}\big(H-\dom\phi-x\big)
  \\&=   (1-\eta_k)^{r-1}c\tk_{x\mapsto v}(\emptyset)\prod_{0\le\ell\le k} \xi(\ell)\,,
  \end{align}
  where we have separated out the only term $c\tk_{x\mapsto v}(\emptyset)$ which depends on $v$, so that the remaining product over $\ell$ is independent of $v$. 
  By the update rule we have $c\tk_{x\mapsto v}(\emptyset)=c\tk(\emptyset)c\tk(v)$, which gives 
  \begin{align}
  \cC\tk(H-\dom\phi) 
    &=   \Ex[\Big]{\cC\tk_{x\mapsto v}\big(H-\dom\phi-\{x\}\big)}[v\in V_x]
  \\&\ge (1-\eta_k)^{r-1}c\tk(\emptyset)\vnorm{V_x\setminus B_k(x)}{\cC\tk(x)}\prod_{0\le\ell\le k}\xi(\ell)  \\&=   (1-\eta_k)^r\tfrac{c\tk(\emptyset)}{\prod_{0\le\ell\le k}d_\phi^{(\ell)}(\emptyset)}\prod_{0\le\ell\le k}\cD\tl_\phi(H-\dom\phi)\,,
  \end{align}
  where for the last line we observe that density terms involving $x$ are taken care of by $\vnorm{V_x\setminus B_k(x)}{\cC\tk(x)}$ via \eqref{eq:GPEemb:goodvs}, and the other terms are given by $\xi$ via~\eqref{eq:xi}.
  \end{proof}

The proof for Lemma~\ref{lem:GPEcount} is similar, but we must proceed by induction on the height $\ell$ of the GPE and handle bad vertices more carefully. 
For the latter consideration, we use the following consequence of the Cauchy--Schwarz inequality,
which we prove along with several related tools in Section~\ref{sec:CS}.

\begin{restatable}{lemma}{ECSdist}\label{lem:ECSdist}
Let $W$, $X$, and $Y$ be discrete random variables such that $W$ takes values in $[0,1]$, $X$ takes values in the non-negative reals, and $Y$ is real-valued. 
Suppose also that for $0\le\eps\le 1$ and $d\ge0$ we have 
\begin{align}
\Ex{XY}&=(1\pm\eps)d\cdot\Ex{X}&
&\text{and}&
\Ex{XY^2}&\le(1+\eps)d^2\cdot\Ex{X}\,. 
\end{align}
Then
\begin{align}
\Ex{WXY} &= \left(1-\eps\pm 2\sqrt{\frac{\eps\Ex{X}}{\Ex{WX}}}\,\right)d\cdot \Ex{WX}\,,
\shortintertext{and}
\Ex{WXY^2} &= \left(1-2\eps\pm7\sqrt{\eps}\frac{\Ex{X}}{\Ex{WX}}\right)d^2\cdot \Ex{WX}\,.
\end{align}
\end{restatable}

\begin{proof}[Proof of Lemma~\ref{lem:GPEcount} for $\ell\ge 1$]
Given $\ell$, we prove the height $\ell$ case of Lemma~\ref{lem:GPEcount} by induction on $r=v(H)-\abs{\dom\phi}$, assuming the $\ell'\le\ell$ cases of Lemma~\ref{lem:onestep} and $\ell'<\ell$ cases of Lemma~\ref{lem:GPEcount}.
 
  As in the previous proof, the case $r=0$ is a tautology, and the statement for $r=1$ follows directly from the definition of $\ell$-GPE. 
The same applications of properties~\ref{gpe:c0}, \ref{gpe:reg}, and~\ref{ve:worst} as for~\eqref{eq:GPEemb:cldens} and~\eqref{eq:GPEemb:sizeVx} give again
  \begin{equation}\label{eq:GPEcount:sizeVx}
  \vnorm{V_x}{\cC^{(\ell')}(x)}=(1\pm\eta_0)d\tz_\phi(x)\prod_{\ell''\in[\ell']}\Big(1\pm\frac{\eps'_{\ell''}}{\delta_{\ell''}}\Big)d^{(\ell'')}_\phi(x)
= (1\pm\eta_{\ell'})\prod_{0\le\ell''\le\ell} d^{(\ell'')}_\phi(x)\,.
  \end{equation}
When $r=1$, with $\ell'=\ell$, and multiplied by the factor $c\tl(\emptyset)$, this is the desired statement.
  
  Now given $r\ge 2$, fix $x\in V(H)\setminus\dom\phi$. We use the statement of Lemma~\ref{lem:GPEcount} for heights $\ell'<\ell$ and with the complex $H-x$, and (the induction assumption in this proof) for height $\ell$. 
  We have a partition of $V_x$ into the bad vertices $B_0(x)$, and $B_{\ell'}(x)\setminus B_{\ell'-1}(x)$ for $\ell'\in[\ell]$, and the good vertices $V_x\setminus B_\ell(x)$. 
  As in the previous proof, we separately consider density terms for $\emptyset$, $x$, and the ones of the form $\xi(\ell')$
obtained via the induction on $r$.
From~\eqref{eq:xi} recall that $\xi$ is independent of $v$.
  The desired counting statement is then
  \[
  \cC\tl(H-\dom\phi) = (1\pm r\eta_\ell)c\tl(\emptyset)\prod_{0\le\ell'\le\ell}d^{(\ell')}_\phi(x)\xi(\ell')\,.
  \]
  
  As in the previous proof, by Lemma~\ref{lem:onestep}, \eqref{eq:GPEcount:sizeVx}, and the fact that in any valid ensemble we have $\eta_{\ell'-1}<1$ for all $\ell'$, we calculate for $1\le\ell'\le\ell$ the bound
  \begin{align}
   \vnorm{B_{\ell'}(x)\setminus B_{\ell'-1}(x)}{\cC^{(\ell')}(x)} &\le \vnorm{B_{\ell'}(x)\setminus B_{\ell'-1}(x)}{\cC^{(\ell'-1)}(x)} 
   \\&\le k\Delta^2\eps'_{\ell'}\cdot\vnorm{V_x}{\cC^{(\ell'-1)}(x)}
   \\&\le 2k\Delta^2\eps'_{\ell'}\prod_{0\le\ell''<\ell'} d^{(\ell'')}_\phi(x)\,.\label{eq:GPEcount:badx}
  \end{align}
  By definition, for each $v\in V_x\setminus B_{\ell'}(x)$, the partial embedding $\phi\cup\{x\to v\}$ together with the stack of candidate graphs $\cC^{(0)}_{x\mapsto v},\dots,\cC^{(\ell')}_{x\mapsto v}$ obtained by the update rule is an $\ell'$-GPE. 
  Applying for each $1\le\ell'\le\ell$ the $\ell'$ case of Lemma~\ref{lem:GPEcount}  with the partial embedding $\phi\cup\{x\mapsto v\}$ and updated candidate graphs (where we note that $c^*$ and $h^*$ are large enough and $\eta_{\ell'}$ small enough for this to be valid), it follows that for each such choice of $v$ we have
  \begin{align}
   \cC^{(\ell')}_{x\mapsto v}(H-\dom\phi-x)&=\big(1\pm(r-1)\eta_{\ell'}\big)\tfrac{c\tlp_{x\mapsto v}(\emptyset)}{\prod_{0\le\ell''\le\ell'}d_{\phi\cup\{x\mapsto v\}}^{(\ell'')}(\emptyset)}\prod_{0\le\ell''\le\ell'}\cD^{(\ell'')}_{\phi\cup\{x\mapsto v\}}(H-\dom\phi-x)\\
   &=\big(1\pm(r-1)\eta_{\ell'}\big)c^{(\ell')}(\emptyset)c^{(\ell')}(v)\cdot\prod_{0\le\ell''\le\ell'}\xi(\ell'')\,,
  \end{align}
  where the second line follows from the update rule and definition of $\xi$.
  We carefully account for the empty set in level $\ell$ and not below.
  Then by the fact that $\cC^{(\ell')}\le\cC\tl$, for each $1\le\ell'\le\ell$ and $v\in V_x\setminus B_{\ell'}(x)$ we have
  \begin{equation}\label{eq:GPEcount:countell}
  \begin{aligned}
   \cC^{(\ell)}_{x\mapsto v}(H-\dom\phi-x)
     &\le\tfrac{c\tl(\emptyset)}{c^{(\ell')}(\emptyset)}\cdot\cC^{(\ell')}_{x\mapsto v}(H-\dom\phi-x)
   \\&=\big(1\pm(r-1)\eta_{\ell'}\big)c^{(\ell)}(\emptyset)c^{(\ell')}(v)\cdot\prod_{0\le\ell''\le\ell'}\xi(\ell'')\,.
  \end{aligned}
  \end{equation}

  Putting \eqref{eq:GPEcount:sizeVx}, \eqref{eq:GPEcount:badx}, and \eqref{eq:GPEcount:countell} together will give us the required lower bound on $\cC^{(\ell)}(F)$, but for the upper bound we still need to show that the contribution made by $v\in B_0(x)$ is small. 
  Letting $H'$ be the $k$-complex on $r':=2r-1\le c^*$ vertices obtained by taking two disjoint copies of $H$ and identifying each vertex in $\dom\phi\cup\{x\}$ with the corresponding vertex in the other copy, we have the following counts in the bottom level of the stack by \ref{gpe:c0},
  \begin{align}
   \cC^{(0)}(H-\dom\phi)
     &= (1\pm r\eta_0)\tfrac{c\tz(\emptyset)}{d_\phi\tz(\emptyset)}\cD\tz_\phi(H-\dom\phi)\,,
   \\&= (1\pm r\eta_0)c\tz(\emptyset)d\tz_\phi(x)\xi(0)\,,
   \label{eq:GPE-EXYbd}\\
   \cC^{(0)}(H'-\dom\phi)
   &=(1\pm r'\eta_0)\tfrac{c\tz(\emptyset)}{d_\phi\tz(\emptyset)}\cD\tz_\phi(H'-\dom\phi)
   \\&=(1\pm r'\eta_0)c\tz(\emptyset)d\tz_\phi(x)\xi(0)^2\,.\label{eq:GPE-EXY2bd}
  \end{align}
  From this, apply Lemma~\ref{lem:ECSdist} to the experiment of choosing a uniform random $v\in V_x$, with 
  \begin{align}
    X&:=c\tz_{x\mapsto v}(\emptyset)\,,
  & Y&:=\cC\tz_{x\mapsto v}(H-\dom\phi -x)/c\tz_{x\mapsto v}(\emptyset)\,,
  & W&:=\mathbbm{1}_{v\in B_0(x)}\,.
  \end{align}
    Property~\ref{gpe:c0} gives $\Ex{X}=(1\pm\eta_0)c\tz(\emptyset)d\tz_\phi(x)$, and statements~\eqref{eq:GPE-EXYbd} and~\eqref{eq:GPE-EXY2bd} give bounds on $\Ex{XY}$ and $\Ex{XY^2}$.
  We also have $\Ex{WX}\le\eta_0\Ex{X}$ by \ref{gpe:c0} and condition~\ref{thc:hered}.
  Hence we conclude 
  \begin{equation}
  \Ex{WXY}\le 5r'\eta_0 \cdot c\tz(\emptyset) d\tz_\phi(x)\xi(0) \le 10\eta_0 c^*\cdot c\tz(\emptyset)d\tz_\phi(x)\xi(0)\,.
  \end{equation}
  Again, taking care to deal with the empty set in level $\ell$, we deduce the upper bound bound
  \begin{equation}\label{eq:GPEcount:b0bound}
  10\eta_0c^*\cdot c\tl(\emptyset)d\tz_\phi(x)\xi(0)
  \end{equation}
  on the contribution to $\cC\tl(H-\dom\phi)$ from vertices $v\in B_0(x)$.

To complete the proof we substitute these bounds into the expression
  \begin{align*}
   \cC^{(\ell)}(H-\dom\phi) 
   &= \Ex[\Big]{\indicator{v\notin B_{\ell}(x)}\cC\tl_{x\mapsto v}\big(H-\dom\phi-\{x\}\big)} 
 \\&\qquad\pm \sum_{\ell'\in[\ell]}\Ex[\Big]{\indicator{v\in B_{\ell'}(x)\setminus B_{\ell'-1}(x)}\cC^{(\ell)}_{x\mapsto v}\big(H-\dom\phi-\{x\}\big)}
 \\&\qquad\pm \Ex[\Big]{\indicator{v\in B_0(x)}\cC\tz_{x\mapsto v}\big(H-\dom\phi-\{x\}\big)}\,.
 \end{align*}
 Using~\eqref{eq:GPEcount:sizeVx},~\eqref{eq:GPEcount:badx}, \eqref{eq:GPEcount:countell}, and~\eqref{eq:GPEcount:b0bound}, we obtain
  \begin{align*}
   \cC^{(\ell)}(H-\dom\phi) 
   &= \big(1\pm(r-1)\eta_\ell\big)c\tl(\emptyset)\vnorm{V_x\setminus B_\ell(x)}{\cC\tl(x)} \cdot \prod_{0\le\ell''\le\ell}\xi(\ell'')
 \\&\qquad\pm \sum_{\ell'\in[\ell]}\big(1+(r-1)\eta_{\ell'}\big)c\tl(\emptyset)\vnorm{B_{\ell'}(x)\setminus B_{\ell'-1}(x)}{\cC^{(\ell')}(x)}\cdot \prod_{0\le\ell''\le\ell'}\xi(\ell')
 \\&\qquad\pm 10\eta_0c^*\cdot c\tl(\emptyset)d\tz_\phi(x)\xi(0)
 \\&= \big(1\pm(r-1)\eta_\ell\big)c\tl(\emptyset)\bigg(1\pm\frac{\vnorm{B_\ell(x)}{\cC\tl(x)}}{\vnorm{V_x}{\cC\tl(x)}}\bigg)(1\pm\eta_0)\bigg(\prod_{\ell''\in[\ell]}\Big(1+\frac{\eps'_{\ell''}}{\delta_{\ell''}}\Big)\bigg)\cdot \prod_{0\le\ell''\le\ell}d^{(\ell'')}_\phi(x)\xi(\ell'')
 \\&\qquad\pm \sum_{\ell'\in[\ell]}\big(1+(r-1)\eta_{\ell'}\big)c\tl(\emptyset)\cdot2k\Delta^2\frac{\eps'_{\ell'}}{d^{(\ell')}_\phi(x)}\cdot\prod_{0\le\ell''\le\ell'}d^{(\ell'')}_\phi(x)\xi(\ell'')
 \\&\qquad\pm 10\eta_0c^*\cdot c\tl(\emptyset)d\tz_\phi(x)\xi(0)\,.
  \end{align*}
  This is almost the desired statement. By collecting terms we have
  \[
  \cC\tl(H-\dom\phi) = (1\pm r\eta_\ell)c\tl(\emptyset)\prod_{0\le\ell''\le\ell}d^{(\ell'')}_\phi(x)\xi(\ell'')
\,,
  \]
  where the relative error is given by $r\eta_\ell$, provided the following holds:
  \begin{align}
  1+r\eta_\ell &\ge \big(1+(r-1)\eta_\ell\big)\bigg(1+\frac{\vnorm{B_\ell(x)}{\cC\tl(x)}}{\vnorm{V_x}{\cC\tl(x)}}\bigg)(1+\eta_0)\prod_{\ell''\in[\ell]}\Big(1+\frac{\eps'_{\ell''}}{\delta_{\ell''}}\Big)
  \\&\qquad + \sum_{\ell'\in[\ell]}\big(1+(r-1)\eta_{\ell'}\big) 2k\Delta^2\cdot \frac{\eps'_{\ell'}}{\delta_{\ell'}\prod_{\ell'<\ell''\le\ell} \delta_{\ell''}\xi(\ell'')}
  \\&\qquad + \frac{10\eta_0c^*}{\prod_{0<\ell''\le\ell} \delta_{\ell''}\xi(\ell'')}\,.
  \end{align}
  The definition of a valid ensemble is chosen to make this inequality hold. 
  Considering the right-hand side, the first line can be made at most $1+(r-2/3)\eta_\ell$.
   Each of the two remaining terms can be made at most $\eta_\ell/3$.   Essentially the point is that where we have products of `missing' minimum densities in the denominator of error terms, there is a $\eps_{\ell'}$ or $\eta_0$ to compensate in the numerator. 
  The $\eps_{\ell'}'$ parameters are chosen to be small enough to compensate for any product of minimum densities from the same level or higher, and $\eta_0$ is small enough to compensate for any densities in levels above $0$. 

  Here we require the upper bound on $r$, since it implies the $\xi(\ell')$ terms corresponding to edges remaining after $x$ is embedded cannot be too small.
  We give the required calculations below, relying on the facts that for all $\ell'\ge 1$, we have
\begin{align}\label{eq:GPEcount:densitylbs}
\delta_{\ell'}       &\le d^{(\ell')}_\phi(x)\,,&
\delta_{\ell'}^{c^*-1} &\le \xi(\ell')\,.
\end{align}
The first bound states the contribution to the final count at level $\ell'$ from embedding $x$ is at least $\delta_{\ell'}$, which holds by assumption: $\delta_{\ell'}$ is a minimum density. 
Then with $2r-1\le c^*$ the first inequality implies the second because $\xi(\ell')$ is a product over the remaining $r-1$ vertices of their contributions. 
The next claim deals with the smaller two error terms, and a subsequent claim deals with the main term.

  \begin{claim} \ref{ve:worst} implies both
  \begin{gather}
	\sum_{\ell'\in[\ell]}\big(1+(r-1)\eta_{\ell'}\big) 2k\Delta^2\cdot \frac{\eps'_{\ell'}}{\delta_{\ell'}\prod_{\ell'<\ell''\le\ell} \delta_{\ell''}\xi(\ell'')}
 \le \frac{\eta_\ell}{3}\,,
 \shortintertext{and}
  \frac{10\eta_0c^*}{\prod_{0<\ell''\le\ell} \delta_{\ell''}\xi(\ell'')} \le \frac{10\eta_0c^*}{\prod_{0<\ell''\le\ell} \delta_{\ell''}^{c^*}} \le \frac{\eta_\ell}{3}\,.
  \end{gather}
  \end{claim}
  \begin{claimproof}
  For the first statement, since we have $(r-1)\eta_{\ell'}\le 1/2$ and \eqref{eq:GPEcount:densitylbs} it suffices to ensure that
  \[
  \eps_{\ell'}' \le \frac{\eta_\ell}{9k^2\Delta^2} \delta_{\ell'}\prod_{\ell'<\ell''\le\ell} \delta_{\ell''}^{c^*}  \]
  for each $\ell'\in[\ell]$, which holds by \ref{ve:worst}.
   In the second statement the first inequality holds by \eqref{eq:GPEcount:densitylbs}, and the second holds by \ref{ve:worst}.
  \end{claimproof} 
 
  \begin{claim} \ref{ve:worst} implies 
  \[
  \big(1+(r-1)\eta_\ell\big)\bigg(1+\frac{\vnorm{B_\ell(x)}{\cC\tl(x)}}{\vnorm{V_x}{\cC\tl(x)}}\bigg)(1+\eta_0)\prod_{\ell''\in[\ell]}\Big(1+\frac{\eps'_{\ell''}}{\delta_{\ell''}}\Big)
  \le 1+\Big(r-\frac{2}{3}\Big)\eta_\ell
  \]
  \end{claim}
  \begin{claimproof}
  First we bound $\vnorm{B_\ell(x)}{\cC\tl(x)}$. 
  By~\eqref{eq:GPEcount:sizeVx}, \eqref{eq:GPEcount:badx}, and $\vnorm{B_0(x)}{\cC\tz(x)}\le\eta_0\vnorm{V_x}{\cC\tz(x)}$, we have
  \begin{align}
  \vnorm{B_\ell(x)}{\cC\tl(x)} 
  &\le \vnorm{B_0(x)}{\cC\tz(x)}+\sum_{\ell'\in[\ell]}\vnorm{B_{\ell'}(x)\setminus B_{\ell'-1}(x)}{\cC^{(\ell')}(x)}
\\&\le 2\eta_0d\tz_\phi(x) + \sum_{\ell'\in[\ell]}2k\Delta^2\eps_{\ell'}'\prod_{0\le\ell''<\ell'}d^{(\ell'')}_\phi(x)\,.
  \end{align}
  Hence (using that $\eta_\ell<1/2$), we have
  \begin{align}
  \frac{\vnorm{B_\ell(x)}{\cC\tl(x)}}{\vnorm{V_x}{\cC\tl(x)}} 
  &\le \frac{4\eta_0}{\prod_{\ell'\in[\ell]}d^{(\ell')}(x)}+ \sum_{\ell'\in[\ell]}4k\Delta^2\frac{\eps_{\ell'}'}{\prod_{\ell'\le\ell''\le\ell}d^{(\ell'')}(x)}
\\&\le \frac{4\eta_0}{\prod_{\ell'\in[\ell]}\delta_{\ell'}}+ \sum_{\ell'\in[\ell]}4k\Delta^2\frac{\eps_{\ell'}'}{\prod_{\ell'\le\ell''\le\ell}\delta_{\ell''}}\,.
  \end{align}
  For the claim, by $r\eta_\ell<1/2$ it now suffices to show
  \[
  \Bigg(1+\frac{4\eta_0}{\prod_{\ell'\in[\ell]}\delta_{\ell'}}+ \sum_{\ell'\in[\ell]}4k\Delta^2\frac{\eps_{\ell'}'}{\prod_{\ell'\le\ell''\le\ell}\delta_{\ell''}}\Bigg)(1+\eta_0)\prod_{\ell''\in[\ell]}\Big(1+\frac{\eps'_{\ell''}}{\delta_{\ell''}}\Big) \le 1+\frac{\eta_\ell}{9} \le \frac{1+(r-2/3)\eta_\ell}{1+(r-1)\eta_\ell}\,.
  \]
  We use that $\eta_\ell \ll k,\, 1$. 
  The first bracketed term and $(1+\eta_0)$ are each at most $1+\eta_\ell/36$ by \ref{ve:worst}, and similarly we have
  \[
  1+\frac{\eps_{\ell''}'}{\delta_{\ell''}} \le  1+\frac{\eta_\ell}{72k} \le \Big(1+\frac{\eta_\ell}{36}\Big)^{1/k}\,,  \]
  which shows the product over $[\ell]$ is also at most $1+\eta_\ell/36$. It follows that $(1+\eta_\ell/36)^3\le 1+\eta_\ell/9$ as required.
  \end{claimproof}

This completes the proof of Lemma~\ref{lem:GPEcount}.
\end{proof}

\section{Homomorphism counts and the Cauchy--Schwarz inequality}\label{sec:CS}

We can always consider a $k$-complex on $t$ vertices as $t$-partite with parts of size one, and in this case we represent the sum giving $\cH(F)$ by an expectation as follows. 
A partite homomorphism $\phi:F\to\cH$ must map vertex $j$ of $F$ to a vertex $x_j\in V_j$ of $\cH$, so for $\abs{J}=t$ indexing the vertex sets, a partite homomorphism $\psi$ from $F$ to $\cH$ is equivalent to a vector of vertices $x_J\in V_J$, and we have
\[
\cH(F) = \Ex[\Big]{\prod_{e\in F}h(x_e)}[x_J\in V_J]\,,
\]
where the expectation is over the uniform distribution on vectors $x_J\in V_J$, and we write $x_e$ for the natural projection of $x_e$ onto $V_e$. 

Let $\vec{a}\in\{0,1,2\}^J$. 
We use the following notation for the count of octahedra such as $\oct{k}{\vec{a}}$ in $\cH$.  
Suppose that for $j\in J$ and $i\in[\vec a_j]$, vertices $x_j^{(i)}$ are chosen uniformly at random (with replacement) from $V_j$. 
For $e\subset J$ and $\omega\in\prod_{j\in J}[\vec a_j]$ we write $x_e^{(\omega)}$ for the vector indexed by $j\in e$ of vertices $x_j^{(\omega_j)}$.
Then we have the notation 
\[
\cH\big(F(\vec{a})\big) = \Ex[\Big]{\;\,\prod_{\substack{e\in F\\\mathclap{\omega:\omega_i\in[\vec a_i]}}}\;\,h(x_e^{(\omega)})}[x_{j}^{(i)}\in V_{j}\text{ for each $j\in J$ and $i\in[\vec a_j]$}]\,,
\]
for the expected weight of a uniformly random partite homomorphism from $\oct{k}{\vec{a}}$ to $\cH$.
With this notation in place, we turn to the main tool of the paper.

\subsection{The Cauchy--Schwarz inequality and related results}

We make extensive use of the Cauchy--Schwarz inequality in the form $\Ex{XY}^2\leq \Ex{X}\Ex{XY^2}$, where we take care to ensure that $X$ takes non-negative values throughout. 
First, we restate Lemma~\ref{lem:ECSdist} and give a proof.

\ECSdist*

\begin{proof}
For the first statement, observe that $X$, $W$, and $1-W$ are all non-negative random variables. 
Then we have 
\begin{align}
  (1+\eps)d^2\cdot\Ex{X} &\ge \Ex{XY^2} =\Ex{WXY^2}+\Ex{(1-W)XY^2}
\\&\ge \frac{\Ex{WXY}^2}{\Ex{WX}}+\frac{\Ex{(1-W)XY}^2}{\Ex{(1-W)X}}\,,\label{eq:CSdist:1}
\end{align}
where the second inequality is by two applications of Cauchy--Schwarz. 
Given fixed $\Ex{WXY}$ the right hand side is minimised when $\Ex{XY}=(1-\eps)d\cdot\Ex{X}$, so we may assume $\Ex{(1-W)XY}=(1-\eps)d\cdot\Ex{X}-\Ex{WXY}$.
  
Let $\Ex{WXY}=(1-\eps+c)d\cdot\Ex{WX}$. Then from~\eqref{eq:CSdist:1} we have
\begin{align}
  (1+\eps)\Ex{X} &\ge (1-\eps+c)^2\Ex{WX}+\frac{\big((1-\eps)\Ex{X}-(1-\eps+c)\Ex{WX}\big)^2}{\Ex{(1-W)X}}
\\&= (1-\eps+c)^2\Ex{WX}+\frac{\big((1-\eps)\Ex{(1-W)X}-c\Ex{WX}\big)^2}{\Ex{(1-W)X}}
\\&=(1-\eps)^2\Ex{X}+\frac{c^2\Ex{X}\Ex{WX}}{\Ex{(1-W)X}}\,,
\end{align}
and so
\[
3\eps -\eps^2 \ge c^2\frac{\Ex{WX}}{\Ex{(1-W)X}} \ge c^2\frac{\Ex{WX}}{\Ex{X}}\,,
\]
which is a contradiction if $c^2\ge 4\eps \Ex{X}/\Ex{WX}$, as required.

For the second statement, we have by Cauchy--Schwarz and the first part,
\begin{align}
\Ex{WXY^2} \ge \frac{\Ex{WXY}^2}{\Ex{WX}} &\ge \left(1-\eps- 2\sqrt{\frac{\eps\Ex{X}}{\Ex{WX}}}\,\right)^2d^2\cdot \Ex{WX}
\\&\ge\left(1-2\eps- 4\sqrt{\frac{\eps\Ex{X}}{\Ex{WX}}}\,\right)d^2\cdot \Ex{WX}\,,
\end{align}
and similarly
\begin{align}
\Ex{(1-W)XY^2} &\ge \frac{\Ex{(1-W)XY}^2}{\Ex{(1-W)X}} 
\\&\ge \frac{\left((1-\eps)\Ex{(1-W)X}-2\sqrt{\eps\Ex{X}\Ex{WX}}\right)^2}{\Ex{(1-W)X}}d^2
\\&\ge\left((1-2\eps)\Ex{(1-W)X}-4\sqrt{\eps\Ex{X}\Ex{WX}}\right)d^2\,,
\end{align}
so that
\begin{align*}
\Ex{WXY^2} &\le (1+\eps)d^2\cdot\Ex{X}-\left((1-2\eps)\Ex{(1-W)X}-4\sqrt{\eps\Ex{X}\Ex{WX}}\right)d^2
\\&\le\left(1-2\eps+7\sqrt{\eps}\frac{\Ex{X}}{\Ex{WX}}\right)d^2\cdot \Ex{WX}\,.\qedhere
\end{align*}
\end{proof}

\begin{corollary}\label{cor:ECSconc}
Let $X$ and $Y$ be random variables such that $X$ takes values in the non-negative reals and $Y$ is real-valued. 
Suppose also that for $0\le\eps\le 1$ and $d\ge0$ we have 
\begin{align}
\Ex{XY}&=(1\pm\eps)d\cdot\Ex{X}&
&\text{and}&
\Ex{XY^2}&\le(1+\eps)d^2\cdot\Ex{X}\,. 
\end{align}
Let $W$ be the indicator of the event that $Y=(1\pm2\eps^{1/4})d$. Then 
\[
\Ex{WX}\ge (1-4\eps^{1/4})\Ex{X}\,.
\] 
\end{corollary}
\begin{proof}
Write $\eps'=2\eps^{1/4}$ and let $Z$ indicate the event that $Y>(1+\eps')d$. 
Then using Lemma~\ref{lem:ECSdist} with weight $1-Z$ we have
\begin{align}
  (1+\eps)d\cdot\Ex{X} &\ge \Ex{XY} = \Ex{(1-Z)XY}+\Ex{ZXY}
\\&\ge\left(1-\eps-2\sqrt{\frac{\eps\Ex{X}}{\Ex{(1-Z)X}}}\,\right)d\cdot\Ex{(1-Z)X}+(1+\eps')d\cdot\Ex{ZX}
\\&\ge\left(1-\eps-2\sqrt{\eps}+\eps'\frac{\Ex{ZX}}{\Ex{X}}\right)d\cdot\Ex{X}\,,
\end{align}
which implies that $\Ex{ZX}\le 2\eps^{1/4}\Ex{X}$. 

With a similar argument we deal with the event that $Y<(1-\eps')d$, now using the letter $Z$ for this event we calculate
\begin{align}
  (1-\eps)d\cdot\Ex{X} &\le \Ex{XY} = \Ex{(1-Z)XY}+\Ex{ZXY}
\\&\le\left(1-\eps+2\sqrt{\frac{\eps\Ex{X}}{\Ex{(1-Z)X}}}\,\right)d\cdot\Ex{(1-Z)X}+(1-\eps')d\cdot\Ex{ZX}
\\&\le\left(1+2\sqrt{\eps}-\eps'\frac{\Ex{ZX}}{\Ex{X}}\right)d\cdot\Ex{X}\,,
\end{align}
which implies that $\Ex{ZX}\le 2\eps^{1/4}\Ex{X}$. 
Together, the two arguments prove that $\Ex{WX}\geq (1-4\eps^{1/4})\Ex{X}$ as required.
\end{proof}

\begin{corollary}\label{cor:higherECSconc}
Let $X$ and $Y$ be random variables such that $X$ takes values in the non-negative reals and $Y$ is real-valued. 
Suppose also that for a natural number $t\ge2$, and reals $0\le\eps<2^{2-2t}$ and $d\ge0$ we have 
\begin{align}
\Ex{XY}&=(1\pm\eps)d\cdot\Ex{X}&
&\text{and}&
\Ex{XY^{2^t}}&\le(1+\eps)d^{2^t}\cdot\Ex{X}\,. 
\end{align}
Let $W$ be the indicator of the event that $Y=(1\pm2\eps^{1/8})d$. Then 
\[
\Ex{WX}\ge (1-4\eps^{1/8})\Ex{X}\,.
\] 
\end{corollary}
\begin{proof}
Let $Z=Y^{2^{t-1}}$ and $\tilde d=d^{2^{t-1}}$. 
Then by the Cauchy--Schwarz inequality we have
\[
\Ex{XZ}^2\le \Ex{X}\Ex{XZ^2} = \Ex{X}\Ex{XY^{2^t}}\le (1+\eps)\tilde d^2 \cdot \Ex{X}^2\,.
\]
By $t-1$ further applications of the Cauchy--Schwarz inequality we also have
\[
\Ex{XZ}\ge \Ex{X}^{1-2^{t-1}}\Ex{XY}^{2^{t-1}} \ge (1-\eps)^{2^{t-1}}\tilde d\cdot\Ex{X} \ge (1-2^{t-1}\eps)\tilde d \cdot\Ex{X}\,.
\]
With $\tilde\eps=\eps^{1/2}\ge2^{t-1}\eps$ this implies
\begin{align}
\Ex{XZ} &= (1\pm\tilde\eps)\tilde d\cdot\Ex{X}\,,&
\Ex{XZ^2} &\le (1+\tilde\eps)\tilde d^2\cdot\Ex{X}\,.
\end{align}
The result now follows from Corollary~\ref{cor:ECSconc}. 
Note that $Z=(1\pm2\tilde\eps^{1/4})\tilde d$ implies the event $Y=(1\pm2\eps^{1/8})d$ which is indicated by $W$, hence by Corollary~\ref{cor:ECSconc} we obtain $\Ex{WX}\ge (1-4\eps^{1/8})\Ex{X}$. 
\end{proof}

\subsection{Lower bounds on octahedra}\label{sec:lbocts}

The common theme in the following results is an application of the Cauchy--Schwarz inequality to the expectation in a normalised homomorphism count. 

\begin{lemma}\label{lem:CSoctlow}
For every natural number $k\ge 2$, vertex set $J$ of size $k$, index $i\in J$ and vectors $\vec{a},\,\vec{b},\,\vec{c}\in\{1,2\}^J$ which satisfy $\vec{a}_j=\vec{b}_j=\vec{c}_j$ for all $j\in J\setminus\{i\}$ and $\vec{a}_i=0$, $\vec{b}_i=1$, $\vec{c}_i=2$ the following holds.
Let $\cH$ be a $k$-partite $k$-graph on vertex set $\{V_j\}_{j\in J}$. Then 
\[
\cH\big(\oct{k}{\vec{c}}\big)\ge\frac{\cH\big(\oct{k}{\vec{b}}\big)^2}{\cH\big(\oct{k-1}{\vec{a}}\big)}\,.
\]
\end{lemma}
\begin{proof}
We prove the case $J=\{0,1,\dotsc,k-1\}$ and $i=0$, writing $f=[k-1]$ for the indices on which $\vec{a}$, $\vec{b}$, and $\vec{c}$ agree. 
The other cases follow by relabelling indices. 

Observe that a copy of $\oct{k}{\vec{c}}$ simply consists of two copies of $\oct{k}{\vec{b}}$ agreeing on a copy of $\oct{k-1}{\vec{a}}$. Let $X$ be the random variable giving the weight of a uniform random copy of $\oct{k-1}{\vec{a}}$, and $Y$ be the random variable which, given a uniform random copy of $\oct{k-1}{\vec{a}}$, returns the total weight of the ways to extend it to a copy of $\oct{k}{\vec{b}}$. More concretely, we choose uniformly at random (with replacement) vertices $x_j^{(i)}\in V_j$ for each $i\in [\vec a_j]$, and let
\begin{align}
X&:=\prod_{\substack{e\subset f,\\\mathclap{\omega:\omega_i\in[\vec a_i]}}}\;g(x_e\tw)&
&\text{and}&
Y&:=\Ex[\Big]{\;\,\prod_{\substack{e\subset f,\\\mathclap{\omega:\omega_i\in[\vec a_i]}}}\;\,g(x_0,x_e\tw)}[x_0\in V_0]\,.
\end{align}
Thus we have $\Ex{X}=\cH\big(\oct{k-1}{\vec{a}}\big)$, $\Ex{XY}=\cH\big(\oct{k}{\vec{b}}\big)$, and $\Ex{XY^2}=\cH\big(\oct{k}{\vec{c}}\big)$. Since $X$ is a nonnegative random variable, the Cauchy-Schwarz inequality $\Ex{XY}^2\le\Ex{X}\Ex{XY^2}$ gives the required statement.
\end{proof}

Lemma~\ref{lem:CSoctlow} justifies the term `minimal' used in the following definition.

\begin{definition}\label{def:minimal}
Let $\cH$ be a $k$-partite $k$-graph and $\eta\ge0$. Then we say that $\cH$ is \emph{$\eta$-minimal} if, for every $i\in [k]$ and for every $\vec{a},\vec{b},\vec{c}\in\{0,1,2\}^k$ which satisfy $\vec{a}_j=\vec{b}_j=\vec{c}_j$ for all $j\in [k]\setminus\{i\}$ and $\vec{a}_i=0$, $\vec{b}_i=1$, $\vec{c}_i=2$, we have
\[
\cH\big(\oct{k}{\vec{c}}\big) \le 
(1+\eta)\frac{\cH(\oct{k}{\vec{b}})^2}{\cH(\oct{k-1}{\vec{a}})}\,.
\]
\end{definition}

Suppose $\Gam$ and $\cG$ are $k$-partite $k$-graphs, and $\cG$ agrees with $\Gam$ on edges of size $k-1$ and less. Suppose furthermore that the density of $\cG$ relative to $\Gam$ is $d$. If $\Gam$ is a complete graph, then it is well known that $\cG$ has at least $d^{2^k}$ times as many octahedra as $\Gam$. For general $\Gam$ this statement is false, but we will now show that if $\Gam$ is $\eta$-minimal it is approximately true (and generalise it).

\begin{corollary}\label{cor:relCSoctlow}
For all natural numbers $k\geq 2$ and vectors $\vec{s},\,\vec{s'}\in\{1,2\}^{k}$ with $\vec{s}\ge\vec{s'}$ pointwise and such that $\vec{s}$ has $t$ more $2$ entries than $\vec{s'}$, the following holds.  
Suppose that $\cG$ and $\Gam$ are $k$-partite $k$-graphs on the same partite vertex set, with $g(e)=\gam(e)$ for all $e$ with $|e|<k$, and suppose $\cG\big(\oct{k}{\vec{s'}}\big)=d\cdot \Gam\big(\oct{k}{\vec{s'}}\big)$.
Moreover suppose that $\Gam$ is $\eta$-minimal.
Then
\[
\cG\big(\oct{k}{\vec{s}}\big) \ge \frac{d^{2^t}}{(1+\eta)^{2^t-1}} \Gam\big(\oct{k}{\vec{s}}\big)\,.
\]
\end{corollary}
\begin{proof}
We prove the case $\vec{s}=(\vec{2}^{t+a},\vec{1}^{k-t-a})$ where $a\ge 0$ is an integer; the other cases follow by relabelling indices. Letting $\vec{s}^{(i)}:=(\vec{2}^{t+a-i},\vec{1}^{k+i-t-a})$ and $\vec{r}^{(i)}:=(\vec{2}^{t+a-i-1},0,\vec{1}^{k+i-t-a})$, we have $\vec{s}^{(0)}=\vec{s}$ and $\vec{s}^{(t)}=\vec{s'}$. By Lemma~\ref{lem:CSoctlow} and $\eta$-minimality of $\Gam$ respectively, for each $0\le i\le t-1$ we have
\begin{align}
\cG\big(\oct{k}{\vec{s}^{(i)}}\big)&\ge\frac{\cG\big(\oct{k}{\vec{s}^{(i+1)}}\big)^2}{\cG\big(\oct{k-1}{\vec{r}^{(i+1)}}\big)}&
&\text{and}&
\Gam\big(\oct{k}{\vec{s}^{(i)}}\big) &\le(1+\eta)\frac{\Gam\big(\oct{k}{\vec{s}^{(i+1)}}\big)^2}{\Gam\big(\oct{k-1}{\vec{r}^{(i+1)}}\big)}\,.
\end{align}
Note that since $\cG$ and $\Gam$ agree on edges of size less than $k$, the denominators in both fractions are equal, so for each $0\le i\le t-1$ we have
\[\frac{\cG\big(\oct{k}{\vec{s}^{(i)}}\big)}{\Gam\big(\oct{k}{\vec{s}^{(i)}}\big)}\ge \frac{1}{1+\eta}\bigg(\frac{\cG\big(\oct{k}{\vec{s}^{(i+1)}}\big)}{\Gam\big(\oct{k}{\vec{s}^{(i+1)}}\big)}\bigg)^2\,,\]
and thus
\[\frac{\cG\big(\oct{k}{\vec{s}^{(0)}}\big)}{\Gam\big(\oct{k}{\vec{s}^{(0)}}\big)}\ge \frac{1}{(1+\eta)^{2^t-1}}\bigg(\frac{\cG\big(\oct{k}{\vec{s}^{(t)}}\big)}{\Gam\big(\oct{k}{\vec{s}^{(t)}}\big)}\bigg)^{2^t}=\frac{d^{2^t}}{(1+\eta)^{2^t-1}}\,,\]
as desired.
\end{proof}

In particular, it follows that if $\cG$ is $(\eps,d)$-regular with respect to the $\eta$-minimal $\Gam$, then $\cG$ is itself $\eps'$-minimal, where $\eps'$ is small provided $\eta$ is sufficiently small and $\eps$ is small enough compared to $d$.

\begin{corollary}\label{cor:subregular}
Given $\eps',\, d>0$, then for $\eps$, $\eta$ small enough that
\[
\eps'\ge\max\Big\{1-\tfrac{(1-\eps/d)^{2^k}}{(1+\eta)^{2^k-1}},\,\big(1+\eps d^{-2^k}\big)(1+\eta)^{2^k-1}-1\Big\}
\,,\]
the following holds.
Let $\Gam$ and $\cG$ be $k$-partite $k$-graphs on the same partite vertex set, such that $\gam(e)=g(e)$ whenever $|e|< k$. Suppose that $\cG$ is $(\eps,d)$-regular with respect to $\Gam$, and that $\Gam$ is $\eta$-minimal. Then $\cG$ is $\eps'$-minimal, and for each $\vec{c}\in\{1,2\}^k$ we have $\cG\big(\oct{k}{\vec{c}}\big)=(1\pm\eps')d^r\Gam\big(\oct{k}{\vec{c}}\big)$ with $r=\prod_{i\in [k]}\vec c_i$.

Moreover we note that if the above inequality for $\eps'$ is tight, we have
\[
\eps'\le 2^{2^k}\big(\eps d^{-2^k}+\eta\big)\,.
\]
\end{corollary}

\begin{proof}
We begin with the second statement, comparing $\cG\big(\oct{k}{\vec{c}}\big)$ to $\Gam\big(\oct{k}{\vec{c}}\big)$.
Corollary~\ref{cor:relCSoctlow} with $\vec s=\vec c$ and $\vec s'=\vec1^k$ and the regularity bound on $\cG\big(\oct{k}{\vec1^k}\big)$ give the required lower bound, since for any $1\le r\le 2^k$ we have by choice of $\eps$ and $\eta$,
\[
\frac{(1-\eps/d)^{r}}{(1+\eta)^{r-1}}\ge\frac{(1-\eps/d)^{2^k}}{(1+\eta)^{2^k-1}}\ge 1-\eps'\,.
\]
To obtain the upper bound, suppose for contradiction that $\cG\big(\oct{k}{\vec{c}}\big)>(1+\eps')d^r\Gam\big(\oct{k}{\vec{c}}\big)$, 
where $r=\prod_{i\in [k]}\vec c_i$.
Then applying Corollary~\ref{cor:relCSoctlow} with $\vec s= \vec 2^k$ and $\vec s'= \vec c$, we have

\begin{align}
\cG\big(\oct{k}{\vec{2}^k}\big)
  &>   \frac{\big((1+\eps')d^r\big)^{2^k/r}}{(1+\eta)^{2^k/r-1}}\Gam\big(\oct{k}{\vec{2}^k}\big) 
   \ge \frac{(1+\eps')^{2^k/r}}{(1+\eta)^{2^k-1}}d^{2^k}\Gam\big(\oct{k}{\vec{2}^k}\big)
\\&\ge \big(d^{2^k}+\eps\big)\Gam\big(\oct{k}{\vec{2}^k}\big)\,,
\end{align}
by $1\le r\le 2^k$ and choice of $\eps$, $\eta$. 
This contradicts the $(\eps, d)$-regularity of $\cG$ with respect to $\Gam$.

The minimality argument is essentially identical. Suppose for contradiction that $\cG$ is not $\eps'$-minimal, and let $\vec{a},\,\vec{b},\,\vec{c}\in\{0,1,2\}^k$ be vectors witnessing this.
That is, these vectors agree on $[k]\setminus\{j\}$ for some $j\in[k]$ and we have $\vec a_j=0$, $\vec b_j=1$, $\vec c_j=2$, and
\begin{equation}\label{eq:subregular:fail}
 \cG\big(\oct{k}{\vec{c}}\big)>(1+\eps')\frac{\cG\big(\oct{k}{\vec{b}}\big)^2}{\cG\big(\oct{k-1}{\vec{a}}\big)}\,.
\end{equation} 
Observe that $\vec{b}$ cannot contain any zero entries, since otherwise the three octahedron counts are the same as in $\Gam$, and since $\eps'\ge \eta$ the three vectors then witness that $\Gam$ is not $\eta$-minimal. Let $t$ be the number of $2$ entries in $\vec{b}$. 
By Corollary~\ref{cor:relCSoctlow}, we have $\cG\big(\oct{k}{\vec{b}}\big)\ge\frac{d^{2^t}}{(1+\eta)^{2^t-1}}\Gam\big(\oct{k}{\vec{b}}\big)$, so since $\cG$ and $\Gam$ agree on edges of size at most $k-1$, we have
\[
\cG\big(\oct{k}{\vec{c}}\big)\gByRef{eq:subregular:fail}(1+\eps')\frac{d^{2^{t+1}}}{(1+\eta)^{2^{t+1}-2}}\frac{\Gam\big(\oct{k}{\vec{b}}\big)^2}{\Gam\big(\oct{k-1}{\vec{a}}\big)}\ge \frac{(1+\eps')d^{2^{t+1}}}{(1+\eta)^{2^{t+1}-1}}\Gam\big(\oct{k}{\vec{c}}\big)\,,
\]
where the second inequality uses the $\eta$-minimality of $\Gam$. 
Applying Corollary~\ref{cor:relCSoctlow} with $\vec s= \vec 2^k$ and $\vec s'= \vec c$, we obtain
\begin{align}
\cG\big(\oct{k}{\vec{2}^k}\big)
  &>(1+\eta)^{2^{k-t-1}-1}\Bigg(\frac{(1+\eps')d^{2^{t+1}}}{(1+\eta)^{2^{t+1}-1}}\Bigg)^{2^{k-t-1}}\Gam\big(\oct{k}{\vec{2}^k}\big)
\\&=\frac{(1+\eps')^{2^{k-t-1}}}{(1+\eta)^{2^k-1}}d^{2^k}\Gam\big(\oct{k}{\vec{2}^k}\big)\,.
\end{align}
Since $t\le k-1$, and since $(1+\eps')(1+\eta)^{1-2^k}d^{2^k}\ge d^{2^k}+\eps$, this is the desired contradiction to the $(\eps,d)$-regularity of $\cG$ with respect to $\Gam$.
\end{proof}

\subsection{Regular subgraphs of regular graphs}

In this section we show that, given an $\eta$-minimal graph $\Gam$, if we replace the $\ell$-edges $V_{[\ell]}$ by a subgraph which is relatively dense and regular with respect to $\Gam[V_1,\dots,V_\ell]$ then the result is still $\eta'$-minimal. This is a generalisation of the `Slicing Lemma' for $2$-graphs, which says that large subsets of a regular pair induce a regular pair; in other words, replacing $1$-edges with a relatively dense subgraph preserves regularity of the $2$-edges. Note that regularity is a trivial condition for $1$-graphs.

\begin{lemma}\label{lem:slicing}
Given $\eps',\, d>0$, then for $\eps$, $\eta$ small enough that
\begin{align}
\min\{\eps',1/2\}\ge\eta+\max\Big\{
  & 2^7k^3\Big(1-\tfrac{(1-\eps/d)^{2^{k-1}}}{(1+\eta)^{2^{k-1}-1}}\Big),\,
\\& 2^7k^3\left(\big(1+\eps d^{-2^{k-1}}\big)(1+\eta)^{2^{k-1}-1}-1\right),\,
\\& 2^9k^3\left(\big(1+100\sqrt{\eta}d^{-2^{k-1}}\big)(1+2\eta)-1\right)\Big\}
\,,
\end{align}
the following holds.
Let $\Gam$ be an $\eta$-minimal $k$-partite $k$-graph with parts $V_1,\dots,V_k$, and let $\cG$ be a subgraph on the same vertex set, which agrees with $\Gam$ except on $V_{[\ell]}$ for some $\ell<k$, and which has the property that $\cG[V_1,\dots,V_\ell]$ is $(\eps,d)$-regular with respect to $\Gam[V_1,\dots,V_\ell]$. 
Then $\cG$ is $\eps'$-minimal, and for each $\vec{s}\in\{0,1,2\}^k$ we have $\cG\big(\oct{k}{\vec{s}}\big)=(1\pm\eps')d^r\Gam\big(\oct{k}{\vec{s}}\big)$ with $r:=\prod_{i\in[\ell]}\vec{s}_i$.

Moreover, we note that when $\eps'<1/2$ and the above inequality for $\eps'$ is tight, we have
\[
\eps' \le 2^{2^{k-1}+18}k^3(\eps +\sqrt\eta)d^{-2^{k-1}}\,.
\]
\end{lemma}

\begin{proof}
Let $\xi$ be maximal such that $\big(\frac{1+2k\xi}{1-2k\xi}\big)^2(1+\eta)\le1+\eps'$, noting that this gives $2^{-7}k^{-3}(\eps'-\eta)\le\xi\le (4k)^{-1}\eps'$. 
The choice of $\eps$, $\eta$ ensure that when Corollary~\ref{cor:subregular} is applied (e.g.\ to $\ind{\cG}{V_1,\dotsc,V_\ell}$ and $\ind{\Gam}{V_1,\dotsc,V_\ell}$) with $k\subref{cor:subregular}=\ell$ and $\eps$, $\eta$ as in this lemma, the resulting $\eps'\subref{cor:subregular}$ is at most $\xi$, and that
\begin{equation}\label{eq:slicing:xibd}
\big(1+100\sqrt{\eta}d^{-2^{k-1}}\big)(1+2\eta)\le1+\tfrac14\xi\,.
\end{equation}

 The following claim, and choice of $\xi$, gives the desired counting in $\cG$.
 \begin{claim}\label{claim:slicing}
  Given $\vec{s}\in\{0,1,2\}^k$, let $q:=\sum_{i\in[k]}\vec s_i$ and $r:=\prod_{i\in[\ell]}s_i$. Then we have
  \begin{equation}\label{eq:slice:G}
   \cG\big(\oct{k}{\vec{s}}\big)=(1\pm q\xi)d^r\Gam\big(\oct{k}{\vec{s}}\big)\,.
   \end{equation}
 \end{claim} 
 The required counting statements in $\cG$ follow because $q\le 2k$ and $\xi\le(4k)^{-1}\eps'$.
 The desired $\eps'$-minimality follows directly from this claim and $\eta$-minimality of $\Gam$. 
 Indeed, let $\vec{a}$, $\vec{b}$ and $\vec{c}$ be vectors in $\{0,1,2\}^k$ agreeing at all indices except $j$, and with $\vec a_j=0$, $\vec b_j=1$, $\vec c_j=2$. Let $t_{\vec{a}}=\prod_{i\in[\ell]}a_i$, and define similarly $t_{\vec{b}}$ and $t_{\vec{c}}$. Note that either $j\in[\ell]$ and we have $t_{\vec{a}}=0$ and $t_{\vec{c}}=2t_{\vec{b}}$, or $j\not\in[\ell]$ and all three are equal. Since $\sum_{i\in[k]}\vec{c}_i\le 2k$, we have by the claim,
 \begin{align*}
  \cG\big(\oct{k}{\vec{c}}\big)
    &\le(1+2k\xi)d^{t_{\vec{c}}}\Gam\big(\oct{k}{\vec{c}}\big)
  \\&\le(1+2k\xi)(1+\eta)d^{t_{\vec{c}}}\frac{\Gam\big(\oct{k}{\vec{b}}\big)^2}{\Gam\big(\oct{k}{\vec{a}}\big)}
  \\&\le(1+2k\xi)(1+\eta)d^{t_{\vec{c}}}\frac{(1-2k\xi)^{-2}d^{-2t_{\vec{b}}}\cG\big(\oct{k}{\vec{b}}\big)^2}{(1+2k\xi)^{-1}d^{-t_{\vec{a}}}\cG\big(\oct{k}{\vec{a}}\big)}=\Bigg(\frac{1+2k\xi}{1-2k\xi}\Bigg)^2(1+\eta)\frac{\cG\big(\oct{k}{\vec{b}}\big)^2}{\cG\big(\oct{k}{\vec{a}}\big)}\,,
 \end{align*}
 as desired. It remains only to prove the claim, which we now do by induction on the number of zeroes in $\vec{s}$ outside $[\ell]$.

\begin{claimproof}[Proof of Claim~\ref{claim:slicing}]
The base case is that all entries of $\vec{s}$ outside $[\ell]$ are equal to zero. Note that if $\vec{s}=\vec{0}^k$ then the claim is trivial, so we assume this is not the case, and hence $q=\sum_{i\in[k]}\vec{s}_i\ge 1$. 
Write $\vec s'\in\{0,1,2\}^\ell$ for the first $\ell$ entries of $\vec s$, which for the base case are the only entries which may be non-zero, giving $\cG\big(\oct{k}{\vec{s}}\big)=\cG[V_1,\dots,V_\ell]\big(\oct{\ell}{\vec{s}'}\big)$.
 Since $\cG[V_1,\dots,V_\ell]$ is $(\eps,d)$-regular with respect to $\Gam[V_1,\dots,V_\ell]$, which is $\eta$-minimal, by Corollary~\ref{cor:subregular} and choice of $\eps$, $\eta$, the claim statement follows.
 
For the induction step, suppose that $j\notin[\ell]$ is such that $s_j\neq 0$. For $i=0,1,2$, let $\vec s^{(i)}$ be the vector equal to $\vec s$ at all entries except the $j$th, and with $\vec s^{(i)}_j=i$. By induction, the claim statement holds for $\vec s^{(0)}$. 
Again, write $\vec s'\in\{0,1,2\}^\ell$ for the first $\ell$ entries of $\vec s$.
We define random variables $W$, $X$, $Y$ as follows. The random experiment we perform is to choose, for each $i\in[\ell]$, uniformly at random (with replacement) $\vec s\tz_i$ vertices in $V_i$. 
We let $X$ be the weight of the copy of $\oct{k-1}{\vec s\tz}$ in $\Gam$ on these vertices, $WX$ be the weight of the copy of $\oct{k-1}{\vec s\tz}$ in $\cG$ on these vertices, and $XY$ be the expected weight, over a uniformly random choice of vertex in $V_j$, of the copy of $\oct{k}{\vec{s}^{(1)}}$ in $\Gam$. Note that since $\cG$ is a subgraph of $\Gam$, we always have $0\le W\le 1$. 
More formally (and dealing with the trivial exceptional case $X=0$), let $x_i^{(m)}\in V_i$ be chosen independently, uniformly at random for each $i\in[k]\setminus\{j\}$ and $m\in[\vec s\tz_i]$. 
Write $\Omega = \prod_{i\in[k]\setminus\{j\}}[\vec s_i]$ and $\Omega' = \prod_{i\in[k]}[\vec s^{(1)}_i]$, and define
 \begin{align*}
  X&= \prod_{e\subset[k]\setminus\{j\}}\prod_{\omega\in\Omega}\gam\big(x_e^{(\vec{\omega})}\big)  \,,\\
  Y&=\Ex[\Big]{\prod_{e\subset[k],\,j\in e}\prod_{\omega\in\Omega'}\gam\big(x_e^{(\vec{\omega})}\big)}[x_j^{(1)}\in V_j]\,,\quad\text{and}\\
  W&=\begin{cases} \tfrac{1}{X}\prod_{e\subset[k]\setminus\{j\}}\prod_{\omega\in\Omega} g\big(x_e^{(\vec{\omega})}\big)&\text{ if }X>0 \\ 1 & \text{ if }X=0  \end{cases}\,.
 \end{align*}
The key feature of these definitions is that $\Ex{XY^i}=\Gam\big(\oct{k}{\vec{s}^{(i)}}\big)$ for each $i=0,1,2$, and similarly $\Ex{WXY^i}=\cG\big(\oct{k}{\vec{s}^{(i)}}\big)$.
If $\Ex{X}=0$ then trivially the claim holds, since $\Gam\big(\oct{k}{\vec{s}}\big)=0$. So we may assume $\Ex{X}>0$, and let $d'$ be such that $\Ex{XY}=d'\Ex{X}$. 
By Lemma~\ref{lem:CSoctlow} and the $\eta$-minimality of $\Gam$, we have $\Ex{XY^2}= (1\pm\eta)\frac{\Ex{XY}^2}{\Ex{X}}=(1\pm\eta)(d')^2\Ex{X}$. We are thus in a position to apply Lemma~\ref{lem:ECSdist}, with $\eps\subref{lem:ECSdist}=\eta$. We obtain
\begin{align*}
 \Ex{WXY}&=\Big(1-\eta\pm2\sqrt{\tfrac{\eta\Ex{X}}{\Ex{WX}}}\Big)d'\cdot\Ex{WX}\quad\text{and}\\
 \Ex{WXY^2}&=\Big(1-2\eta\pm7\sqrt{\eta}\frac{\Ex{X}}{\Ex{WX}}\Big)(d')^2\Ex{WX}\,.
\end{align*}
Recall that from the induction hypothesis with $q=\sum_{i\in[k]}\vec s\tz_i$ and $r=\prod_{i\in[\ell]}\vec s\tz_i$ we have 
\[
\Ex{WX}=\cG\big(\oct{k-1}{\vec s\tz}\big)=(1\pm q\xi)d^r\Gam\big(\oct{k}{\vec{s}^{(0)}}\big)=(1\pm q\xi)d^r\Ex{X}\,.
\]
This gives
\begin{align*}
 \cG\big(\oct{k}{\vec{s}^{(1)}}\big)&=\Big(1-\eta\pm2\sqrt{\eta(1+ 2q\xi)d^{-r}}\Big)\cdot(1\pm q\xi)d^r\Gam\big(\oct{k}{\vec{s}^{(1)}}\big)\quad\text{and}\\
 \cG\big(\oct{k}{\vec{s}^{(2)}}\big)&=\Big(1-2\eta\pm7\sqrt{\eta}(1+2q\xi)d^{-r}\Big)\cdot(1\pm q\xi)(1\pm2\eta)d^r\Gam\big(\oct{k}{\vec{s}^{(2)}}\big)\,,
\end{align*} 
where we use the $\eta$-minimality of $\Gam$ in obtaining the second statement. 
By choice of $\xi$ and~\eqref{eq:slicing:xibd}, this proves the claim for $\vec{s}^{(i)}$ with $i=1,2$, and in particular for $\vec{s}$, as desired. 
\end{claimproof}
The proof of Claim~\ref{claim:slicing} completes the proof of Lemma~\ref{lem:slicing}.
\end{proof}

\section{Inheritance of regularity}\label{sec:inherit}

Our goal in this section is to prove Lemma~\ref{lem:k-inherit}. 
Note that in proving the counting and embedding lemmas (see Section~\ref{sec:count}) for $k$-graphs, we must apply Lemma~\ref{lem:k-inherit} for $k\subref{lem:k-inherit}$-graphs where $k\subref{lem:k-inherit}$ takes values up to $k$, which means we must mention $(k+1)$-partite $(k+1)$-graphs in the proof below. 
Our definitions mean that in Section~\ref{sec:count}, whenever we are applying Lemma~\ref{lem:k-inherit} to a $(k+1)$-partite $(k+1)$-graph, the graphs are trivial and equal to $1$ on edges of size $k+1$. 
This feature is visible in the graph case: when $k=2$ the inheritance lemmas of~\cite{CFZextremal,ABSSregularity} involve a $3$-partite graph, and to deduce similar results from our inheritance lemma one must form a $3$-graph from this $3$-partite graph by giving edges of size $3$ weight $1$. 

We with a brief outline of the method for proving Lemma~\ref{lem:k-inherit}. 
First, let $\cH$ be the $k$-graph on $V_0,\dotsc, V_k$ with edge weights
\[
h(e):=
\begin{cases}
\gam(e) &\quad e\not\in V_{[k]} \\
g(e)    &\quad e\in V_{[k]}\,.
\end{cases}
\]
By Lemma~\ref{lem:slicing} and~\ref{inh:regJ}, $\cG$ is regular with respect to $\cH$, and by~\ref{inh:regf}, $\ind{\cH}{V_1,\dotsc,V_k}$ is regular with respect to $\ind{\Gam}{V_1,\dotsc,V_k}$. 
This, together with Corollary~\ref{cor:subregular}, in particular allows us to estimate $\cG\big(\oct{k+1}{\vec{1}^{k+1}}\big)$ and $\cG\big(\oct{k+1}{1,\vec{2}^{k}}\big)$ accurately.

These two quantities are by definition equal to the averages, over $v\in V_0$, of $\cG_v\big(\oct{k}{\vec{1}^{k}}\big)$ and $\cG_v\big(\oct{k}{\vec{2}^{k}}\big)$ respectively. Using~\ref{inh:count}, we conclude that on average the relative density of $\cG_v$ with respect to $\Gam_v$ is about $dd'$, and the number of octahedra it contains is about $(dd')^{2^k}$ times the number of octahedra in $\Gam_v$.

However, we can also give a lower bound on the average number of octahedra in $\cG_v$ using its density relative to $\Gam_v$ and Corollary~\ref{cor:relCSoctlow}, whenever $\Gam_v$ satisfies the counting conditions of that lemma. The assumption~\ref{inh:count} implies that these counting conditions are typically satisfied, and the few atypical vertices do not much affect the argument. Using the defect Cauchy--Schwarz inequality and the fact that we know the average density of $\cG_v$ relative to $\Gam_v$, we conclude that the only way this lower bound does not contradict the previous estimate is if typically $\cG_v$ has density about $dd'$ relative to $\Gam_v$ and number of octahedra about $(dd')^{2^k}$ times the number in $\Gam_v$. In other words, $\cG_v$ is typically $(\eps',dd')$-regular with respect to $\Gam_v$, as desired.

\begin{proof}[Proof of Lemma~\ref{lem:k-inherit}]

We use the letter $v$ for a vertex in $V_0$ to draw attention to the special role of the set $V_0$, but use $x_j$ for a vertex in $V_j$ when $j\in [k]$. 
As in the proof of Lemmas~\ref{lem:GPEcount} and~\ref{lem:GPEemb}, we use the correspondence between copies of $\oct{k+1}{1,\vec{a}}$ in $\cG$ or $\Gam$, and the average of the counts of $\oct{k}{\vec{a}}$ in the graphs $\cG_v$ or $\Gam_v$ over $v\in V_0$. 
More precisely, we have for any $\vec{a}\in\{0,1,2\}^k$,
\begin{align}
\cG\big(\oct{k+1}{1,\vec{a}}\big) &= \Ex[\big]{\cG_{v}\big(\oct{k}{\vec{a}}\big)}[v\in V_0]\,,\label{eq:GF1acor}\\
\Gam\big(\oct{k+1}{1,\vec{a}}\big) &= \Ex[\big]{\Gam_{v}\big(\oct{k}{\vec{a}}\big)}[v\in V_0]\,.\label{eq:GamF1acor}
\end{align} 
When $\Gam_v$ is well-behaved (in a way we make precise below) we are able to count carefully in $\cG_v$ but when $\Gam_v$ is not well-behaved, we can bound weights in $\cG_v$ from above by those in $\Gam_v$. 

Though in general it is difficult to control the weight of the empty set, in this proof we only embed a single vertex into $V_0$, hence there is not much to control. 
Instead of the usual sprinkling of weights involving the empty set, for this proof we can assume without loss of generality that $\gam(\emptyset)=g(\emptyset)=p(\emptyset)=1$ and avoid most of these factors. 
We will have similar correcting factors when counting in $\cG_v$ and $\Gam_v$, however.

The first step of the proof is to use the counting conditions in $\Gam$ to establish the existence of $U\subset V_0$ such that for each $v\in U$, $\Gam_v$ is well-behaved. 
We also give additional properties of $\Gam$ and $U$ that are useful later. 
Property~\ref{inh:count} specifies a kind of pseudorandomness for $\Gam$, and the natural definition of a well-behaved vertex $v\in V_0$ is that its link $\Gam_v$ is similarly pseudorandom, so our definition of $U$ will involve control of the counts of $\oct{k}{\vec{a}}$ in links. 
As ever, we must deal carefully with the weight of the empty set in these links, but for edges of size greater than one, we will see that~\ref{inh:count} implies concentration of the edge weights by the Cauchy--Schwarz inequality. 
We state the definition in terms of $\cP$ rather than $\cP_v$ for more convenient use later.

Write $\eta'=2^{3/2}\eta^{1/4}$ (so we have $\eta'<1/2$), and let $U\subset V_0$ be those vertices $v\in V_0$ such that for any $\vec{a}\in\{0,1,2\}^k\setminus\{\vec0^k\}$ we have 
\begin{equation}\label{eq:Udef}
\Gam_v\big(\oct{k}{\vec{a}}\big)=(1\pm \eta')\frac{\gam(v)}{p(0)}\cP\big(\oct{k+1}{1,\vec{a}}\big)\,.
\end{equation}
The counting assumptions~\ref{inh:count} are a form of pseudorandomness which suggests that $U$ will be a large subset of $V_0$, which we prove in the necessary weighted setting below. 

\begin{claim}\label{clm:U}
\hfill
\begin{enumerate}
\item\label{itm:Gammin}
$\Gam$ is $16\eta$-minimal.
\item\label{itm:Gamvmin}
For $v\in U$, $\Gam_v$ is $16\eta'$-minimal.
\item\label{itm:Ularge}
The contribution to $\Gam\big(\oct{k+1}{\vec1^{k+1}}\big)$ from homomorphisms which use a vertex in $V_0\setminus U$ is at most $3^{k+3}\eta'\cP\big(\oct{k+1}{\vec1^{k+1}}\big)$. 
\end{enumerate}
\end{claim}

\begin{claimproof}
To see \ref{itm:Gammin}, we use \ref{inh:count}. 
Let $j\in \{0\}\cup[k]$, and vectors $\vec{a},\vec{b},\vec{c}\in\{0,1,2\}^{k+1}$ be equal on $\{0\}\cup[k]\setminus\{j\}$ and satisfy $\vec{a}_j=0$, $\vec{b}_j=1$, $\vec{c}_j=2$. Then by \ref{inh:count} we have
\begin{align}
\Gam\big(\oct{k+1}{\vec{a}}\big)\Gam\big(\oct{k+1}{\vec{c}}\big) 
  &\le (1+\eta)^2\cdot\cP\big(\oct{k+1}{\vec{a}}\big)\cP\big(\oct{k+1}{\vec{c}}\big)
\\&= (1+\eta)^2\cdot\cP\big(\oct{k+1}{\vec{b}}\big)^2
\\&\le \frac{(1+\eta)^2}{(1-\eta)^2}\cdot\Gam\big(\oct{k+1}{\vec{b}}\big)^2\,,
\end{align}
which shows $\Gam$ is minimal with parameter $(1+\eta)^2(1-\eta)^{-2}-1\le 16\eta$. 

The proof of \ref{itm:Gamvmin} is similar but we use the definition of $U$. 
Let $j\in [k]$, and $\vec{a},\vec{b},\vec{c}\in\{0,1,2\}^k$ be equal on $[k]\setminus\{j\}$ and satisfy $\vec{a}_j=0$, $\vec{b}_j=1$, $\vec{c}_j=2$.
If $\vec a=\vec 0^k$ then the required bound is trivial, otherwise by the fact that $v\in U$ we have
\begin{align}
\Gam_v\big(\oct{k}{\vec{a}}\big)\Gam_v\big(\oct{k}{\vec{c}}\big) 
  &\le (1+\eta')^2\cdot\frac{\gam(v)^2}{p(0)^2}\cP\big(\oct{k+1}{1,\vec{a}}\big)\cP\big(\oct{k+1}{1,\vec{c}}\big)
\\&=   (1+\eta')^2\cdot\frac{\gam(v)^2}{p(0)^2}\cP\big(\oct{k+1}{1,\vec{b}}\big)^2
\\&\le \frac{(1+\eta')^2}{(1-\eta')^2}\cdot\Gam_v\big(\oct{k}{\vec{b}}\big)^2\,,
\end{align}
which shows that when $v\in U$, $\Gam_v$ is minimal with parameter $(1+\eta')^2(1-\eta')^{-2}-1\le 16\eta'$.

Part \ref{itm:Ularge} resembles a step in the proof of Lemma~\ref{lem:GPEcount} involving $\cC\tz$. 
We first establish a lower bound on $\vnorm{U}{\Gam}$. 
Fix $\vec{a}\in\{0,1,2\}^k$ and recall that $+2\oct{k+1}{0,\vec{a}}$ is the $(k+1)$-complex obtained by taking two vertex-disjoint copies of $\oct{k+1}{1,\vec{a}}$ and identifying their first vertices. 
Consider the experiment where $v\in V_0$ is chosen uniformly at random, and let 
\begin{align}
   X&:=\gam(v)\,,
&  Y&:=\frac{\Gam_v(\oct{k}{\vec{a}})}{\gam(v)}\,.
\end{align}
By \eqref{eq:GamF1acor} and \ref{inh:count} we have 
\begin{align}
\Ex{XY}
    &=\Ex[\big]{\Gam_v\big(\oct{k}{\vec{a}}\big)}
    =\Gam\big(\oct{k+1}{1,\vec{a}}\big)
    =(1\pm\eta)\cP\big(\oct{k+1}{1,\vec{a}}\big)
  \\&=(1\pm\eta)p(0)\cdot\frac{\cP\big(\oct{k+1}{1,\vec{a}}\big)}{p(0)}\,,
\\\Ex{XY^2}
    &=\Ex*{\frac{\Gam_v\big(\oct{k}{\vec{a}}\big)^2}{\gam(v)}}
    =\Gam\big({+2}\oct{k+1}{0,\vec{a}}\big)
    =(1\pm\eta)\cP\big({+2}\oct{k+1}{0,\vec{a}}\big)
  \\&=(1\pm\eta)p(0)\cdot\left(\frac{\cP\big(\oct{k+1}{1,\vec{a}}\big)}{p(0)}\right)^2\,.
\end{align}
Noting that $\Ex{X}=(1\pm\eta)p(0)$ by \ref{inh:count}, we can apply Lemma~\ref{lem:ECSdist} and Corollary~\ref{cor:ECSconc} in the arguments below with an appropriate $\vec a$, $\eps\subref{lem:ECSdist}=\eps\subref{cor:ECSconc}=4\eta$, and 
\[
d\subref{lem:ECSdist}=d\subref{cor:ECSconc}=\frac{\cP\big(\oct{k+1}{1,\vec{a}}\big)}{p(0)}\,.
\]
To bound $\vnorm{U}{\Gam}$, for $\vec{a}\in\{0,1,2\}^k\setminus\{\vec0^k\}$, let $U_{\vec{a}}\subset V_0$ be those vertices $v$ which satisfy
\[
\Gam_v\big(\oct{k}{\vec{a}}\big) = (1\pm \eta')\frac{\gam(v)}{p(0)}\cP\big(\oct{k+1}{1,\vec{a}}\big)\,,
\]
so that $U$ is the intersection of the $3^k-1$ different $U_{\vec{a}}$. 
By Corollary~\ref{cor:ECSconc} we have $\vnorm{U_\vec{a}}{\Gam}\ge (1-2\eta')\vnorm{V_0}{\Gam}$, and hence $\vnorm{U}{\Gam}\ge(1-3^{k+1}\eta')\vnorm{V_0}{\Gam}$, so that
\begin{equation}\label{eq:Ularge}
\vnorm{V_0\setminus U}{\Gam}\le 3^{k+1}\eta'\vnorm{V_0}{\Gam}\,.
\end{equation}
 
The contribution to $\Gam\big(\oct{k+1}{\vec1^{k+1}}\big)$ from homomorphisms that use a vertex in $V_0\setminus U$ can be written as
\begin{equation}
\Ex[\big]{\Gam_v\big(\oct{k}{\vec1^k}\big)\indicator{v\in V_0\setminus U}}\,,
\end{equation}
which is a weighting of $\Ex[\big]{\Gam_v\big(\oct{k}{\vec1^k}\big)}$ by $W:=\indicator{v\in V_0\setminus U}$. 
We apply Lemma~\ref{lem:ECSdist} with this weight $W$ and $X$, $Y$ as above with $\vec a=\vec1^k$ to obtain 
\begin{align}
  \Ex[\big]{\Gam_v\big(\oct{k}{\vec1^k}\big)\indicator{v\in V_0\setminus U}} &\leq \left(1-4\eta+4\sqrt{\frac{\eta\vnorm{V_0}{\Gam}}{\vnorm{V_0\setminus U}{\Gam}}}\right)  \frac{\cP\big(\oct{k+1}{\vec1^{k+1}}\big)}{p(0)}\cdot \vnorm{V_0\setminus U}{\Gam}
\\&\le (1+\eta)\big(3^{k+1}\eta'+4\sqrt{3^{k+1}\eta\eta'}\big)\cP\big(\oct{k+1}{\vec1^{k+1}}\big)\,,
\end{align}
and note that the coefficient of $\cP\big(\oct{k+1}{\vec1^{k+1}}\big)$ here is at most $3^{k+3}\eta'$.
\end{claimproof}

With the set $U$ understood, we proceed by counting $\oct{k+1}{1,\vec2^k}$ in $\cG$ two different ways.  
Firstly, we estimate counts of $\oct{k+1}{\vec1^{k+1}}$ and $\oct{k+1}{1,\vec2^k}$ in $\cG$ with Corollary~\ref{cor:subregular} and Lemma~\ref{lem:slicing}.
We give crude values of the constants that work in the argument, but make no effort to optimise them. 

Let $\cH$ have layer $k+1$ given by $\cG$, and lower layers given by $\Gam$.
Then by assumption~\ref{inh:regJ} $\cH$ is $(\eps,d')$-regular with respect to $\Gam$, and we obtain $\cG$ from $\cH$ by replacing weights on $V_{[k]}$ with those from $\cG$.
By Claim~\ref{clm:U}\ref{itm:Gammin}, and Corollary~\ref{cor:subregular} for $(k+1)$-graphs, $\cH$ is $\eps_m$-minimal where
\[
\eps_m = 2^{2^{k+1}}\Big(\eps(d')^{-2^{k+1}}+\eta\Big) > \max\Big\{1-\frac{(1-\eps/d')^{2^{k+1}}}{(1+\eta)^{2^{k+1}-1}},\,\big(1+\eps (d')^{-2^{k+1}}\big)(1+\eta)^{2^{k+1}-1}-1\Big\}\,.
\]
We can now apply Lemma~\ref{lem:slicing} for $(k+1)$-graphs to $\cG$ and $\cH$ to obtain the required counts in $\cG$. 
With $\eps\subref{lem:slicing}=\eps$, $\eta\subref{lem:slicing}=\eps_m$ as above, and $d\subref{lem:slicing}=d$, we obtain that for
\[
\eps_m' = 2^{2^k+22}k^3\big(\eps^{1/2}(d')^{-2^k}+\eta^{1/2}\big)d^{2^{-k}}\,,
\]
the $(k+1)$-graph $\cG$ is $\eps_m'$-minimal, and the remaining assertions of Corollary~\ref{cor:subregular} and Lemma~\ref{lem:slicing} give
\begin{align}
  \cG\big(\oct{k+1}{\vec1^{k+1}}\big) &= (1\pm \eps_m') d \cdot \cH\big(\oct{k+1}{\vec1^{k+1}}\big) = (1\pm \eps_m)(1\pm \eps_m')dd'\Gam\big(\oct{k+1}{\vec1^{k+1}}\big)
\\& = (1\pm\eps_m'') dd' \cdot \cP\big(\oct{k+1}{\vec1^{k+1}}\big)\,,\label{eq:G11k}
\\\cG(\oct{k+1}{1,\vec2^k}) &\le (1+\eps_m') d^{2^k}\cH\big(\oct{k+1}{1,\vec2^k}\big) \le (1+\eps_m)(1+\eps_m')(dd')^{2^k}\Gam\big(\oct{k+1}{1,\vec2^k}\big)
\\&\le(1+\eps_m'')(dd')^{2^k}\cdot\cP\big(\oct{k+1}{1,\vec2^k}\big)\,,\label{eq:G12k}
\end{align}
where
\[
\eps_m'' = 2^{2^k+25}k^3\big(\eps^{1/2}(d')^{-2^{k+1}}+\eta^{1/2}\big)d^{2^{-k}}\,.
\]

The second method for counting $\oct{k+1}{1,\vec2^k}$ involves counting $\oct{k}{\vec2^k}$ in the links of vertices $v\in V_0$. 
We have $\cG_v\le\Gam_v$ and since we do not try to control $\cG_v$ directly when $v\notin U$, we define
\begin{align}
d_v&=
\begin{cases}
\frac{\cG_v(\oct{k}{\vec1^k})}{\Gam_v(\oct{k}{\vec1^k})} & \text{if } v\in U\,,\\
0 & \text{otherwise}\,.
\end{cases}
\end{align}

\begin{claim}\label{clm:Edvbounds}
Writing
\[
\zeta:=\max\left\{\eps_m''+\eta'+\frac{3^{k+3}\eta'}{dd'},\,\eps_m''+2\eta'+2\eta'\eps_m'',\,\frac{(1+16\eta')^{2^k-1}}{1-16\eta'} (1+\eps_m'')-1\right\}
\]
we have
\begin{align}
\Ex{\gam(v)d_v}                 &=  (1\pm\zeta)dd'\cdot p(0)             \,,
&&\text{and}&
\Ex[\big]{\gam(v)d_v^{2^{k-1}}} &\le (1  +\zeta)(dd')^{2^{k-1}}\cdot p(0)\,.
\end{align}
Moreover, we note that a crude calculation gives
\[
\zeta \le 2^{2^{k+1}+50}k^3\big(\eps^{1/2} + \eta^{1/4}\big)(dd')^{-2^{k+1}}\,.
\]
\end{claim}
\begin{claimproof}
First we bound $\Ex{\gam(v)d_v}$. 
By \eqref{eq:GF1acor} we have 
\begin{align}
\cG(\oct{k+1}{\vec1^{k+1}}) &= \Ex[\big]{\cG_v\big(\oct{k}{\vec1^k}\big)}\\
&\le \Ex[\big]{d_v\Gam_v\big(\oct{k}{\vec1^k}\big)} + \Ex[\big]{\Gam_v\big(\oct{k}{\vec1^k}\big)\indicator{v\in V_0\setminus U}}\,.
\end{align} 
By the definition \eqref{eq:Udef} of $U$, for the first expectation we have an upper bound on $\Gam_v(\oct{k}{\vec1^k})$ which depends only on $\gam(v)$, and by Claim~\ref{clm:U}\ref{itm:Ularge} we have a bound on the final expectation which represents copies of $\oct{k+1}{\vec1^{k+1}}$ using a vertex in $V_0\setminus U$.
We combine these facts with \eqref{eq:G11k} to obtain a lower bound on $\Ex{\gam_0(v)d_v}$.
That is, 
\begin{align}
(1-\eps_m'') dd' \cdot \cP\big(\oct{k+1}{\vec1^{k+1}}\big)
  &\le \cG\big(\oct{k+1}{\vec1^{k+1}}\big) 
\\&\le (1+\eta')\frac{\cP\big(\oct{k+1}{\vec1^{k+1}}\big)}{p(0)}\Ex{\gam(v)d_v} + 3^{k+3}\eta'\cP\big(\oct{k+1}{\vec1^{k+1}}\big)\,,
\end{align}
which yields the lower bound
\begin{equation}\label{eq:Edv_lb}
  \Ex{\gam(v)d_v} \ge \left(1-\eps_m''-\eta'-\frac{3^{k+3}\eta'}{dd'}\right) dd'\cdot p(0)\,.
\end{equation}
For a corresponding upper bound we have
\begin{align}
(1+\eps_m'') dd' \cdot \cP\big(\oct{k+1}{\vec1^{k+1}}\big) 
  &\ge \cG\big(\oct{k+1}{\vec1^{k+1}}\big)  
   \ge \Ex[\big]{d_v\Gam_v\big(\oct{k}{\vec1^k}\big)}
\\&\ge (1-\eta')\frac{\cP\big(\oct{k+1}{\vec1^{k+1}}\big)}{p(0)}\Ex{\gam(v)d_v}\,,
\end{align} 
by the definition of $U$ and \eqref{eq:G11k}. 
We conclude
\begin{equation}
\Ex{\gam(v) d_v} \le \left(1+\eps_m''+2\eta'+2\eta'\eps_m''\right) dd'\cdot p(0)\,.
\end{equation}

For the second statement, Claim~\ref{clm:U}\ref{itm:Gamvmin} means that when $v\in U$ we can apply Corollary~\ref{cor:relCSoctlow} to $\cG_v\le\Gam_v$ and obtain a lower bound on $\cG_v\big(\oct{k}{\vec2^k}\big)$,
\begin{align}
\cG_v\big(\oct{k}{\vec2^k}\big)
  &\ge \frac{d_v^{2^k}}{(1+16\eta')^{2^k-1}}\Gam_v\big(\oct{k}{\vec2^k}\big)\label{eq:GvOkb-lb}
\\&\geq \frac{1-\eta'}{(1+16\eta')^{2^k-1}}\frac{\cP\big(\oct{k+1}{1,\vec2^k}\big)}{p(0)} \cdot \gam(v)d_v^{2^k}\,.
\intertext{Then by \eqref{eq:GF1acor} again,}
p(0) \cG\big(\oct{k+1}{1,\vec2^k}\big)
&\ge \frac{1-\eta'}{(1+16\eta')^{2^k-1}}\cP\big(\oct{k+1}{1,\vec2^k}\big)\Ex[\big]{\gam(v)d_v^{2^k}}\,,
\end{align}
which together with \eqref{eq:G12k} implies the required upper bound 
\[
\Ex[\big]{\gam(v)d_v^{2^k}} \le \frac{(1+16\eta')^{2^k-1}}{1-16\eta'} (1+\eps_m'')(dd')^{2^k}\cdot p(0)\,.\qedhere
\]
\end{claimproof}

Claim~\ref{clm:Edvbounds} means that we have concentration of $d_v$ by Corollary~\ref{cor:higherECSconc} with $X=\gam(v)$ and $Y=d_v$. 
Writing $\Uconc\subset U$ for the vertices $v$ with $d_v= (1\pm2\zeta^{1/8})dd'$, we have
\begin{equation}\label{eq:Uconc}
\vnorm{\Uconc}{\Gam}\ge\big(1-4\zeta^{1/8}\big)p(0)\,.
\end{equation}

It remains to show that for almost all of the weight in $\Uconc$, $\cG_v$ is regular in the sense that the weight of $\oct{k}{\vec2^k}$ is close to minimal. 
Let $\Ureg\subset U$ be the vertices $v\in U$ with 
\[
\cG_v\big(\oct{k}{\vec2^k}\big)\le (d_v^{2^k}+\eps')\Gam_v\big(\oct{k}{\vec2^k}\big)\,.
\]
For all vertices $v\in U$ we have the lower bound~\eqref{eq:GvOkb-lb} on $\cG_v\big(\oct{k}{\vec2^k}\big)$, hence we are supposing that for $v\in U\setminus\Ureg$ we have an additive improvement on~\eqref{eq:GvOkb-lb} of at least $\eps'\Gam_v\big(\oct{k}{\vec2^k}\big)$. 
Then we have
\begin{align}
\cG\big(\oct{k+1}{1,\vec2^k}\big)
  &= \Ex[\big]{\cG_v\big(\oct{k}{\vec2^k}\big)}
\\&\ge \frac{1}{(1+16\eta')^{ 2^k-1}}\Ex[\Big]{\big(d_v^{ 2^k}+\eps'\indicator{v\in U\setminus\Ureg}\big)\Gam_v\big(\oct{k}{\vec2^k}\big)}
\\&\ge \frac{1-\eta'}{(1+16\eta')^{2^k-1}}\left(\Ex[\big]{\gam(v)d_v^{2^k}}+\eps'\vnorm{U\setminus\Ureg}{\Gam}\right)\cP\big(\oct{k+1}{1,\vec2^k}\big)/p(0)
\\&\ge \frac{1-\eta'}{(1+16\eta')^{2^k-1}}\left(\frac{\Ex{\gam_0(v)d_v}^{2^k}}{\vnorm{U}{\Gam}^{2^k-1}}+\eps'\vnorm{U\setminus\Ureg}{\Gam}\right)\cP\big(\oct{k+1}{1,\vec2^k}\big)/p(0)
\\&\ge \frac{(1-\eta')(1-\zeta)^{2^k}}{\big((1+16\eta')(1+\eta)\big)^{2^k-1}}\left( (dd')^{2^k}+\frac{\eps'\vnorm{U\setminus\Ureg}{\Gam}}{p(0)}\right)\cP\big(\oct{k+1}{1,\vec2^k}\big)\,,
\end{align}
where the fourth line is by the Cauchy--Schwarz inequality, and the fifth is by Claim~\ref{clm:Edvbounds}, and the fact that $\vnorm{U}{\Gam}\le\vnorm{V_0}{\Gam}\le (1+\eta)p(0)$.
With \eqref{eq:G12k} we have
\begin{align}\label{eq:Ureg}
\vnorm{U\setminus\Ureg}{\Gam}
  &\le \frac{1}{\eps'}\left(\frac{\big((1+16\eta')(1+\eta)\big)^{2^k-1}(1+\eps_m'')}{(1-\eta')(1-\zeta)^{2^k}}-1\right) (dd')^{2^k} p(0)
\\&\le \frac{1}{\eps'}2^{2^{k+2}+54}k^3\big(\eps^{1/2}+\eta^{1/4}\big)(dd')^{-2^k}p(0)
\\&\le 2^{2^{k+2}+54}k^3\big(\eps^{1/4}+\eta^{1/8}\big)(dd')^{-2^k}p(0)\,,
\end{align}
where for the last line we use that $\eps'\ge \max\{\eps^{1/4},\,\eta^{1/8}\}$.

Now, for Lemma~\ref{lem:k-inherit} we may take $V_0'=\Uconc\cap\Ureg$, since then for $v\in V_0'$ the link $\cG_v$ inherits both the desired relative density and regularity from $\cG$. 
Moreover, by~\eqref{eq:Uconc} and~\eqref{eq:Ureg} we have
\begin{align}
\vnorm{V_0}{\Gam} 
  &\ge \vnorm{\Uconc}{\Gam} - \vnorm{U\setminus\Ureg}{\Gam}
\\&\ge \left(1-4\zeta^{1/8}-2^{2^{k+2}+54}k^3\big(\eps^{1/4}+\eta^{1/8}\big)(dd')^{-2^k}\right)p(0)
\\&\ge \left(1-2^{2^{k+6}}k^3\big(\eps^{1/16}+\eta^{1/32}\big)(dd')^{-2^k}\right)\vnorm{V_0}{\Gam}\,,
\end{align}
where we again use \ref{inh:count} for the last line. 
To complete the proof, observe that we choose $\eps$, $\eta$ in terms of $\eps'$, $d$, $d'$, and $k$ to satisfy
\[
\min\{\eps', 2^{-k}\} \ge 2^{2^{k+6}}k^3\big(\eps^{1/16}+\eta^{1/32}\big)(dd')^{-2^k}\,.\qedhere
\]
\end{proof}

\section{Concluding remarks}

A feature of this paper is that we intend for the methods, in particular the definitions of THC and GPE, and the proofs that one can obtain these properties,  to be of more interest than the results we obtain in this paper with them.
Indeed, some of our theorems are similar to results that can be deduced from combinations of existing hypergraph regularity and counting methods in the literature.  
In this section we discuss methods presented in this paper and from the literature from the perspective of some possible applications, and highlight useful features of a number of different ways to prove hypergraph counting results.

\subsection{Counting in sparse hypergraphs}

Our THC and GPE methods (Theorems~\ref{thm:GaTHC}, \ref{thm:counting}, and~\ref{thm:embedding}), allow techniques which resemble those for working with the regularity setup of Rödl--Skokan~\cite{RSreg} to be used in sparse graphs. 
If one is interested (say) in a relative hypergraph removal lemma, then one could use Theorem~\ref{thm:counting} to give a somewhat direct proof which has the flavour of generalising the methods of the dense case (embedding vertex-by-vertex) to the sparse case. 
Alternatively, Conlon, Fox, and Zhao~\cite{CFZrelative} show that one can transfer the dense case result (used as a black box) to the sparse case. 
Their approach is certainly easier to write down, and in many cases requires slightly weaker pseudorandomness of the majorising hypergraph (our approach would be much better for removing large graphs with small maximum degree, theirs would be better for cliques), but we claim that having a direct proof may be useful in other applications of the method. 
A direct proof may be more amenable to modification for use in applications that require more precise control of embedded vertices, such as the blow-up lemma. 
A simple example of this is already present here: in Theorem~\ref{thm:embedding} we can exploit the maximum degree of the graph to be embedded to allow embedding large graphs in a way that does not follow easily from the methods of~\cite{CFZrelative}.

\subsection{Embedding in sparse hypergraphs}

As mentioned briefly in Section~\ref{sec:main}, we give two self-contained ways to prove embedding results in this paper. 
Given a THC-graph $\Gam$, the GPE methods yielding Theorem~\ref{thm:embedding} show that one can embed a bounded-degree but large (size growing with $v(\cG)$) complex $F$ into a regular subgraph $\cG\subset\Gam$.
But one can also use Theorem~\ref{thm:counting} (or similar results from the literature) to count \emph{fixed size} subgraphs in $\cG$, and apply Theorem~\ref{thm:GaTHC} to obtain that $\cG$ is itself a THC-graph.
An analogous embedding result for $F$ follows. 
For applications, the latter approach may save quite some effort: to describe $\cG$ as a THC-graph requires only a density graph and the constants $c^*$ and $\eta$, and one can ignore the majorising hypergraph. 
In contrast, a GPE has an ensemble of parameters and much more structure to consider. 
One does lose information going to a THC-graph, however, and one may wish to keep the GPE formalism to allow appealing to special properties of $\Gam$ such as when $\Gam$ is a random graph. A particular example to bear in mind is the result of~\cite{ABETbandwidth}, in which it is shown (among other things) that a triangle factor with $n-cp^{-2}$ vertices has with high probability local resilience $\tfrac13-o(1)$ in $\Gam=G(n,p)$ if $p$ is not too small. This number of vertices in the triangle factor is optimal, and it is possible to prove such a result only because one has access to the graph $\Gam$ (using a graph version of the GPE formalism, as made explicit in~\cite{ABHKPblow}). If one attempted to prove such a result using the THC formalism, without access to $\Gam$, then the best one could hope for would be to prove local resilience for an $\big(n-o(n)\big)$-vertex triangle factor.

We conclude by sketching a new result, showing that sparse pseudorandom hypergraphs have the Ramsey property for large bounded degree hypergraphs, which one can prove using the THC formalism.

\begin{theorem}\label{thm:THCexample}
 Given $\Delta,k,r\ge 2$ there exist $C,\eps>0$ such that the following holds for all sufficiently large $n$. Suppose that $\Gam$ is an $n$-vertex $k$-graph with $n$ vertices, such that for any $k$-graph $F$ with at most $C$ vertices we have $\Gam(F)=(1\pm\eps)p^{e(F)}$. Then however the $k$-edges of $\Gam$ are $r$-coloured, there is a colour $c$ with the following property. For each $k$-graph $H$ with $\Delta(H)\le\Delta$ and $v(H)\le\eps p^{\Delta} n$, there is a copy of $H$ in $\Gam$ all of whose edges have colour $r$.
\end{theorem}

Note that one would expect that one can actually allow $H$ to have up to $\eps n$ vertices. We expect it would not be very hard to prove this, but we prefer to give a clean illustration of how one can do embedding of moderately large graphs.

\begin{proof}[Sketch proof]
We choose $\Delta+2\ll C'\ll C$ and $0<\eps\ll\eta\ll\eta'$.

 Given a $k$-graph $\Gam$, and an $r$-colouring of its edges, we begin by applying the sparse hypergraph regularity lemma, Lemma~\ref{lem:RSchreglem} (with input $\eps_k$ much smaller than $\eta$ and much larger than $\eps$), to the $k$-graphs $G_1,\dots,G_r$, where $G_i$ consists of the colour-$i$ edges of $\Gam$. By a straightforward counting argument, we find a collection of $\ell=R_r^{(k)}(k\Delta)$ clusters $V_1,\dots,V_\ell$ in the resulting family of partitions, and a collection of $2$-cells between all pairs, $3$-cells between all triples, and so on, with the following properties. First, each $i$-cell is regular for $2\le i\le k-1$, and each $i$-cell is supported on the chosen $(i-1)$-cells for $3\le i\le k-1$. Second, for each colour $1\le i\le r$, the graph $G_i$ is regular with respect to each polyad on the chosen cells.
 
 We now assign a colour to each $k$-set in $[t]$ by choosing one of the densest colours in the corresponding $k$-polyad. By definition of $t$, we can choose a colour $1\le c\le r$ and a subset $V'_1,\dots,V'_{k\Delta}$ of clusters such that in each $k$-polyad the graph $G_c$ is regular and has density at least $1/r$. Define a complex $\cG$ on vertex set $V'_1\cup\dots V'_{k\Delta}$ by taking all the edges of the chosen cells on this vertex set, together with the edges of $G_c$ they support.
 
 Given a $k$-graph $H$ with maximum degree $\Delta$ and at most $\eps p^{\Delta} n$ vertices, suppose $V(H)=[n]$ and let $\cH$ be the complex obtained from $H$ by down-closure. Note that $\cH^{(2)}$ has maximum degree at most $(k-1)\Delta$, and hence there is a partition of $V(\cH)$ into $k\Delta$ parts such that no edge of $\cH^{(2)}$ lies in any one part. We fix such a partition, and assign vertices of $H$ to the $k\Delta$ clusters of $G$ according to the partition. Let $\cD$ be the corresponding relative density graph, with $\cD(e)$ being the relative density of the $|e|$-cell on clusters $e$ (if $|e|<k$) or of $G$ relative to the $k$-polyad on clusters $e$ (if $|e|=k$).
 
By Theorem~\ref{thm:GaTHC}, we see that $\Gam$ is a $(\eta,C')$-THC graph, and so is any graph obtained from $\Gam$ by the standard construction. So applying Theorem~\ref{thm:counting}, we obtain that counts of graphs on up to $(\Delta+2)^2$ vertices in $G$ are as one would expect for the density graph $\cD$ of $G$, where $\cD'$ is obtained from $\cD$ by keeping the weights of all edges the same, except for the $k$-edges whose weights are multiplied by $p$.

By Theorem~\ref{thm:GaTHC} again, we see that $G$, and any graph obtained from it by the standard construction, is an $(\eta',\Delta+2)$-THC graph. We apply the standard construction to $G$ to obtain a $v(H)$-partite graph $G_0$, with corresponding density graph $\mathcal{R}_0$ obtained by applying the standard construction to $\cD'$.
 
We now choose in order $1\le i\le n$ an image $v_i$ for the vertex $i$ of $V(H)$ in $X_i$
We do this as follows. First, we look at the vertices of $X_i$ in $G_{i-1}$. These vertices have weight either zero or one in $G_{i-1}$, and the total weight is (because $G_{i-1}$ is an $(\eta',\Delta+2)$-THC graph) equal to $(1\pm\eta')r_{i-1}(i)|X_i|$. Of the vertices with weight one, at most $\eta' r_{i-1}(i)|X_i|$ vertices $v$ are such that the link graph $\big(G_{i-1}\big)_v$ fails to be an $(\eta',c^*)$-THC graph with density graph $\big(\mathcal{R}_{i-1}\big)_i$. We choose a vertex $v_i$ which is not among these failing vertices, and which corresponds to a vertex of $G$ not previously used. We set $G_i:=\big(G_{i-1}\big)_{v_i}$ and $\mathcal{R}_i:=\big(\mathcal{R}_{i-1}\big)_i$.
  
  To see that this is always possible, it is enough to check that $v(H)<(1-2\eta')r_{i-1}(i)|X_i|$. This is true by choice of $\eps$ and because the product defining $r_{i-1}(i)$ contains at most $\Delta$ terms coming from the $k$-level of $\mathcal{R}_0$ (because $\Delta(H)\le\Delta$).
\end{proof}

\renewcommand*{\bibfont}{\small}
\printbibliography

\end{document}